\documentclass[11pt, a4paper]{amsart}

\usepackage{latexsym,amsmath,amssymb,amsthm,mathrsfs,color}
\usepackage{float}

\usepackage[bbgreekl]{mathbbol}
\usepackage{bbm}
\usepackage{tikz-cd}
\usepackage[all]{xy}
\usepackage{stmaryrd}
\usepackage[pagebackref,colorlinks=true, citecolor=brown, filecolor=black, linkcolor=blue, urlcolor=black]{hyperref}
\usepackage[toc,page]{appendix}

\usepackage[a4paper]{geometry}

\UseRawInputEncoding

\usepackage[shortalphabetic]{amsrefs}
 \makeatletter
\def\append@label@year@{%
    \safe@set\@tempcnta\bib@year
    \edef\bib@citeyear{\the\@tempcnta}%
    \ifnum\bib@citeyear>9
      \append@to@stem{%
          \ifx\bib@year\@empty
          \else
            \@xp\year@short \bib@citeyear \@nil
          \fi
      }%
    \fi
}
\makeatother

\def\upddots{\mathinner{\mkern 1mu\raise 1pt \hbox{.}\mkern 2mu
\mkern 2mu \raise 4pt\hbox{.}\mkern 1mu \raise 7pt\vbox {\kern 7
pt\hbox{.}}} }

\numberwithin{equation}{section}

\begin{document}
\setlength{\unitlength}{2.5cm}

\newtheorem{thm}{Theorem}[section]
\newtheorem{cor}[thm]{Corollary}
\newtheorem{lemma}[thm]{Lemma}
\newtheorem{prop}[thm]{Proposition}
\newtheorem {conj}[thm]{Conjecture}
\newtheorem {assu}[thm]{Assumption}
\newtheorem {ques/conj}[thm]{Question/Conjecture}
\newtheorem {ques}[thm]{Question}
\newtheorem{defn}[thm]{Definition}
\newtheorem{remark}[thm]{Remark}
\newtheorem{prpt}[thm]{Property}
\newtheorem{eg}[thm]{Example}
\newtheorem{fact}[thm]{Fact}
\newtheorem{algo}[thm]{Algorithm}
\newtheorem{conv}[thm]{Convention}
\newtheorem{problems}[thm]{Problems}
\newtheorem{appl}[thm]{Application}
\newtheorem{assumption}[thm]{Working Hypothesis}

\newcommand{\N}{\mathbbm{N}}
\newcommand{\Z}{\mathbbm{Z}}
\newcommand{\Q}{\mathbbm{Q}}
\newcommand{\R}{\mathbbm{R}}
\newcommand{\BC}{\mathbbm{C}}

\newcommand{\CP}{\mathbb{CP}}
\newcommand{\Cat}{\mathcal{C}}
\renewcommand{\P}{\mathbb{P}}
\renewcommand{\H}{\mathbb{H}}
\newcommand{\D}{\mathbb{D}}
\newcommand{\A}{{\mathbb A}}

\newcommand{\ord}{{\mathrm{ord}}}

\newcommand{\p}{\mathfrak{p}}
\renewcommand{\a}{\mathfrak{a}}
\newcommand{\m}{\mathfrak{m}}

\newcommand{\Acal}{\mathcal{A}}
\newcommand{\Tcal}{\mathcal{T}}
\newcommand{\Ecal}{{\mathcal E}}
\newcommand{\Fcal}{{\mathcal F}}
\newcommand{\Homcal}{{\mathcal {H}om}}
\newcommand{\Gcal}{{\mathcal G}}
\newcommand{\Hcal}{{\mathcal H}}
\newcommand{\Ocal}{{\mathcal O}}

\newcommand{\Tau}{{\mathcal T}}
\newcommand{\B}{{\mathcal B}}
\newcommand{\OO}{{\mathcal O}}

\newcommand{\inv}{^{-1}}
\newcommand{\sm}{\setminus}
\newcommand{\ol}{\overline}

\newcommand{\Speh}{\mathrm{Speh}}

\newcommand{\WFP}{\mathfrak{p}^{m}}

\newcommand{\SL}{\mathrm{SL}}
\newcommand{\GL}{\mathrm{GL}}
\newcommand{\SO}{\mathrm{SO}}
\newcommand{\GSpin}{\mathrm{GSpin}}
\newcommand{\rO}{\mathrm{O}}
\newcommand{\Sp}{\mathrm{Sp}}
\newcommand{\St}{\mathrm{St}}
\newcommand{\RU}{\mathrm{U}}
\newcommand{\Cent}{\mathrm{Cent}}
\newcommand{\Wh}{{\rm Wh}}
\newcommand{\Hom}{{\rm Hom}}

\newcommand{\JL}{\mathrm{JL}}
\newcommand{\LJ}{\mathrm{LJ}}

\newcommand{\Fr}{\mathrm{Fr}}
\newcommand{\RG}{\mathrm{G}}

\newcommand{\Ind}{\mathrm{Ind}}
\newcommand{\LLC}{\mathrm{LLC}}

\renewcommand{\sp}{\text{sp}}

\newcommand{\WFN}{\overline{\mathfrak{n}}^{m}}

\newcommand{\PRep}{\underline{\mathrm{PRep}}}
\newcommand{\Rep}{\underline{\mathrm{Rep}}}
\newcommand{\Seg}{\underline{\mathrm{Segments}}}

\newcommand{\half}[1]{\frac{#1}{2}}
\newcommand{\ul}[1]{\underline{\vphantom{p}}_{#1}}

\renewcommand{\implies}{\Rightarrow}
\newcommand{\implied}{\Leftarrow}

\newcommand{\dirlim}{\varinjlim}
\newcommand{\invlim}{\varprojlim}

\newcommand{\comment}[1]{}
\newcommand{\EE}{\mathcal{E}}
\newcommand{\FF}{\mathcal{F}}
\newcommand{\CO}{\mathcal{O}}

\newcommand{\msc}[1]{\mathscr{#1}}
\newcommand{\mfr}[1]{\mathfrak{#1}}
\newcommand{\mca}[1]{\mathcal{#1}}
\newcommand{\mbf}[1]{{\bf #1}}
\newcommand{\mbm}[1]{\mathbbm{#1}}
\newcommand{\Irr}{ {\rm Irr} }
\newcommand{\Irrg}{ {\rm Irr}_{\rm gen} }
\newcommand{\diag}{{\rm diag}}
\newcommand{\uchi}{ \underline{\chi} }
\newcommand{\Tr}{ {\rm Tr} }
\newcommand{\der}\de
\newcommand{\Stab}{{\rm Stab}}
\newcommand{\Ker}{{\rm Ker}}
\newcommand{\bfp}{\mathbf{p}}
\newcommand{\bfq}{\mathbf{q}}
\newcommand{\KP}{{\rm KP}}
\newcommand{\Sav}{{\rm Sav}}
\newcommand{\de}{{\rm der}}
\newcommand{\tnu}{{\tilde{\nu}}}
\newcommand{\lest}{\leqslant}
\newcommand{\gest}{\geqslant}
\newcommand{\tu}{\widetilde}
\newcommand{\tchi}{\tilde{\chi}}
\newcommand{\tomega}{\tilde{\omega}}
\newcommand{\BDI}{{\rm Inv}_{\rm BD}}
\newcommand{\AZ}{{\rm AZ}}
\newcommand{\Spr}{{\rm Spr}}
\newcommand{\MC}{{\rm MC}}
\newcommand{\into}{\hookrightarrow}
\newcommand{\onto}{\twoheadrightarrow}

\newcommand{\cu}[1]{\textsc{\underline{#1}}}
\newcommand{\set}[1]{\left\{#1\right\}}
\newcommand{\wt}[1]{\overline{#1}}
\newcommand{\wtsf}[1]{\wt{\sf #1}}
\newcommand{\anga}[1]{{\left\langle #1 \right\rangle}}
\newcommand{\angb}[2]{{\left\langle #1, #2 \right\rangle}}
\newcommand{\wm}[1]{\wt{\mbf{#1}}}
\newcommand{\elt}[1]{\pmb{\big[} #1\pmb{\big]} }
\newcommand{\ceil}[1]{\left\lceil #1 \right\rceil}
\newcommand{\floor}[1]{\left\lfloor #1 \right\rfloor}
\newcommand{\val}[1]{\left| #1 \right|}
\newcommand{\exc}{ {\rm exc} }
\newcommand{\Sat}{{\rm Sat}}

\newcommand{\AP}[1]{\mathrm{AP}_{#1}}
\newcommand{\sub}[1]{{\vphantom{\mfr{p}}}_{#1}}

\let\oldbullet\bullet
\renewcommand{\bullet}{{\vcenter{\hbox{\tiny$\oldbullet$}}}}

\title[Covering Barbasch--Vogan and wavefront sets of genuine representations]{Covering Barbasch--Vogan duality and wavefront sets of genuine representations}

\author[F. Gao]{Fan Gao
}
\address{School of Mathematical Sciences, 
Zhejiang University, 
866 Yuhangtang Road, 
Hangzhou, China 310058}
\email{gaofan@zju.edu.cn}

\author[B. Liu]{Baiying Liu}
\address{Department of Mathematics,
Purdue University,
West Lafayette, IN, 47907, USA}
\email{liu2053@purdue.edu}

\author[C.-H. Lo]{Chi-Heng Lo}
\address{Department of Mathematics,
National University of Singapore,
119076, Singapore}
\email{ch\_lo@nus.edu.sg}

\author[F. Shahidi]{Freydoon Shahidi}
\address{Department of Mathematics,
Purdue University,
West Lafayette, IN, 47907, USA}
\email{freydoon.shahidi.1@purdue.edu}


\date{\today}



\begin{abstract}
In this paper, we start by defining a covering Barbasch--Vogan duality and prove some of its properties. Then, for genuine representations of $p$-adic covering groups, we formulate an upper bound conjecture for their wavefront sets using this covering Barbasch--Vogan duality and reduce it to anti-discrete representations. The formulation generalizes that of 
Ciubotaru--Kim and 
Hazeltine--Liu--Lo--Shahidi for linear algebraic groups. As evidence, we prove this upper bound conjecture for Kazhdan--Patterson coverings of general linear groups.
\end{abstract}

\maketitle

\tableofcontents

\section{Introduction} \label{S:intro}
Let $F$ be a non-Archimedean local field of characteristic 0. Denote  by $F^{\rm alg}$ the algebraic closure of $F$. Let $\mca{G}$ be  a connected linear reductive group scheme over $\Z$ such that its base change $\mbf{G}$ to $F$ is  a split reductive group over $F$. Consider the group $G:=\mbf{G}(F)$ or its finite degree central covers
$$\begin{tikzcd}
\mu_n \ar[r, hook] & \wt{G}^{(n)} \ar[r, two heads] & G,
\end{tikzcd}$$
where we assume that $F$ contains the full group $\mu_n$ of $n$-th roots of unity. 
In this paper, we focus exclusively on covering groups that arise from the Brylinski--Deligne framework \cite{BD}, and we also write $\wt{G}:=\wt{G}^{(n)}$ whenever $n$ is understood. If $n=1$, then the covering groups are just the linear algebraic groups, which form a very special family among all the covers.

A genuine representation of $\wt{G}$ is one such that $\mu_n$ acts via a fixed embedding $\mu_n \into \BC^\times$. Denote by $\Irrg(\wt{G})$ the set of isomorphism classes of irreducible genuine representations of $\wt{G}$.

Let $\mca{N}(G)$ denote the partially-ordered set of nilpotent orbits in $\mfr{g}={\rm Lie}(G)$ under the conjugation action of $G$. Here the partial order is given by the closure ordering in the usual topology of $\mfr{g}$ induced from that of $F$. One also has the set $\mca{N}(\mbf{G})$ consisting of geometric (stable) nilpotent orbits endowed with a partial order by the Zariski-topology. If we denote by $\mca{G}_\BC$ the complex group arising from the base change of $\mca{G}$ to $\BC$, then it is well-known (see \cite{CMBO21}*{Lemma 2.1.1}) that there is a canonical bijection
\begin{equation} \label{E:idN}
\mca{N}(\mbf{G}) \longrightarrow \mca{N}(\mca{G}_\BC).
\end{equation}
Later, we will use such an identification.

Every $(\pi, V_\pi) \in \Irrg(\wt{G})$ defines a character distribution $\chi_\pi$ in a neighborhood of $0$ in $\mfr{g}$. Moreover, there exists a compact open subset $S_\pi$ of $0$ such that for every smooth function $f$ with compact support in $S_\pi$, one has (see \cites{How1, HC99, Li3})
\begin{equation} \label{E:char}
\chi_\pi(f) = \sum_{\mca{O} \in \mca{N}{(G)}}{c_{\mca{O}, \psi}(\pi) }\cdot \int \hat{f} \ \mu_\mca{O}.
\end{equation}
 Here $\mu_\mca{O}$ is a certain Haar measure on $\mca{O}$ properly normalized, and $\hat{f}$ is the Fourier transform of $f$ with respect to the Cartan--Killing form on $\mfr{g}$ and a non-trivial character 
$$\psi: F\to \BC^\times;$$
one has $c_\mca{O}(\pi):=c_{\mca{O}, \psi}(\pi) \in \BC$. 
We have used  implicitly in \eqref{E:char} an exponent map ${\rm exp}: L \to \wt{G}$ defined for a sufficiently small open set $L \subset \mfr{g}$ containing $0$; it is used to ``pull-back" the character distribution of $\pi$ defined on $\wt{G}$ to be on (a small neighborhood of 0 in) $\mfr{g}$, see \cite{Li3}*{\S 4.3}. 

Denote
$$\mca{N}(\pi) = \set{\mca{O} \in \mca{N}(G): \ c_\mca{O}(\pi) \ne 0},$$
and let
$${\rm WF}(\pi) \subseteq \mca{N}(\pi)$$
be the subset consisting of all maximal elements in $\mca{N}(\pi)$. We also write
$${\rm WF}^{\rm geo}(\pi)\subseteq \mca{N}(\pi)\otimes F^{\rm alg} \subset \mca{N}(\mbf{G}),$$
consisting of maximal elements. Remark that ${\rm WF}^{\rm geo}$ is not necessarily a singleton, as observed in \cite{Tsa24}.

It is important to determine ${\rm WF}(\pi)$ and ${\rm WF}^{\rm geo}(\pi)$ since they gives rise to the Gelfand--Kirillov dimension of $\pi$, which {measures} in an asymptotic way a certain size of $\pi$.

Moreover, one can define global {analogues} $\mca{N}(\Pi), {\rm WF}(\Pi)$ of the above nilpotent sets associated with genuine automorphic representations $\Pi$, which {have} many deep applications, for example in the theory of descent and the Gan--Gross--Prasad conjectures and others, see \cites{GJR04, GJR05, GRS11, JLXZ16, FG18, JLX20, JiZh20, LX23} 
and references therein.
It should be mentioned that the study of ${\rm WF}(\Pi)$ for central covers has applications to problems concerning linear algebraic groups as well. A prototype is the theta correspondence relating the representations of {$\SO_{2m+1}$} and the double cover of $\Sp_{2r}$, which relies crucially on the fact that ${\rm WF}(\omega_\psi)$ of the Weil representation $\omega_\psi$ consists of the minimal orbits of the ambient symplectic group. Particularly interesting is the usage of theta representations with small wavefront sets to understand L-functions for linear algebraic groups, see for example \cites{BG1, BGH96, Tak2}.

Back to our local setting over $F$. It is however a difficult problem to determine the set ${\rm WF}(\pi)$ or ${\rm WF}^{\rm geo}(\pi)$. In view of the local Langlands correspondence, it is thus intriguing to ask whether and how the wavefront sets can be determined from the enhanced Langlands parameter of $\pi$, or as advocated in Ciubotaru--Mason-Brown--Okada \cite{CMBO21}, how the nilpotent component of the L-parameter of $\AZ(\pi)$ can be {controlled by} ${\rm WF}(\pi)$. Here, $\AZ$ is the Aubert--Zelevinsky involution.

For linear $G$ (i.e., for $n=1$), there is already an abundance of literature on partial answers to this question. For example, for $\GL_r(F)$ when $F$ is $p$-adic, the geometric wavefront set is computed in \cite{MW87} in terms of the Zelevinsky classification. For automorphic representations of $\GL_r$,
 Ginzburg \cite{Gin0} and D. Jiang \cite{Jia14} gave several speculations for the set ${\rm WF}(\pi)$ in terms of the M{\oe}glin--Waldspurger classification of the spectrum of $\GL_r$ (see \cite{LX21} for some recent progress). In this case, the Arthur parametrization of certain unitary representations $\pi$ is expected to encode information on ${\rm WF}(\pi)$. Indeed, several conjectures have been formulated in \cite{Jia14} regarding ${\rm WF}(\pi)$ using local Arthur parameters of $\pi$. 

Moreover, in the more recent work \cites{Jia14, JiLi16, JL25, JLZ, CJLZ24, HLLS24} of D. Jiang, B. Liu, L. Zhang, D. Liu, C. Chen, A. Hazeltine, C.-H. Lo, and F. Shahidi; and also the work \cites{Oka21, CMBO21, CMBO24, CMBO25, CK24} of D. Ciubotaru, L. Mason-Brown, E. Okada, and J. Kim,  authors of both groups have conducted an extensive study on the wavefront set problem in the $p$-adic case, albeit using different approaches. In particular, Jiang--Liu--Zhang defined in \cite{JLZ} a notion of arithmetic wavefront set associated with tempered $\pi$ by a descent construction directly using the enhanced L-parameter $(\phi_\pi, \eta_\pi)$ of $\pi$, and gave precise (and partially conjectural) information on ${\rm WF}(\pi)$. On the other hand, the approach of Ciubotaru et al relies more on using the Hecke algebras and understanding the $K$-types of $\pi \in \Irr(G)$. In their work, the family of unipotent representations of $G$ has been studied extensively for the wavefront sets problem. 

Coarsely speaking, the results obtained in \cites{CMBO21, CMBO24, CMBO25} show that for unipotent representations with real infinitesimal characters $\pi \in \Irr(G)$, one has ${\rm WF}^{\rm geo}(\AZ(\pi)) = \{d_{BV,G}(\mca{O}(\phi_\pi))\}$, which can be viewed as $p$-adic analogue of the results of Barbasch--Vogan in the archimedean setting \cite{BV85}. Here, let $G^\vee$ be the complex Langlands dual group of $\mbf{G}$; then
$$d_{BV,G}: \mca{N}(G^\vee) \to \mca{N}(\mbf{G})$$
 is the Barbasch--Vogan duality (sometimes also coined as Lusztig--Spaltenstein--Barbasch--Vogan duality), where we have identified $\mca{N}(\mbf{G})$ with its complex counterpart in view of \eqref{E:idN}. 
 For general $\pi\in \Irr(G)$, from the independent work of Hazeltine--Liu--Lo--Shahidi \cite{HLLS24} and Ciubotaru--Kim \cite{CK24}, it is expected that
\begin{equation} \label{E1}
 {\rm WF}^{\rm geo}(\AZ(\pi)) \lest d_{BV, G}(\mca{O}(\phi_\pi))
\end{equation}
always holds, and equality is achieved at the packet level of $\phi_\pi$.

For covering group $\wt{G}^{(n)}$, a better understood family of genuine representations concerns the theta representations $\Theta(\wt{G}^{(n)})$ of $\wt{G}^{(n)}$. In a recent work of Karasiewicz--Okada--Wang \cite{KOW}, the geometric wavefront ${\rm WF}^{\rm geo}(\Theta(\wt{G}^{(n)}))$ for \emph{unramified} $\Theta(\wt{G}^{(n)})$  has been explicitly determined, assuming that $p$ is sufficiently large. We note that prior to \cite{GaTs} and \cite{KOW}, the wavefront set problem of theta representations (especially for $\wt{\SO}_k^{(n)}, \wt{\Sp}_{2r}^{(n)}$) has been studied extensively in the work of Bump--Friedberg--Ginzburg \cites{BFrG2, BFrG}, Friedberg--Ginzburg \cites{FG15, FG18, FG20},  Kaplan  \cite{Kap17-1} and many others, both local and global.

It is however an easily observable fact that a naive analogue of \eqref{E1} fails for ${\rm WF}^{\rm geo}(\Theta(\wt{G}^{(n)}))$, where $\Theta(\wt{G}^{(n)})=\AZ({\rm St}(\wt{G}^{(n)}))$ with ${\rm St}(\wt{G}^{(n)})$ being the covering Steinberg representation of $\wt{G}^{(n)}$. In fact, $\Theta(\wt{G}^{(n)})$ could be generic, if $n$ is big enough compared to the rank of $G$. This reflects the general fact that genuine representations of $\wt{G}$ tend to possess bigger Gelfand--Kirillov dimensions.
To remedy this and posit a more reasonable formula as \eqref{E1} for covers, it is essential to define naturally a covering Barbasch--Vogan duality
$$d_{BV,G}^{(n)}: \mca{N}(\wt{G}^\vee) \longrightarrow \mca{N}(\mbf{G}),$$
which replaces $d_{BV,G}$ in the covering setting and presumably captures the geometric wavefront set of $\pi$ (or rather, of $\AZ(\pi)$) in the same line as \eqref{E1}. Here $\wt{G}^\vee$ is the complex dual group of $\wt{G}^{(n)}$ (see \cite{Wei18a}). This is the starting point and motivation of our paper.

\subsection{Main results}
The goal of our paper is of two-fold. First, we define and explicate the sought map $d_{BV, G}^{(n)}$, and moreover prove some of the expected properties. Second, using $d_{BV, G}^{(n)}$, we generalize the wavefront set upper bound conjecture to covering groups. We also give some evidence for this conjecture.  
We elaborate further and state our main results below.

First, we define the covering Barbasch--Vogan duality $d_{BV,G}^{(n)}$ using the theory of annihilator varieties of highest weight modules for complex Lie algebras, see Definition \ref{D:dBV}. The definition is clearly motivated from and parallel to the setting of Barbasch--Vogan's work \cite{BV85} concerning the linear $d_{BV,G}$. However, in the covering setting, one needs to deform the Dynkin element associated with the orbit $\mca{O}(\phi_\pi) \in \mca{N}(\wt{G}^\vee)$ and consider a non-integral highest weight module of $\mca{G}_\BC$. Then $d_{BV, G}^{(n)}$ is by definition the associated annihilator variety orbit of this highest weight module. Note that this generalizes the metaplectic Barbasch-Vogan duality considered by Barbasch--Ma--Sun--Zhu (see \cite{BMSZ23}*{Theorem B}). {As commented in loc. cit., this duality is observed by M{\oe}glin, M{\oe}glin--Renard before.}

For our purpose, it is crucial to explicate this map $d_{BV, G}^{(n)}$. For $G$ of classical type, we utilize the earlier work of Bai--Ma--Wang \cite{BMW25}, while for $G$ of exceptional type, it will follow from Bai--Gao--Xie--Wang \cite{BGWX25}. 

As an amalgam from Theorem \ref{T:ABCD} and Theorem \ref{T:EFG}, the main result of this part can be summarized as follows:
\begin{thm}
The covering Barbasch--Vogan duality $d_{BV,G}^{(n)}:\mca{N}(\wt{G}^\vee) \to \mca{N}(\mbf{G})$ is explicitly given in \S \ref{SS:fix-cov} for covers of $G$ of classical type and in Appendix \ref{A:dBVexc} for exceptional type. Moreover, it satisfies the following properties:
\begin{enumerate}
\item[(i)] it is order-reversing, i.e., $d_{BV, G}^{(n)}(\mca{O}) \gest d_{BV, G}^{(n)}(\mca{O}')$ whenever $\mca{O} \lest \mca{O}'$;
\item[(ii)] if $\wt{M} \subseteq \wt{G}$ is a covering Levi subgroup, then one has
\begin{equation} \label{ppcBV}
d_{BV, G}^{(n)} \circ \Sat_{\wt{M}^\vee}^{\wt{G}^\vee}(\mca{O}) = \Ind_{\mbf{M}}^{\mbf{G}}\circ d_{BV, M}^{(n)}(\mca{O}),
\end{equation}
for  every $\mca{O} \in \mca{N}(\wt{M}^\vee)$. 
\end{enumerate}
\end{thm}

As mentioned, we have the following clone from the linear algebraic case:
\begin{conj} \label{MC}
Let $\wt{G}$ be a Brylinski--Deligne cover of $G$. 
Then
\begin{enumerate}
\item[(i)] for every $\pi \in \Irrg(\wt{G}^{(n)})$, one has 
\begin{equation} \label{E:ubd0}
{\rm WF}^{\rm geo}(\AZ(\pi)) \lest d_{BV,G}^{(n)}(\mca{O}(\phi_\pi)),
\end{equation}
where $\phi_\pi: {\rm WD}_F \to {}^L \wt{G}$ is the hypothetical L-parameter of $\pi$;
\item[(ii)] the equality in (i) above is achieved by some $\pi$ in the tempered L-packet. {Namely, for each tempered $L$-parameter $\phi$ such that $\Pi_{\phi}$ is non-empty, there exists a $\pi \in \Pi_{\phi}$ such that
\[{\rm WF}^{\rm geo}(\AZ(\pi)) =\{ d_{BV,G}^{(n)}(\mca{O}(\phi))\}. \]
}
\end{enumerate}
\end{conj}

Since $\AZ$ is an involution, \eqref{E:ubd0} is equivalent to ${\rm WF}^{\rm geo}(\pi) \lest d_{BV,G}^{(n)}(\mca{O}(\phi_{\AZ(\pi)}))$.
As mentioned, it was shown by Karasiewicz--Okada--Wang \cite{KOW} that equality in \eqref{E:ubd0} holds for $\AZ(\pi)$ being an unramified theta representation (with the assumption that $p\nmid n$ and $p$ is big enough), see \S \ref{SS:theta}. In this case, $\pi$ is a covering Steinberg representation and thus a discrete series, it implies Part (ii) of Conjecture \ref{MC} as well.

In fact, using the property of $d_{BV, G}^{(n)}$ as in \eqref{ppcBV}, we show that proving the inequality \eqref{E:ubd0} for all $\pi\in \Irrg(\wt{G})$ is equivalent to proving the same inequality for all discrete series representations of all Levi subgroups of $\wt{G}$. This equivalence is given in Theorem \ref{T:red}, where the proof (under certain working hypothesis) parallels that for the linear case in \cite{HLLS24}.

Another main result of our paper is to show that \eqref{E:ubd0} holds for Kazhdan--Patterson covers $\wt{\GL}_r$ of $\GL_r$, under the Working Hypothesis \ref{as:dual}. 
More precisely, the following is amalgam of Theorem \ref{T:Speh} and Corollary \ref{C:GLlest}, the proof of the latter uses Theorem \ref{T:red}.

\begin{thm}
Assume the Working Hypothesis in \ref{as:dual}. Also assume $\gcd(p, n)=1$. Consider the genuine Speh representation $\pi:=Z(\rho, [a, b])$ of the Kazhdan--Patterson cover $\wt{\GL}_{rk}^{(n)}$ with $\rho$ a supercuspidal representation of $\wt{\GL}_r^{(n)}$.  
Then 
$${\rm WF}^{\rm geo}(\pi) = d_{BV,G}^{(n)}( \mca{O}(\phi_{\AZ(\pi)}) ).$$ 
Moreover, for every $\pi' \in \Irrg(\wt{\GL}_{r'}^{(n)})$, one has 
$${\rm WF}^{\rm geo}(\pi) \lest d_{BV,G}^{(n)}( \mca{O}(\phi_{\AZ(\pi)}) );$$
that is, Conjecture \ref{MC} (i) holds in this case.
\end{thm}

\subsection{Acknowledgement} 
We would like to thank Zhanqiang Bai, Caihua Luo, Jiajun Ma, Runze Wang, Xun Xie, Shilin Yu, Qixian Zhao, Chengbo Zhu, and Jiandi Zou for various in-depth discussions related to the topics discussed here. 

The work of F. G. is partially supported by the National Key R{\&}D Program of China (No. 2022YFA1005300) and also by NSFC-12171422. The work of B.\! L. is partially supported by the NSF Grant DMS-1848058 and the Travel Support for Mathematicians from Simons Foundation.

\section{Covering Barbasch--Vogan duality}

\subsection{\texorpdfstring{Definition of $d_{BV, G}^{(n)}$}{}}
We retain the notation as in \S \ref{S:intro}. In particular, we have $G=\mbf{G}(F)$ and consider $\wt{G}:=\wt{G}^{(n)}$ an $n$-fold cover of $G$ associated with a Weyl-invariant quadratic form $Q: Y \to \Z$. Let $B_Q$ be the bilinear form associated with $Q$.

Let $(X, \Phi; Y, \Phi^\vee)$ be the root datum of $\mbf{G}$. The Langlands dual group $\mbf{G}^\vee$ of $\mbf{G}$ is defined over $\Z$ with root datum $(Y, \Phi^\vee; X, \Phi)$. Let $\mbf{T}^\vee \subseteq \mbf{G}^\vee$ be its torus. Similarly, the dual group $\wm{G}^\vee$ is also defined over $\Z$, but with a modified root datum (see \cite{Wei18a})
$$(Y_{Q,n}, \Phi_{Q,n}^\vee; X_{Q,n}, \Phi_{Q,n}),$$
where 
$$Y_{Q,n}:=\set{y\in Y: \ B_Q(y, z) \in n\Z \text{ for all } z \in Y} \subseteq Y$$ is the character lattice of the dual group $\wm{T}^\vee$ of the covering torus $\wt{T}$.
Also, $X_{Q,n}=\Hom_\Z(Y_{Q,n},\Z)$.
We write $\tilde{\alpha}_{Q,n}^\vee=n_\alpha \cdot \alpha^\vee$, where $n_\alpha = n/\gcd(n, Q(\alpha^\vee))$. Then $\Phi_{Q,n}^\vee = \set{\tilde{\alpha}_{Q,n}^\vee: \ \alpha\in \Phi}$. We have
$${\rm Lie}(\mbf{T}^\vee(\R)) = \Hom(Y, \R) = X \otimes \R$$
and similarly
$${\rm Lie}(\wm{T}^\vee(\R))= \Hom(Y_{Q,n}, \R) = X_{Q,n} \otimes \R.$$
The canonical inclusion $X \to X_{Q,n}$ induces an isomorphism of the above $\R$-vector spaces, the inverse of which we denote by
$$\mfr{s}:=\mfr{s}_{Q,n}: {\rm Lie}(\wm{T}^\vee(\R)) \longrightarrow {\rm Lie}(\mbf{T}^\vee(\R)).$$

Let $\wt{G}^\vee:=\wm{G}^\vee(\BC)$ be the group of complex points of $\wm{G}^\vee$.  
Let $$\mca{O} \in \mca{N}(\wt{G}^\vee)$$ be a nilpotent orbit. Let 
$h_\mca{O} \in {\rm Lie}(\wm{T}^\vee(\R))$ be the \emph{dominant} semisimple element associated with $\mca{O}$. By applying the map $\mfr{s}$ to $h_\mca{O}/2$, one obtains an element
$$\frac{h_\mca{O}^{(n)}}{2}:=\mfr{s}\left(\frac{h_\mca{O}}{2}\right) \in {\rm Lie}(\mbf{T}^\vee(\R)) \subset {\rm Lie}(\mbf{T}^\vee(\BC)).$$
We may view $h_\mca{O}^{(n)}/2$ as a weight of the complex group $\mca{G}_\BC$. It should be noted that $h_\mca{O}^{(n)}$ depends on $Q$ as well, but we omit it to save notation.

Recall that for every weight $\lambda$ of $\mca{G}_\BC$, one has the Verma module $I(\lambda)$ of $\mfr{g}_\BC:={\rm Lie}(\mca{G}_\BC)$ and a unique quotient $L(\lambda)$, which is called the highest weight module associated with $\lambda$. Moreover, it is well-known (\cites{BV83, Jos85}) that the annihilator variety ${\rm AV}_{\mfr{g}_{\BC}}(\lambda)$ of $L(\lambda)$ satisfies
$${\rm AV}_{\mfr{g}_\BC}(\lambda) = \overline{\mca{O}}_{L(\lambda)}$$
for a unique nilpotent orbit (conveniently denoted by) $\mca{O}_{L(\lambda)} \subset \mfr{g}_\BC$. Here $\overline{\mca{O}}$ denotes the Zariski closure of $\mca{O}$. From now on, by abuse of notation, we will use ${\rm AV}_{\mfr{g}_\BC}(\lambda)$ to denote the orbit $\mca{O}_{L(\lambda)}$.

\begin{defn} \label{D:dBV}
Let $\wt{G}$ be a Brylinski--Deligne cover associated with $Q$.
Consider the duality $d_{BV, \mca{G}_\BC}^{(n)}: \mca{N}(\wt{G}^\vee) \longrightarrow \mca{N}(\mca{G}_\BC)$
given by $d_{BV, G}^{(n)}(\mca{O}):= {\rm AV}_{\mfr{g}_\BC}(h_\mca{O}^{(n)}/2)$. In view of the bijection \eqref{E:idN}, this gives rise to a duality map
$$d_{BV, G}^{(n)}: \mca{N}(\wt{G}^\vee) \longrightarrow \mca{N}(\mbf{G}),$$
which is called the covering Barbasch--Vogan duality.
\end{defn}

Henceforth, by abuse of notation, we will also write 
$$d_{BV, G}^{(n)}(\mca{O})= {\rm AV}_{\mfr{g}_\BC}(h_\mca{O}^{(n)}/2) \in \mca{N}(\mbf{G})$$
with \eqref{E:idN} implemented implicitly. Assume $\mbf{G}$ is isogenous with $\prod_i^k \mbf{G}_i$, where the derived subgroup of every $\mbf{G}_i$ is almost simple. Assume $\wt{G}$ and $\wt{\prod_{i=1}^k G_i}$ are ``compatible", for example, $\wt{G}$ is the pull-back from $\wt{\prod_{i=1}^k G_i}$ (if the isogeny is $\mbf{G} \to \wt{\prod_{i=1}^k \mbf{G}_i}$) or reversely $\wt{\prod_{i=1}^k G_i}$ is pull-back from $\wt{G}$. Then, every $\mca{O} \in \mca{N}(\wt{G}^\vee)$ is of the form $\mca{O} = \prod_{i=1}^k \mca{O}_i$ with $\mca{O}_i \in \mca{N}(\wt{G}_i^\vee)$, and we have
$$d_{BV,G}^{(n)}(\mca{O}) = \prod_{i=1}^k d_{BV,G_i}^{(n)}(\mca{O}_i).$$

It is clear that if $n=1$, then $d_{BV, G}^{(1)}$ is just the classical Barbasch--Vogan duality, see \cite{BV85}. For the classical metaplectic double cover $\wt{\Sp}_{2r}^{(2)}$, the map $d_{BV, \Sp_{2r}}^{(2)}$ agrees with the metaplectic Barbasch--Vogan duality by \cite{BMSZ23}*{Theorem B}.


\subsection{\texorpdfstring{Explicating $h_\mca{O}^{(n)}/2$}{}}
Let $\mca{O} \in \mca{N}(\wt{G}^\vee)$. Let $\Delta \subset \Phi$ be the set of simple roots of $\mbf{G}$. The simple roots of $\wm{G}^\vee$ are
$$\set{\tilde{\alpha}_{Q,n}^\vee: \alpha \in \Delta},$$
where $n_\alpha = n/\gcd(n, Q(\alpha^\vee))$ as before. The orbit $\mca{O}$ is uniquely determined by its weighted Dynkin diagram, which we assume is decorated by the numerical weights 
$$\set{c_\alpha(\mca{O}): \ \alpha\in \Delta}$$
with $c_\alpha(\mca{O}) \in \set{0,1,2}$. For each root $\tilde{\alpha}_{Q,n}^\vee \in \Delta_{Q,n}^\vee$ of $\wm{G}^\vee$, let $\tilde{\omega}_\alpha$ be the associated 
fundamental coweight of $\wm{G}^\vee$. It is clear $\tilde{\omega}_\alpha = \omega_\alpha/n_\alpha$, where $\omega_\alpha$ is the fundamental weight of $\mbf{G}$ associated with $\alpha$. This gives that
\begin{equation} \label{E:hO}
\frac{h_\mca{O}^{(n)}}{2} = \sum_{\alpha \in \Delta} \frac{c_\alpha(\mca{O})}{2n_\alpha} \cdot \omega_\alpha \in X\otimes \R.
\end{equation}

In principle, for every concrete $\wt{G}^{(n)}$ and $\mca{O} \in \mca{N}(\wt{G}^\vee)$, one can then compute ${\rm AV}_{\mca{G}_\BC}(h_\mca{O}^{(n)}/2)$ and thus $d_{BV,G}^{(n)}$ by using the results in \cite{BMW25} and \cite{BGWX25}. For each $G$ of classical type, however, it is important to give a uniform combinatorial formula of $d_{BV, G}^{(n)}$ in terms of the partition representing $\mca{O}$. This we achieve in \S \ref{S:dBV-cla}. For exceptional groups, we explain more in \S \ref{S:dBV-exc}.

\subsection{Conjecture on the wavefront sets for covers} \label{SS:WFconj}
In \cite{Wei18a}, Weissman defined an L-group ${}^L\wt{G}$ of $\wt{G}^{(n)}$ which sits in a short exact sequence
$$\wt{G}^\vee \into {}^L \wt{G} \onto W_F.$$
This is a split extension, and most often (see \cite{GG18}*{\S 6} for more details) there is a splitting such that ${}^L \wt{G} \simeq \wt{G}^\vee \times W_F$; in this paper we assume and fix such an isomorphism arising from a distinguished genuine central character of $Z(\wt{T})$.  This also gives rise to an isomorphism ${}^L\wt{M} \simeq \wt{M}^\vee \times W_F$ for a covering Levi subgroup $\wt{M} \subset \wt{G}$. We have natural embedding ${}^L\wt{M} \into {}^L \wt{G}$ compatible with the embedding $\wt{M}^\vee \into \wt{G}^\vee$.

In any case, one has a natural identification
$$\mca{N}({}^L \wt{G}) \simeq \mca{N}(\wt{G}^\vee).$$
Let ${\rm WD}_F:=W_F \times \SL_2(\BC)$ be the Weil-Deligne group of $F$. Given an irreducible genuine representation $\pi$ of $\wt{G}^{(n)}$, denote by 
$$\phi_{\pi}: {\rm WD}_F \longrightarrow {}^L \wt{G}$$
the hypothetical L-parameter of $\pi$. Then we have a nilpotent orbit
$$\mca{O}(\phi_\pi) \in \mca{N}(\wt{G}^\vee)$$
associated with $\phi_\pi(1, e_\alpha)$, where $e_\alpha$ is the upper unipotent element inside $\SL_2$.

We also consider the Aubert--Zelevinsky dual ${\rm AZ}(\pi)$ of $\pi$, defined exactly as in the linear case.
More precisely, let $\msc{R}(\wt{G})$ be the Grothendieck group of smooth admissible genuine $\wt{G}$-representations of finite length. The Aubert--Zelevinsky duality  (see \cites{Zel80, Aub95})
$$\text{AZ}:\msc{R}(\wt{G})\longrightarrow \msc{R}(\wt{G})$$  is given by 
\begin{equation} \label{E:AZdef}
    \text{AZ}(\pi)=\sum_{\substack{L \text{ Levi} \\  T \subseteq L} } (-1)^{\text{rank}({L})}\text{Ind}^{\wt{G}}_{\wt{L}} \circ \text{Jac}^{\wt{G}}_{\wt{L}}(\pi).
\end{equation}
Here ${\rm Jac}_{\wt{L}}^{ \wt{G} }$ is the Jacquet functor and ${\rm Ind}_{\wt{L}}^{\wt{G}}$ is the normalized parabolic induction. {It is proved in \cite{BJ16}*{Theorem 4.2, 4.3} that $\AZ$ is an involution and preserves irreducibility.}

Conjecture \ref{MC} is a generalization of the conjectures given by Hazeltine--Liu--Lo--Shahidi \cite{HLLS24} and Ciubotaru--Kim \cite{CK24} in the linear case. We recall it as follows. 

\begin{conj} \label{Conj}
Let $\wt{G}^{(n)}$ be a Brylinski--Deligne cover of $G$. 
Let 
$$d_{BV, G}^{(n)}: \mca{N}(\wt{G}^\vee) \longrightarrow \mca{N}(\mbf{G})$$ be the covering Barbasch--Vogan duality
given in Definition \ref{D:dBV}. Then
\begin{enumerate}
\item[(i)] for every $\pi \in \Irrg(\wt{G}^{(n)})$, one has 
\begin{equation} \label{E:ineq}
{\rm WF}^{\rm geo}(\AZ(\pi)) \lest d_{BV, G}^{(n)}(\mca{O}(\phi_\pi)),
\end{equation}
where $\phi_\pi: {\rm WD}_F \to {}^L \wt{G}$ is the L-parameter of $\pi$;
\item[(ii)] if $\phi_\pi$ is tempered, then the equality in (i) above is achieved by some $\pi$ in the L-packet.
\end{enumerate}
\end{conj}


\section{\texorpdfstring{Properties of $d_{BV, G}^{(n)}$ for classical $G$}{}}  \label{S:dBV-cla}
The goal of this section is to explicate for $G$ of classical map the map $d_{BV,G}^{(n)}$, which is uniformly but abstractly defined. Meanwhile, we also prove some of its expected properties. For the Dynkin diagram of root system of $G$, we follow Bourbaki's notation as in \cite{BouL2}. 

For any partition $\mfr{p}$ we write $\mfr{p}^*$ for its transpose. For $X \in \set{B, C, D}$, we use $\mfr{p}_X$ to denote the $X$-collapse of $\mfr{p}$. Also, if we arrange the partition $\mfr{p}=[p_1, p_2, \ldots, p_k]$ with $p_1 \gest p_2 \gest \cdots. \gest p_k$, then $\mfr{p}^+:=[p_1+1, p_2, \ldots, p_k]$ and $\mfr{p}^-:=[p_1, \ldots, p_{k-1}, p_k-1]$. The two natural maps
$$(\cdot)^-{}_C: \mca{N}(\SO_{2r+1}) \longrightarrow \mca{N}(\Sp_{2r})$$
and 
$$(\cdot)^+{}_B: \mca{N}(\Sp_{2r}) \longrightarrow \mca{N}(\SO_{2r+1})$$
are then both bijections, and are inverses of each other. We often omit the parentheses between the superscript and subscript. For example, we shall write $\mfr{p}_{D} {}^{+} \vphantom{\mfr{p}}_{B}  {}^{-\ast}$ instead of $((((\mfr{p}_{D})^{+})_{B})^{-})^{\ast}$.

\subsection{\texorpdfstring{Fix cover $\wt{G}^{(n)}$ of type $A, B, C, D$}{}} \label{SS:fix-cov}
\subsubsection{\texorpdfstring{Type $A_r$}{}}\label{sec combinatorial definition}
Let $G$ be of type $A_r$ with Dynkin diagram given as follows.
\vskip 5pt
$$
\qquad 
\begin{picture}(4.7,0.2)(0,0)
\put(1,0){\circle{0.08}}
\put(1.5,0){\circle{0.08}}
\put(2,0){\circle{0.08}}
\put(2.5,0){\circle{0.08}}
\put(3,0){\circle{0.08}}
\put(1.04,0){\line(1,0){0.42}}
\multiput(1.55,0)(0.05,0){9}{\circle{0.02}}
\put(2.04,0){\line(1,0){0.42}}
\put(2.54,0){\line(1,0){0.42}}
\put(1,0.1){\footnotesize $\alpha_{1}$}
\put(1.5,0.1){\footnotesize $\alpha_{2}$}
\put(2,0.1){\footnotesize $\alpha_{r-2}$}
\put(2.5,0.1){\footnotesize $\alpha_{r-1}$}
\put(3,0.1){\footnotesize $\alpha_{r}$}
\end{picture}
$$
\vskip 10pt

Let $\wt{G}$ be an $n$-fold cover associated with $Q$. The number $n_\alpha=n/\gcd(n,Q(\alpha^\vee))$ is independent of the choice of the root $\alpha$. For simplicity, we consider cover of $G$ such that $n_\alpha =n$.

The dual group $\wt{G}^\vee$ is also of type $A$.  Let $\mca{O}_\mfr{p} \subset \wt{G}^\vee$ be an orbit associated with a partition
$$\mfr{p}=(p_1, p_2, \ldots, p_k).$$
For every number $m\gest 1$ we consider the partition 
$$\mfr{s}(m;n):=(n^a b)=(n, n, \ldots, n, b)$$
of $m$, where $m=n a + b$ with $0\lest b < n$. 
Then we define a combinatorial duality
\begin{equation} \label{E:dBV-A}
d_{com,A}^{(n)}(\mfr{p}) := \sum_{i=1}^k \mfr{s}(p_i; n).
\end{equation}
Here, sum of partitions is the obvious one. Note that if $n=1$, then we get
$$d_{com,A}^{(1)}(\mfr{p}) = \mfr{p}^*.$$
The map $d_{com, A}^{(n)}$ is our formula-candidate for $d_{BV,G}^{(n)}$, likewise for other classical types below.

\subsubsection{\texorpdfstring{Type $B_r$}{}} 
Consider $G=\SO_{2r+1}$ with simple roots given as follows
\vskip 5pt
$$ \qquad 
\begin{picture}(4.7,0.2)(0,0)
\put(1,0){\circle{0.08}}
\put(1.5,0){\circle{0.08}}
\put(2,0){\circle{0.08}}
\put(2.5,0){\circle{0.08}}
\put(3,0){\circle{0.08}}
\put(1.04,0){\line(1,0){0.42}}
\multiput(1.55,0)(0.05,0){9}{\circle*{0.02}}
\put(2.04,0){\line(1,0){0.42}}
\put(2.54,0.015){\line(1,0){0.42}}
\put(2.54,-0.015){\line(1,0){0.42}}
\put(2.68,-0.05){{\large $>$}}
\put(1,0.1){\footnotesize $\alpha_1$}
\put(1.5,0.1){\footnotesize $\alpha_2$}
\put(2,0.1){\footnotesize $\alpha_{r-2}$}
\put(2.5,0.1){\footnotesize $\alpha_{r-1}$}
\put(3,0.1){\footnotesize $\alpha_r$}
\end{picture}
$$
\vskip 10pt

Let $\wt{G}$ be the $n$-fold cover associated with the unique $Q$ such that $Q(\alpha_1^\vee)=2$ (and thus $Q(\alpha_r^\vee)=4$). That is, $\wt{G}$ is the $n$-fold cover obtained from restricting the $n$-fold cover of $\SL_{2r+1}$ associated with $Q(\alpha^\vee)=1$ for any root $\alpha$ of $\SL_{2r+1}$. One has
$$\wt{G}^\vee \simeq
\begin{cases}
\Sp_{2r} & \text{ if $n$ is odd or $n=2k$ with $k$ odd},\\
\SO_{2r+1} & \text{ otherwise.}
\end{cases}
$$
Given any orbit of $\wt{G}^\vee$ represented by a partition $\mfr{p}$, we define 
\begin{equation} \label{E:dBV-B}
d_{com,B}^{(n)}(\mfr{p}):= 
\begin{cases}
d_{com,A}^{(n)}(\mfr{p})^{+}\sub{B} & \text{ if  $n$ is odd};\\
d_{com,A}^{(n/2)}(\mfr{p})^{+}\sub{B} & \text{ if  $n$ is even with $n/2$ odd};\\
d_{com,A}^{(n/2)}(\mfr{p})_B & \text{ if  $n$ is even with $n/2$ even}.
\end{cases}
\end{equation}
Note that the above formula can be rendered simpler if we write $n^{\ast}:= n/ \gcd(n,2)$, and re-group the definition into only two cases. 

\subsubsection{\texorpdfstring{Type $C_r$}{}}
Consider $G=\Sp_{2r}$ with simple roots given as follows
$$ \qquad 
\begin{picture}(4.7,0.2)(0,0)
\put(1,0){\circle{0.08}}
\put(1.5,0){\circle{0.08}}
\put(2,0){\circle{0.08}}
\put(2.5,0){\circle{0.08}}
\put(3,0){\circle{0.08}}
\put(1.04,0){\line(1,0){0.42}}
\multiput(1.55,0)(0.05,0){9}{\circle*{0.02}}
\put(2.04,0){\line(1,0){0.42}}
\put(2.54,0.015){\line(1,0){0.42}}
\put(2.54,-0.015){\line(1,0){0.42}}
\put(2.68,-0.05){{\large $<$}}
\put(1,0.1){\footnotesize $\alpha_1$}
\put(1.5,0.1){\footnotesize $\alpha_2$}
\put(2,0.1){\footnotesize $\alpha_{r-2}$}
\put(2.5,0.1){\footnotesize $\alpha_{r-1}$}
\put(3,0.1){\footnotesize $\alpha_r$}
\end{picture}
$$
\vskip 10pt
We consider $Q$ such that $Q(\alpha_r^\vee)=1$. Let $\wt{G}$ be the $n$-fold cover of $\Sp_{2r}$.
It gives
$$
\wt{G}^\vee=
\begin{cases}
\SO_{2r+1} & \text{ if $n$ is odd};\\
\Sp_{2r} & \text{ if $n$ is even}.
\end{cases}
$$
Given any orbit of $\wt{G}^\vee$ represented by a partition $\mfr{p}$, we define
\begin{equation} \label{E:dBV-C}
d_{com,C}^{(n)}(\mfr{p}) := 
\begin{cases}
d_{com,A}^{(n)}(\mfr{p})^{-} \sub{C} & \text{ if  $n$ is odd};\\
d_{com,A}^{(n/2)}(\mfr{p})^{+-}\sub{C} & \text{ if  $n$ is even with $n/2$ odd};\\
d_{com,A}^{(n/2)}(\mfr{p})_C & \text{ if  $n$ is even with $n/2$ even}.
\end{cases}
\end{equation}
Here for a partition $\mfr{p}$, the partition $(\mfr{p}^+)^-$ is as explained at the beginning of this section.

\subsubsection{\texorpdfstring{Type $D_r$}{}}
Consider $G=\SO_{2r}$ with Dynkin diagram given as follows
\vskip 10pt
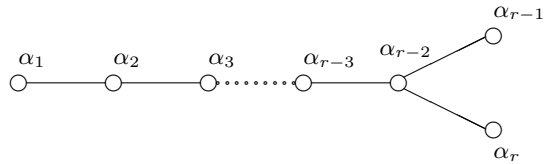
$$
\begin{picture}(4.7,0.4)(0,0)
\put(1,0){\circle{0.08}}
\put(1.5,0){\circle{0.08}}
\put(2,0){\circle{0.08}}
\put(2.5,0){\circle{0.08}}
\put(3,0){\circle{0.08}}
\put(3.5, 0.25){\circle{0.08}}
\put(3.5, -0.25){\circle{0.08}}
\put(1.04,0){\line(1,0){0.42}}
\put(1.54,0){\line(1,0){0.42}}
\multiput(2.05,0)(0.05,0){9}{\circle{0.02}}
\put(2.54,0){\line(1,0){0.42}}
\put(3.03,0.03){\line(2,1){0.43}}
\put(3.03,-0.03){\line(2,-1){0.43}}
\put(1,0.1){\footnotesize $\alpha_1$}
\put(1.5,0.1){\footnotesize $\alpha_2$}
\put(2,0.1){\footnotesize $\alpha_3$}
\put(2.5,0.1){\footnotesize $\alpha_{r-3}$}
\put(2.9,0.15){\footnotesize $\alpha_{r-2}$}
\put(3.5,0.35){\footnotesize $\alpha_{r-1}$}
\put(3.5,-0.4){\footnotesize $\alpha_r$}
\end{picture}
$$
\vskip 30pt
Consider the $n$-fold cover $\wt{\SO}_{2r}$ obtained from restricting $\wt{\SL}_{2r}$ to the embedded $\SO_{2r} \subset \SL_{2r}$. Then $\wt{\SO}_{2r}$ is the one associated with $Q(\alpha_i^\vee)=2$ for every $1\lest i \lest r$. We have
$$\wt{G}^\vee \simeq \SO_{2r}.$$
For each partition $\mfr{p}$ of $\wt{G}^\vee$, we define
\begin{equation} \label{E:dBV-D}
d_{com,D}^{(n)}(\mfr{p}) := 
\begin{cases}
d_{com,A}^{(n)}(\mfr{p})_D & \text{ if  $n$ is odd};\\
d_{com,A}^{(n/2)}(\mfr{p})_D & \text{ if  $n$ is even}.
\end{cases}
\end{equation}
Also, if $r$ is even and $\mfr{p}^{\rm I}=\mca{O}^{\rm I}, \mfr{p}^{\rm II}=\mca{O}^{\rm II}$ are the two orbits associated with a very even partition $\mfr{p}$, where we follow Lusztig's convention on $\set{\rm I, II}$ as in \cite{BMW25}. Suppose $d_{com, D}^{(n)}(\mfr{p})$ is also a very even partition, then for $\set{\heartsuit, \spadesuit} = \set{\rm I, II}$ we require that
\begin{equation} \label{E:dBV-D 2}
d_{com,D}^{(n)}(\mfr{p}^\heartsuit) := 
\begin{cases}
d_{com,D}^{(n)}(\mfr{p})^\heartsuit & \text{ if  $r/2$ is even};\\
d_{com,D}^{(n)}(\mfr{p})^{\spadesuit} & \text{ if  $r/2$ is odd}.
\end{cases}
\end{equation}

For $G$ as above type, we will relate the covering Barbasch--Vogan duality $d_{BV, G}^{(n)}$ given in Definition \ref{D:dBV}  with $d_{com, X}^{(n)}, X\in \set{A, B, C, D}$, using the work \cite{BMW25}.

\subsection{\texorpdfstring{Annihilator partition $\AP{X}(\lambda_{X'}^{(n)}(\mfr{p}))$}{}}
Let $X\in\{A,B,C,D\}$, and let $\mathfrak{g}$ be a simple complex Lie algebra of type $X$ with a Cartan subalgebra $\mathfrak{h}$. 
For any weight $\lambda \in \mathfrak{h}^{\ast}$, we set
$$\AP{X}(\lambda):=\text{the partition underlying } {\rm AV}_\mfr{g}(\lambda),$$
which is a partition of type $X$. We call $\AP{X}(\lambda)$ the annihilator partition associated with $L(\lambda)$. Clearly, $\AP{X}(\lambda)$ also determines ${\rm AV}_\mfr{g}(\lambda)$ for $X\in {A, B, C}$.

For $\lambda \in \mfr{h}^{\ast}$, we identify 
$$\lambda= (\lambda_1,\ldots, \lambda_r)= \sum_{i=1}^r \lambda_i \varepsilon_i,$$ where $\lambda_i \in \BC$ and $\{\varepsilon_i\}_{i=1}^r$ is the canonical basis of the Euclidean space $\R^r$. 

\begin{lemma} \label{L:out}
Let $\mfr{g}$ be of type $D_r$ and let $\sigma$ be the outer automorphism of $\mfr{g}$ associated with the exchange of $\alpha_{r-1}:= e_{r-1} - e_r$ and $\alpha_r:=e_{r-1} + e_r$ in the Dynkin diagram. Then for every $\lambda \in \mfr{h}^*$, one has $\AP{D}(\lambda) = \AP{D}({}^\sigma \lambda)$.
\end{lemma}
\begin{proof}
This is a direct consequence of Joseph's formula \cite{Jos85} for computing ${\rm AV}_\mfr{g}(\lambda)$, see also \cite{BMW25}*{Theorem 7.10}.
\end{proof}

In the case we consider, $\lambda$ is always a non-increasing sequence of real numbers. Therefore, we also identify $\lambda$ as the multi-set $\{\lambda_i\}_{i=1}^r$ and then use summation to denote the disjoint union of multi-sets.  Let $\lambda=\{\lambda_1,\ldots, \lambda_r\}$ be a multi-set of real numbers. We define
        \[ {}^-\lambda:= \lambda+ \{-\lambda_r, \ldots, -\lambda_1 \}.\]   
For any $n\in \Z_{>0}$ and set 
$$n^{\ast}:= \frac{n}{\gcd(n,2)}.$$

\begin{defn}\label{def wt partition}
    Let $\mathfrak{p}$ be a partition of type $X$, where $X\in \{A,B,C,D\}$.  Define a multi-set of real numbers $\lambda_{X}^{(n)}(\mathfrak{p})$ case by case as follows. We shall also regard $\lambda_{X}^{(n)}(\mathfrak{p})$ as a
    non-increasing sequence of real numbers.
    
    \begin{enumerate}
        \item [(a)] When $X=A$, write $\mfr{p}=[p_1,\ldots, p_r]$. Define
        \[ \lambda_A^{(n)}(\mathfrak{p}):= \sum_{i=1}^r \left\{ \frac{p_i-1}{2n},\frac{p_i-3}{2n},\ldots, \frac{1-p_i}{2n}  \right\}. \]
        \item [(b)] When $X\in \{B, C,D\}$, define $\lambda^{(n)}_X(\mfr{p})$ to be the multi-set of non-negative real numbers such that
        \[ \lambda^{(n^{\ast})}_A(\mfr{p})=\begin{cases}
            {}^-\lambda^{(n)}_X(\mfr{p})+\{0\} & \text{ if }X=B,\\
            {}^-\lambda^{(n)}_X(\mfr{p})& \text{ if }X\in \{C,D\}.
        \end{cases}\]     
    \end{enumerate}
\end{defn}

We are interested in $\AP{X}(\lambda)$ when the weight $\lambda$ is equal to $\lambda^{(n)}_{X'}(\mfr{p})$, where $\mfr{p}$ is a partition of type $X'$.
Indeed, it follows from the explicit form of $h_\mca{O}^{(n)}/2$ in \eqref{E:hO} that 
$$\lambda_{X'}^{(n)}(\mfr{p}_\mca{O}) =
\begin{cases}
 h_\mca{O}^{(n)}/2 & \text{ for } X'=A,B,C,\\
 h_\mca{O}^{(n)}/2 \text{ or }  {}^\sigma h_\mca{O}^{(n)}/2 & \text{ for } X'=D,
 \end{cases}$$ 
where $\mfr{p}_\mca{O}$ is the partition corresponding to  $\mca{O}$ and $\sigma \in {\rm Aut}(\mfr{g})$ is as in Lemma \ref{L:out}. It also follows from Lemma \ref{L:out} that in order to compute $\AP{D}(\lambda_D^{(n)}(\mfr{p}))$, there is no loss of generality in assuming that $\lambda_D^{(n)}(\mfr{p})$ contains non-negative real numbers, as in Definition \ref{def wt partition} (b).

Here is the first main result of this section.


\begin{thm}\label{T:comparison}
    Let $(X,X')\in \{ (A,A),(D,D)\} \cup (\{B,C\}\times\{B,C\})$ and $n \in \Z_{>0}$. Suppose that $\mfr{p}$ is a partition of type $X'$ such that $d_{com,X}^{(n)}(\mfr{p})$ is well-defined. Then we have 
    \[ d_{com,X}^{(n)}(\mfr{p})= \AP{X}(\lambda_{X'}^{(n)}(\mfr{p})).\]
 In particular, for $G$ of type $A, B, C$, one has $d_{BV,G}^{(n)}(\mfr{p}) = d_{com,X}^{(n)}(\mfr{p})$, where the cover $\wt{G}^{(n)}$ is fixed as in \S \ref{SS:fix-cov}.
\end{thm}
A substantial part of the remaining part of this section is devoted to the proof of Theorem \ref{T:comparison}. Along the proof, we also verify the order-reversing property and compatibility of $d_{BV, G}^{(n)}$ with orbit-induction. The determination of the labelling I, II for very even of type $D$ is given in Proposition \ref{P:D-I/II}.
All these results are then summarized in Theorem \ref{T:ABCD}, which echoes and enriches Theorem \ref{T:comparison} above.

\subsection{\texorpdfstring{A combinatorial description of $\AP{X}(\lambda^{(n)}_{X'}(\mfr{p}))$}{}}
In this subsection, we rephrase the combinatorial computation of $ \AP{X}(\lambda)$ in \cite{BMW25} when $\lambda$ is of the form $\lambda^{(n)}_{X'}(\mfr{p})$.

Let $\lambda$ be a finite sequence of real numbers. We let $P(\lambda)$ denote the Young's tableaux associated to $\lambda$ via the Robinson--Schensted (RS) insertion algorithm (see \cite{Sag01}*{\S 3.1} or  \cite{BMW25}*{\S 2.2} for details), and let $\mfr{p}(\lambda)$ be the associated partition. We write $P_{ij}(\lambda)$ for the entry of the $j$-th box of the $i$-th row of $P(\lambda)$.

In the cases we are interested in, the sequences of real numbers $\lambda$ involved are always non-increasing. In this case, we have a formula for $\mfr{p}(\lambda)$ and $P(\lambda)$.

\begin{lemma}\label{lem RS}
    Suppose that $\lambda$ is a finite non-increasing sequence of real numbers. Identify $\lambda$ as a multi-set and write $\lambda=\{\lambda_1^{r(\lambda_1)},\ldots, \lambda_s^{r(\lambda_s)}\}$, where $\lambda_1 > \cdots > \lambda_s$ and $r(\lambda_i) \in \Z_{>0}$ indicates the multiplicity. Then $\mfr{p}(\lambda)=[r_1,\ldots,r_s]$, where $\{r_i\}_{i=1}^s= \{r(\lambda_i)\}_{i=1}^s$ as multi-sets. Moreover, if we write
    \[ \{ \lambda_k \ | \ r(\lambda_k)\gest r_j  \}= \{ \lambda_{1,j}, \lambda_{2,j}, \ldots,  \lambda_{k_j,j}\} \]
    with $ \lambda_{1,j}< \lambda_{2,j}<\cdots< \lambda_{k_j,j}$, then $P_{ij}(\lambda)=\lambda_{i,j}$.
\end{lemma}
\begin{proof}
    We apply induction on $s$. When $s=1$, the assertion is clear. For $s >1$, let $\lambda^\sharp:=\{\lambda_1^{r(\lambda_1)},\ldots, \lambda_{s-1}^{r(\lambda_{s-1})}\} $. Suppose that $r(\lambda_s)= r_k$. The induction hypothesis on $\lambda^\sharp$ implies that 
    \[P_{ij}(\lambda^\sharp)=\begin{cases}
        \lambda_{i+1,j} &\text{ if } 1 \lest j \lest r(\lambda_s),\\
        \lambda_{i,j} & \text{ otherwise }.
    \end{cases}\]
    Then by the definition of Robinson-Schensted insertion algorithm, it is not hard to check that for $1 \lest x \lest r(\lambda_s)$, we have
    \[
        P_{ij}(\lambda^\sharp + \{ \lambda_s^x  \})= \begin{cases}
                    \lambda_{i+1,j}& \text{ if }x< j \lest r(\lambda_s),\\
            \lambda_{i,j} & \text{ otherwise.}
        \end{cases}\]
        This completes the proof of the lemma.
\end{proof}

With the above lemma, we can characterize whether $\lambda$ is of the form $\lambda_A^{(n)}(\mfr{p})$ for some partition $\mfr{p}$.

\begin{lemma}\label{lem partition of wt}
Let $\lambda=\{\lambda_1^{r(\lambda_1)},\ldots, \lambda_s^{r(\lambda_s)}\}$ be a multi-set of real numbers, where $\lambda_1> \cdots > \lambda_s$ and $r(\lambda_i)\in \Z_{>0}$ indicate the multiplicity. Then $\lambda= \lambda_A^{(n)} (\mfr{p})$ for some partition $\mfr{p}$ if and only if the following conditions hold.
\begin{enumerate}
    \item [(i)] For any $1 \lest i \lest s$,  we have $\lambda_i \in \frac{1}{2n}\Z$.
    \item [(ii)] For any $1 \lest i \lest s$, we have $\lambda_i= - \lambda_{s+1-i}$ and $r(\lambda_i)= r(\lambda_{s+1-i})$.
    \item [(iii)] For any $1 \lest i \lest \lfloor \half{s} \rfloor$, we have $\lambda_{i}-\lambda_{i+1}= \frac{1}{n}$ and $r(\lambda_i) \lest r(\lambda_{i+1})$.
\end{enumerate}
Moreover, in this case we have $\lambda= \lambda_A^{(n)}(\mfr{p}(\lambda)^{\ast})$.
\end{lemma}
\begin{proof}
    The necessary direction follows from the definition of $\lambda_A^{(n)}(\lambda)$. For the sufficient direction, we apply induction on $R:=\max_{1\lest i\lest s}( r(\lambda_i))$. Note that these conditions imply that $\lambda_i= \frac{s-2i+1}{2n}$. When $R=1$, we have $\lambda= \lambda_{A}^{(n)}([s])=\lambda_{A}^{(n)}( \mfr{p}(\lambda)^{\ast}  ) $. In general, let $\lambda^\sharp:=\{\lambda_1^{r(\lambda_1)-1},\ldots, \lambda_s^{r(\lambda_s)-1} \}$. The induction hypothesis and Lemma \ref{lem RS} for $\lambda^\sharp$ imply that 
    \[ \lambda^\sharp= \lambda_A^{(n)}( \mathfrak{p}(\lambda^\sharp)^{\ast}  )=\lambda_A^{(n)}( [r_1-1,\ldots, r_s-1]^{\ast}  ), \]
    where $r_1>\cdots>r_s$ and $\{r_i\}_{i=1}^s= \{r(\lambda_i)\}_{i=1}^s$ as multi-sets. Therefore, 
    \begin{align*}
        \lambda&= \lambda^\sharp + \left\{ \frac{s-1}{2n},\frac{s-3}{2n},\ldots, \frac{1-s}{2n} \right\}\\
        &= \lambda_A^{(n)}( \mathfrak{p}(\lambda^\sharp)^{\ast}  \sqcup [s])\\
        &= \lambda_A^{(n)}( [r_1,\ldots, r_s]^{\ast})\\
        &= \lambda_A^{(n)}( \mathfrak{p}(\lambda)^{\ast}  ).
    \end{align*}
    This completes the proof of the lemma.
\end{proof}

Next, we state several special cases for $\AP{X}(\lambda^{(n)}_{X'}(\mfr{p}))$ when $n=1,2$. Most of them are already proved in the literature. These special cases will be used to interpolate the general case.

\begin{prop}\label{prop ann n=1}
Let $\mfr{q}$ be a partition of type $B, C$ or $D$.
\begin{enumerate}
    \item [(a)] Suppose $\mfr{q}$ is of type $B$. Then the following holds.
    \begin{enumerate}
        \item [(a1)] $\AP{B}(\lambda_{B}^{(2)}(\mfr{q}))= \mfr{q}^{\ast}$ if every part of $\mfr{q}$ is odd.
        \item [(a2)] $\AP{C}(\lambda_{B}^{(1)}(\mfr{q}))= \mfr{q}^- \sub{C} {}^\ast$.
    \end{enumerate}
    \item [(b)] Suppose $\mfr{q}$ is of type $C$. Then the following holds.
    \begin{enumerate}
        \item [(b1)] $\AP{B}(\lambda_{C}^{(1)}(\mfr{q}))=  \mfr{q}^{\ast+}\sub{B} $.
        \item [(b2)] $\AP{C}(\lambda_{C}^{(2)}(\mfr{q}))=  \mfr{q}^\ast  \sub{D} {}^{+-} \sub{C}  $.
        \item [(b3)] $\AP{D}(\lambda_{C}^{(2)}(\mfr{q}))=  \mfr{q} \sub{D} {}^\ast \sub{D}     $ if every part of $\mfr{q}$ is even.
    \end{enumerate}
    \item [(c)] Suppose $\mfr{q}$ is of type $D$. Then the following holds.
    \begin{enumerate}
        \item [(c1)] $\AP{D}(\lambda_D^{(1)}(\mfr{q}))= \mfr{q}^{\ast}\sub{D}$.
    \end{enumerate}
\end{enumerate}
\end{prop}
\begin{proof}
    For Parts (a2), (b1) and (c1), the formulas on the right hand side are the usual Barbasch--Vogan duality maps, which follows from \cite{BV85}, see also \cite{CMBO21}*{\S 4.3}. Part (b2) is established in \cite{BMSZ23}*{Theorem A}, where the right hand side of the formula is the metaplectic Barbasch--Vogan duality map defined in the same paper. For Part (b3), comparing the case of type $C$ and $D$ in \cite{BMW25}*{Theorem 1.7}, we get $\AP{D}(\lambda_{C}^{(2)}(\mfr{q}))= (\mfr{q}^\ast  \sub{D} {}^{+-} \sub{C})\sub{D}$. Note that $\mfr{q}^{\ast}$ is already of type $D$ by our assumption that every part of $\mfr{q}$ is even. Thus using \cite{Ach03}*{Lemma 3.3}, we get 
\[\AP{D}(\lambda_{C}^{(2)}(\mfr{q}))= (\mfr{q}^\ast  \sub{D} {}^{+-} \sub{C})\sub{D}= \mfr{q}^{\ast} \sub{D} {}^\ast \sub{D} {}^\ast \sub{D}= \mfr{q} \sub{D} {}^{\ast} \sub{D}.\]
This proves Part (b3).

    Now we show Part (a1).  Under the assumption that every part of $\mfr{q}$ is odd, every element of the multi-set $\lambda:=\lambda_{B}^{(2)}(\mfr{q})$ lies in $\Z$. We recall the process for computing $\AP{B}(\lambda^{(2)}_{B}(\mfr{q}))$ in this case from \cite{BMW25}*{Theorem 1.7}. 
    
    First, we apply Robinson--Schensted algorithm on ${}^-\lambda$ to get $\mfr{p}({}^-\lambda)$. Next, write $\mfr{p}({}^-\lambda)= [p_1,\ldots, p_{2m+1}]$ where $p_{2m} >0$ and $p_{2m+1}\gest 0$. We obtain a $B$-symbol 
    \[\Lambda= \begin{pmatrix}\lambda_1~\lambda_2~\dots~\lambda_{m+1}\\\mu_1~\mu_2~\dots~\mu_m
	\end{pmatrix}, \]
 where $ \{\lambda_1,\ldots, \lambda_{m+1}\}$ and $\{\mu_1,\ldots, \mu_{m}\}$ are strictly increasing sequences of non-negative integers such that  
 \[ \{ 2 \lambda_i ,2 \mu_j+1\ | \ 1 \lest i\lest m+1, 1 \lest j \lest m \}= \{ p_{l}+2m+1-l \ | \ 1 \lest l \lest 2m+1 \}. \]
 We call this multi-set $\Omega$.
 
On the other hand, for each partition $\mfr{p}'=[p_1',\ldots, p_{2m+1}']$ of type $B$, we can also associate a $B$-symbol
 \[\Lambda(\mfr{p}')= \begin{pmatrix}\lambda_1'~\lambda_2'~\dots~\lambda_{m+1}'\\\mu_1'~\mu_2'~\dots~\mu_m'
	\end{pmatrix}, \]
where $ \{\lambda_1',\ldots, \lambda_{m+1}'\}$ and $\{\mu_1',\ldots, \mu_{m}'\}$ are strictly increasing sequences of non-negative integers such that  
 \[ \{ 2 \lambda_i'+1 ,2 \mu_j'\ | \ 1 \lest i\lest m+1, 1 \lest j \lest m \}= \{ p_{l}'+2m+1-l \ | \ 1 \lest l \lest 2m+1 \}. \]
We denote this multi-set by $\Omega(\mfr{p}')$. Then, $\AP{B}(\lambda^{(2)}_{B}(\mfr{q}))$ is the unique type $B$ special partition such that $\Lambda(\AP{B}(\lambda^{(2)}_{B}(\mfr{q})) )$ is a permutation of $\Lambda$.

To prove Part (a1), it suffices to show that if every part of $\mfr{q}$ is odd, then  $\Lambda(\mfr{q}^{\ast})=\Lambda$. Note that $\mfr{q}^{\ast}$ is special since so is $\mfr{q}$. Indeed, if we write $\mfr{q}^{\ast}= [a^r] \sqcup (\mfr{q}^{\ast})_{<a}$, where every part of $(\mfr{q}^{\ast})_{<a}$ is less than $a$, then both $a$ and $r$ are odd. Recall that ${}^- \lambda + \{0\}= \lambda_{A}^{(1)}(\mfr{q})$. Thus by Lemmas \ref{lem RS} and \ref{lem partition of wt}, we have $\mfr{q}^{\ast}= \mfr{p}(\lambda_{A}^{(1)}(\mfr{q}) )$ and
$\mfr{p}({}^- \lambda)=[a^{r-1},a-1] \sqcup (\mfr{q}^{\ast})_{<a}$. Therefore,
\[ \Omega= \Omega(\mfr{q}^{\ast})- \{a+ 2m+1-r\}+ \{a-1+2m+1-r\},\]
and a direct computation shows that $\Lambda(\mfr{q}^{\ast})=\Lambda$. This completes the proof of Part (a1) and the proposition.
\end{proof}

Now we give several notations in order to state a combinatorial formula for $\AP{X}(\lambda^{(n)}_{X'}(\mfr{p}))$.  Suppose that $\lambda=\lambda_A^{(n)}(\mfr{q})$ for some partition $\mfr{q}$. We decompose 
\begin{align}\label{eq decomp lambda A}
    \lambda= \sum_{i=1}^{l_A} \lambda_{A,i}
\end{align}
as multi-sets, where $\{\lambda_{A,i}\}_{i=1}^{l_A}$ are the maximal sub-multi-sets of $\lambda$ such that any two elements have integral difference. Then we set $\mathfrak{p}_{A,i}(\lambda):= \mfr{p}( \lambda_{A,i} )^{\ast}$.

Suppose $\lambda = \lambda_{X'}^{(n)}(\mfr{q})$, where $X'\in \{B,C,D\}$ and $\mfr{q}$ is a partition of type $X'$. Note that fixing $\lambda$ and $X'$, such $\mfr{q}$ is unique. Similarly, decompose (recall that $n^{\ast}:= n/ \gcd(n,2)$)
\begin{align}\label{eq decomp lambds BCD}
    \lambda_{A}^{(n^{\ast})}(\mfr{q})=\lambda_{X',0}+ \lambda_{X',\half{1}}+ \sum_{i=1}^{l_{X'}} \lambda_{X',i}
\end{align}
as multi-sets, where $\lambda_{X',0}$ (resp. $\lambda_{X',\half{1}}$) is the (possibly empty) multi-set consisting of elements of $\lambda$ in $\Z$ (resp. $\half{1}+\Z$), and $\{\lambda_{X',i}\}_{i=1}^{l_{X'}}$ are the other maximal sub-multi-set of $\lambda$ such that any two elements have integral difference. Note that by condition (ii) of Lemma \ref{lem partition of wt}, $l_{X'}$ must be even, and we may relabel then such that $\lambda_{X',l_{X'}+1-i}=\{-x \ | \ x \in \lambda_{X',i}\}$  for any $1 \lest i \lest l_X$. In particular, $ \mfr{p}(\lambda_{X',i})= \mfr{p}(\lambda_{X',l_{X'}+1-i})$. Finally, for $i \in \{0, \half{1}, 1,\ldots, l_{X'}\}$, define
\begin{align}\label{eq partition lambda sep}
    \mfr{p}_{X',i}(\lambda):= \begin{cases} \mfr{p}( \lambda_{X',i} )^{\ast} \sqcup [1] &\text{ if }i=0 \text{ and }\lambda_{X',0} \text{ is of even length},\\
\mfr{p}( \lambda_{X',i} )^{\ast} & \text{ otherwise.}
\end{cases}
\end{align}
Note that $\mfr{p}_{X',0}(\lambda)$ (resp. $\mfr{p}_{X',\half{1}}(\lambda)$) consists of odd (resp. even) integers; hence it is of type $B$ (resp. $C$). Also, we have the following equality of multi-set
\begin{align*}
    \lambda_{X'}^{(n)}(\mfr{q})= \lambda_{B}^{(1)}(\mfr{p}_{X',0}(\lambda))+\lambda_{C}^{(1)}(\mfr{p}_{X',\half{1}}(\lambda)) + \sum_{i=1}^{\half{l_{X'}}} \{|x|: \ x \in \lambda_{X',i}\}.
\end{align*}

Now we rephrase the main results of \cite{BMW25}*{\S 1.1--\S 1.4} in our special case.

\begin{thm}\label{thm BMW}
    Let $X,X' \in \{B,C,D\}$ or $X=A=X'$ and let $n \in \Z_{>0}$. Suppose $\mfr{q}$ is a partition of type $X'$ and denote $ \lambda:= \lambda_{X'}^{(n)}(\mfr{q})$.  Using decompositions \eqref{eq decomp lambda A} and \eqref{eq decomp lambds BCD} and the notation $\mfr{p}_{X,i}(\lambda)$ defined above, the partition $\AP{X}(\lambda)$ is given case by case as follows.
    \begin{align*}
        \AP{X}(\lambda)= \begin{cases}
            \sum_{i=1}^l \mfr{p}_{A,i}(\lambda)^{\ast}& \text{ if }X=A,\\
           \left( \mfr{p}_{X',0}(\lambda)^\ast   +\mfr{p}_{X',\half{1}}(\lambda)^{*} + \sum_{i=1}^{l_{X'}} \mfr{p}_{X',i}(\lambda)^{\ast}  \right)_{B} & \text{ if }X=B,\\
            \left( \mfr{p}_{X',0}(\lambda)^- \sub{C} {}^\ast +\mfr{p}_{X',\half{1}}(\lambda)^{\ast}\sub{D} {}^{+-} \sub{C}+ \sum_{i=1}^{l_{X'}} \mfr{p}_{X',i}(\lambda)^{\ast}  \right)_{C} & \text{ if }X=C,\\
            \left( \mfr{p}_{X',0}(\lambda)^{-\ast} \sub{D}   +\mfr{p}_{X',\half{1}}(\lambda) \sub{D} {}^{\ast}+ \sum_{i=1}^{l_{X'}} \mfr{p}_{X',i}(\lambda)^{\ast}  \right)_{D} & \text{ if }X=D.
        \end{cases}
    \end{align*}
\end{thm}
\begin{proof}
The case $X=A$ is exactly \cite{BMW25}*{Theorem 1.1}. For the other cases, \cite{BMW25}*{Theorem 1.6} implies that 
\[ \AP{X}(\lambda)= \left(\mfr{q}_0+ \mfr{q}_{\half{1}}+ \sum_{i=1}^{l_{X'}} \mfr{p}_{X',i}(\lambda)^{\ast} \right)_{X},\]
where $\mfr{q}_0$ (resp. $\mfr{q}_{\half{1}}$) is a partition determined by $\lambda_{X',0}$ (resp. $\lambda_{X',\half{1}}$) defined in \eqref{eq decomp lambds BCD}, or equivalently by $\mfr{p}_{X',0}(\lambda)$ (resp. $\mfr{p}_{X',\half{1}}(\lambda)$). Note that $\mfr{q}_0$ (resp. $\mfr{q}_{\half{1}}$) is special or metaplectic special of a certain type determined by $X$.
We demonstrate how to find the formula for $\mfr{q}_0$ (resp. $\mfr{q}_{\half{1}}$) in terms of $\mfr{p}_{X',0}(\lambda)$ (resp. $\mfr{p}_{X',\half{1}}(\lambda)$) case by case.

When $X=B$, we consider another partition $\mfr{q}'= \mfr{p}_{X',0}(\lambda)$ and set $\lambda':= \lambda_{B}^{(2)}(\mfr{q}')$. Note that every part of $\mfr{q}'$ is odd. It follows from the definition and Lemma \ref{lem partition of wt} that $ \mfr{p}_{B,0}(\lambda') =\mfr{p}_{X',0}(\lambda)$, and $\mfr{p}_{B,i}(\lambda')$ is empty for $i \neq 0$. Therefore, we have
\[ (\mfr{q}_0)_{B}= \AP{B}(\lambda')=\AP{B}(\lambda^{(2)}_B(\mfr{q}' ))=\AP{B}(\lambda^{(2)}_B(\mfr{p}_{X',0}(\lambda)))= \mfr{p}_{X',0}(\lambda)^{\ast}. \]
Here the last equality follows from Proposition \ref{prop ann n=1}(a1). Then since $\mfr{q}_0$ is type $B$-special as defined in \cite{BMW25}, we see that $\mfr{q}_0=(\mfr{q}_0)_B= \mfr{p}_{X',0}(\lambda)^{\ast}.$ 

For $\mfr{q}_{\half{1}}$, we compute it by the definition, which we recall now. From $\lambda_{X',\half{1}}$, we compute the partition $\mfr{p}(\lambda_{X',\half{1}})$, and then associate a $D$-symbol $\Lambda_D$. By permuting the entries of $\Lambda_D$, we obtain a unique special $D$-symbol $\Lambda_D^s$. Finally, $\mfr{q}_{\half{1}}$ is the type $D$-special partition corresponding to $\Lambda^s_{D}$. See \cite{BMW25}*{\S 4} for details of the computation. In our situation, note that $\mfr{p}(\lambda_{X',\half{1}})=\mfr{p}_{X',\half{1}}(\lambda)^{\ast}$ is already a type $D$-special partition since $\mfr{p}_{X',\half{1}}(\lambda)$ is already of type $C$. Thus, we have $\Lambda_D=\Lambda^s_D$ already, and $\mfr{q}_{\half{1}}= \mfr{p}(\lambda_{X',\half{1}})=\mfr{p}_{X',\half{1}}(\lambda)^{\ast}.$

When $X=C$, similarly, by taking $\mfr{q}'=\mfr{p}_{X',0}(\lambda)$, $\mfr{q}''=\mfr{p}_{X',\half{1}}(\lambda)$, and set $\lambda'=\lambda_{B}^{(1)}(\mfr{q}')$, $\lambda''=\lambda_{C}^{(2)}(\mfr{q}'')$, we obtain that 
\begin{align*}
    (\mfr{q}_0)_C&= \AP{C}( \lambda_{B}^{(1)}(\mfr{q}'))=  \mfr{p}_{X',0}(\lambda)^- \sub{C} {}^\ast,\\(\mfr{q}_{\half{1}})_C&= \AP{C}( \lambda_{C}^{(2)}(\mfr{q}''))=\mfr{p}_{X',\half{1}}(\lambda)^{\ast}\sub{D} {}^{+-} \sub{C},
\end{align*}
   where we use Proposition \ref{prop ann n=1}(a2) and (b2). This gives the formula for $\mfr{q}_0$ and $\mfr{q}_{\half{1}}$ since they are type $C$ special and metaplectic special as defined in \cite{BMW25}.

   When $X=D$, we take $\mfr{q}'=\mfr{p}_{X',0}(\lambda)$, $\mfr{q}''=\mfr{p}_{X',\half{1}}(\lambda)$, and set $\lambda'=\lambda_{D}^{(1)}(\mfr{q}')$, $\lambda''=\lambda_{C}^{(2)}(\mfr{q}'')$. Then 
\begin{align*}
    (\mfr{q}_0)_D&= \AP{D}( \lambda_{D}^{(1)}(\mfr{q}'))=  \mfr{p}_{X',0}(\lambda)^{-\ast} \sub{C},\\(\mfr{q}_{\half{1}})_D&= \AP{D}( \lambda_{C}^{(2)}(\mfr{q}''))=\mfr{p}_{X',\half{1}}(\lambda)\sub{D} {}^{\ast} \sub{D},
\end{align*}
   where we use Proposition \ref{prop ann n=1}(c1) and (b3). This gives the formula for $\mfr{q}_0$ since it is already of type $D$. For $\mfr{q}_{\half{1}},$ since it is metaplectic special, there is a unique type $D$-special partition $\mfr{p}$ such that $\mfr{q}_{\half{1}}= \mfr{p}^{\ast}$. Then by $\mfr{p}=\mfr{p}^{\ast}\sub{D}{}^{\ast} \sub{D}$,  we have
   \[ \mfr{q}_{\half{1}}= \mfr{p}^{\ast}= \mfr{p}^{\ast}\sub{D}{}^{\ast} \sub{D}{}^{\ast}=(\mfr{q}_{\half{1}})\sub{D}{}^{\ast} \sub{D}{}^{\ast}=\mfr{p}_{X',\half{1}}(\lambda)\sub{D} {}^{\ast} \sub{D}{}^{\ast} \sub{D}{}^{\ast} =\mfr{p}_{X',\half{1}}(\lambda)\sub{D} {}^{\ast}.\]
   Here the last equality follows from the fact that (see \cite{Ach03}*{Lemma 3.3})
   \[ \mfr{p}_{X',\half{1}}(\lambda)\sub{D}{}^{\ast}= \mfr{p}_{X',\half{1}}(\lambda)^{+-}\sub{C},  \]
   which implies that $\mfr{p}_{X',\half{1}}(\lambda)\sub{D}$ is already type $D$-special. Hence, $\mfr{p}_{X',\half{1}}(\lambda)\sub{D} {}^{\ast} \sub{D}{}^{\ast} \sub{D} = \mfr{p}_{X',\half{1}}(\lambda)\sub{D}.$
   This completes the proof of the theorem.
\end{proof}

\begin{remark}
    The same strategy for computing $\mfr{q}_{\half{1}}$ for $X=C,D$ does not work when $X=B$. Indeed, the argument gives $(\mfr{q}_{\half{1}})^{+}\sub{B}= \mfr{p}_{X',\half{1}} (\lambda)^{\ast+}\sub{B}$, where $\mfr{q}_{\half{1}}$ is type $D$-special. However, this does not uniquely determine the partition $\mfr{q}_{\half{1}}$. For example, $[3,3,2,2]$ and $[3,3,3,1]$ are both type $D$-special and 
    \[ [3,3,2,2]^{+} \sub{B}=[3,3,3,1]^{+} \sub{B}= [4,4,2]^{\ast+}\sub{B}.\]
\end{remark}

\subsection{\texorpdfstring{Some identities for $d_{com,X}^{(n)}$ and $\AP{A}(\lambda_{A}^{(n)}(\mfr{p}))$}{}}

In this subsection, we establish several identities involving $d_{com,X}^{(n)}$ and $\AP{A}(\lambda_{A}^{(n)}(\mfr{p}))$. First, we show the compatibility of $d_{com,X}^{(n)}(\mfr{p})$ and $\AP{A}(\lambda_{A}^{(n)}(\mfr{p}))$ with induction of nilpotent orbit.

\begin{lemma}\label{lem ind A} 
Let $n \in \Z_{>0}$. For any pair of partitions $(\mfr{p},\mfr{q})$, the following holds.
\begin{enumerate}
    \item [(a)] $ d_{com,A}^{(n)}( \mathfrak{p}\sqcup \mfr{q})=d_{com,A}^{(n)}( \mathfrak{p})+d_{com,A}^{(n)}( \mathfrak{q}).$
    \item [(b)] $ \AP{A}( \lambda_A^{(n)}( \mathfrak{p}\sqcup \mfr{q}))=\AP{A}( \lambda_A^{(n)}( \mathfrak{p}))+ \AP{A}( \lambda_A^{(n)}( \mfr{q}))$.
\end{enumerate}
\end{lemma}
\begin{proof}
    Part (a) follows directly from the definition. For Part (b), let $\lambda_1:= \lambda_{A}^{(n)}(\mfr{p})$ and $\lambda_2:= \lambda_{A}^{(n)}(\mfr{q})$ as multi-sets of real numbers. We have $\lambda_{A}^{(n)}(\mfr{p}\sqcup \mfr{q})= \lambda_1+\lambda_2$ by definition. Decompose $\lambda$ and $\lambda_j$ $(j=1,2)$ as in \eqref{eq decomp lambda A}:
    \[ \lambda= \sum_{i=1}^{l} \lambda_{A,i},\ \lambda_j= \sum_{i=1}^{l} (\lambda_j)_{A,i}, \]
    where for any $i=1,\ldots, l$, any two elements of the multi-set $ \lambda_{A,i}+(\lambda_{1})_{A,i}+(\lambda_{2})_{A,i}$ have an integral difference, and $(\lambda_j)_{A,i}$ is possibly empty. Note that $\lambda_{A,i} = (\lambda_{1})_{A,i}+(\lambda_{2})_{A,i}$. Also, if we write $(\lambda_j)_{A,i}=\{ \lambda_k^{r(\lambda_k)}\}_{k=1}^s$, where $r(\lambda_k)$ indicates the multiplicity of $\lambda_k$, then Lemma \ref{lem partition of wt} implies that $ r(\lambda_{k_1})\gest r(\lambda_{k_2})$ if $|\lambda_{k_1}|< |\lambda_{k_2}|$.    Therefore, by Lemma \ref{lem RS}, for any $i$, we have
    \[\mfr{p}( \lambda_{A,i} )=  \mfr{p}( (\lambda_{1})_{A,i}+(\lambda_{2})_{A,i} )= \mfr{p}( (\lambda_{1})_{A,i} )+\mfr{p}( (\lambda_{2})_{A,i} ). \]
    Thus, the desired conclusion follows from the formula of $\AP{A}(\lambda)$ given in Theorem \ref{thm BMW}. This completes the proof of the lemma.
\end{proof}

As a corollary, we prove Theorem \ref{T:comparison} for type $A$.

\begin{prop}\label{prop comparison A}
    Theorem \ref{T:comparison} holds when $X=A$, i.e., $d_{com,A}^{(n)}(\mfr{p})= \AP{A}(\lambda_{A}^{(n)}(\mfr{p}))$.
\end{prop}
\begin{proof}
     According to Lemma \ref{lem ind A}, it suffices to show the equality for $\mfr{p}=[d]$, which follows from the definition and Lemma \ref{lem RS}.
\end{proof}

Next, we generalize \cite{Ach03}*{Lemma 3.3}. The idea of our proof is similar to Achar's proof, while the extra complexity arises from the fact that the equality $d_{com,A}^{(n)}(\mfr{p}+\mfr{q})= d_{com,A}^{(n)}(\mfr{p}) \sqcup d_{com,A}^{(n)}(\mfr{q})$ fails in general, unless $\mfr{p},\mfr{q}$ is chosen properly.

\begin{lemma}\label{lem Achar identities} The following identities hold for any positive integer $n$.
    \begin{enumerate}
         \item [(i)] $ d_{com,A}^{(n)}( \mfr{p}^- \sub{C})\sub{C}= d_{com,A}^{(n)}( \mfr{p})^- \sub{C}$ if $\mfr{p}$ is of type $B$.
         \item [(ii)] $ d_{com,A}^{(n)}( \mfr{p} \sub{D})\sub{C}= d_{com,A}^{(n)}( \mfr{p})^{+-} \sub{C}$ if $\mfr{p}$ is of type $C$.
          \item [(iii)] $ d_{com,A}^{(n)}(\mfr{p}^{-}\sub{B})^+\sub{C}=d_{com,A}^{(n)}(\mfr{p})^{+-}\sub{C}$ if $\mfr{p}$ is of type $C$ and both $l(\mfr{p})$ and $n$ are odd. 
    
           \item [(iv)] $ d_{com,A}^{(n)}( \mfr{p}^+ \sub{B})\sub{B}= d_{com,A}^{(n)}( \mfr{p})^+ \sub{B}$ if $\mfr{p}$ is of type $C$ and $n$ is odd.
        \item [(v)] $ d_{com,A}^{(n)}(\mfr{p}^{+-}\sub{D})^+\sub{B}=d_{com,A}^{(n)}(\mfr{p})^+\sub{B}$ if $\mfr{p}$ is of type $C$, $l(\mfr{p})$ is even and $n$ is odd.
        \item [(vi)] $ d_{com,A}^{(n)}(\mfr{p}^{+-}\sub{C})\sub{D}=d_{com,A}^{(n)}(\mfr{p})\sub{D}$ if $\mfr{p}$ is of type $D$ and $n$ is odd.
\end{enumerate}
\end{lemma}
\begin{proof}
Write $\mfr{p}=[p_1,\ldots, p_s]$ and let $p_i=a_i n +b_i$ where $1 \lest b_i \lest n$. We establish these identities simultaneously by applying induction on $A(\mfr{p}):=\max\{ a_i\}_{i=1}^s$. If $A(\mfr{p})=0$, then all of the partitions on both sides of the formulas have length one, in which case the equalities trivially hold.

Suppose $A(\mfr{p})>0$. Consider the decomposition ($\mfr{p}_1$ is possibly empty)
\[ \mfr{p}=  ( [n^l]\sqcup \mfr{p}_1 )+\mfr{p}_2=( [n^l]+ \mfr{p}_2 )\sqcup \mfr{p}_1,\]
where $l=l(\mfr{p}_2)$ and every part of $\mfr{p}_1$ is less than $n+1$. Then we have $A(\mfr{p}_2)=A(\mfr{p})-1$ and $A(\mfr{p}_1)=0$. Set $b:= |\mfr{p}_1|$ for convenience. Under this decomposition, we have
\[ d_{com,A}^{(n)}( ( [n^l]\sqcup \mfr{p}_1 )+\mfr{p}_2 )= d_{com,A}^{(n)}(  [n^l]\sqcup \mfr{p}_1 ) \sqcup d_{com,A}^{(n)}(\mfr{p}_2)= [nl+b] \sqcup d_{com,A}^{(n)}(\mfr{p}_2).\]

In the following, we demonstrate that Part (i) for $\mfr{p}$ can be reduced to Parts (i)--(iii) for $\mfr{p}_2$ or some related $\mfr{p}_2'$ with $A(\mfr{p}_2')\lest A(\mfr{p}_2)$. The computation of collapse makes use of \cite{HLLS24}*{Lemma 12.2} frequently and we do not mention it every time.

We first deal with the case that $\mfr{p}_1$ is empty. Suppose $n$ is even. Then $\mfr{p}_2$ is of type $B$. By Part (i) of $\mfr{p}_2$, we have 
\begin{align*}
    d_{com,A}^{(n)}( ([n^l]+\mfr{p}_2)^{-}\sub{C} )\sub{C}&=   d_{com,A}^{(n)}( [n^l]+\mfr{p}_2{}^-\sub{C} )\sub{C} \\
    &= ([nl] \sqcup d_{com,A}^{(n)}(\mfr{p}_2{}^- \sub{C})) \sub{C}\\
    &= [nl] \sqcup d_{com,A}^{(n)}(\mfr{p}_2{}^- \sub{C}) \sub{C}\\
    &=[nl] \sqcup d_{com,A}^{(n)}(\mfr{p}_2)^- \sub{C}\\
    &=  d_{com,A}^{(n)}(\mfr{p})^{-} \sub{C}.
\end{align*}
Suppose $n$ is odd. If the smallest part of $\mfr{p}_2$ is not less than $n$, then we let $\mfr{p}_2'$ be the partition such that $\mfr{p}= [(2n)^l]\sqcup \mfr{p}_2'$. It follows from the argument for the previous case that Part (i) for $\mfr{p}$ can be reduced to Part (i) for $\mfr{p}_2'$. Therefore, we assume the smallest part of $\mfr{p}_2$ is less than $n$, and hence the largest part of $d_{com,A}^{(n)}(\mfr{p}_2{}^- \sub{B})$ is less than $nl$. Note that $\mfr{p}_2$ is of type $C$ and $l(\mfr{p}_2)= l= l(\mfr{p})$ is odd since $\mfr{p}$ is of type $B$. By Part (iii) for $\mfr{p}_2$, we have 
\begin{align*}
    d_{com,A}^{(n)}( ([n^l]+\mfr{p}_2)^{-}\sub{C} )\sub{C}&=   d_{com,A}^{(n)}( [n^l]+\mfr{p}_2{}^-\sub{B} )\sub{C} \\
    &= ([nl] \sqcup d_{com,A}^{(n)}(\mfr{p}_2{}^- \sub{B})) \sub{C}\\
    &= [nl-1] \sqcup d_{com,A}^{(n)}(\mfr{p}_2{}^- \sub{B})^{+} \sub{C}\\
    &=[nl-1] \sqcup d_{com,A}^{(n)}(\mfr{p}_2)^{+-} \sub{C}\\
     &=([nl] \sqcup d_{com,A}^{(n)}(\mfr{p}_2)^{-} \sub{C}) \sub{C}\\
    &=  d_{com,A}^{(n)}(\mfr{p})^{-} \sub{C}.
\end{align*}
This completes the verification of the case that $\mfr{p}_1$ is empty. In the rest of the proof, we assume $\mfr{p}_1$ is non-empty.

Suppose that $n$ is even and $b$ is odd. In this case, $ \mfr{p}_1$ (resp. $\mfr{p}_2$) is of type $B$ (resp. type $D$). Let $\mfr{p}_2':=\mfr{p}_2\sqcup [1]$. Note that $ d_{com,A}^{(n)}(\mfr{p}_2')= d_{com,A}^{(n)}(\mfr{p}_2)^+$ and $A(\mfr{p}_2)=A(\mfr{p}_2')$. By Part (i) for $\mfr{p}_2'$, we have
\begin{align*}
    d_{com,A}^{(n)}( (([n^l]\sqcup \mfr{p}_1) +\mfr{p}_2)^- \sub{C}   )\sub{C}&= d_{com,A}^{(n)}( ([n^l]\sqcup \mfr{p}_1{}^-)\sub{C} +(\mfr{p}_2) \sub{C}   )\sub{C}\\
   &= (d_{com,A}^{(n)}( [n^l]\sqcup \mfr{p}_1{}^-\sub{C} ) \sqcup d_{com,A}^{(n)}( (\mfr{p}_2) \sub{C}   ) )\sub{C}\\
   &= ([nl+b-1] \sqcup d_{com,A}^{(n)}( (\mfr{p}_2')^{-} \sub{C}   ) )\sub{C}\\
   &=[nl+b-1] \sqcup d_{BV,A}^{(n)}( (\mfr{p}_2')^- \sub{C}   )\sub{C}\\
   &= d_{com,A}^{(n)}( [n^l]\sqcup \mfr{p}_1 )^- \sub{C} \sqcup d_{com,A}^{(n)}(\mfr{p}_2')^{-} \sub{C}\\
   &= d_{com,A}^{(n)}( [n^l]\sqcup \mfr{p}_1 )^- \sub{C} \sqcup d_{com,A}^{(n)}(\mfr{p}_2)^{-+} \sub{C}\\
   &= (d_{com,A}^{(n)}( [n^l]\sqcup \mfr{p}_1 ) \sqcup d_{com,A}^{(n)}(\mfr{p}_2)^{-} )\sub{C}\\
   &= (d_{com,A}^{(n)}( [n^l]\sqcup \mfr{p}_1 ) \sqcup d_{com,A}^{(n)}(\mfr{p}_2) )^-\sub{C}\\
   &= d_{com,A}^{(n)}(\mfr{p})^{-}\sub{C}.
\end{align*}
Next, suppose both $n$ and $b$ are even. In this case, $ \mfr{p}_1$ (resp. $\mfr{p}_2$) is of type $D$ (resp. type $B$). By Part (i) for $\mfr{p}_2$, we have
\begin{align*}
     d_{com,A}^{(n)}( (([n^l]\sqcup \mfr{p}_1) +\mfr{p}_2)^- \sub{C}   )\sub{C}&= d_{com,A}^{(n)}( ([n+1, n^{l-1}]\sqcup \mfr{p}_1)^{-}\sub{C} +\mfr{p}_2{}^{-} \sub{C}   )\sub{C}\\
     &=d_{com,A}^{(n)}( ([n^{l}]\sqcup \mfr{p}_1{}^{+-}\sub{C}) +\mfr{p}_2{}^{-} \sub{C}   )\sub{C}\\
     &= (d_{com,A}^{(n)}( ([n^{l}]\sqcup \mfr{p}_1{}^{+-}\sub{C}))  \sqcup d_{com,A}^{(n)}(\mfr{p}_2{}^{-} \sub{C}   ) )\sub{C}\\
     &= [nl+b] \sqcup d_{com,A}^{(n)}(\mfr{p}_2{}^{-} \sub{C}   )\sub{C}\\
     &= d_{com,A}^{(n)}([n^l]\sqcup \mfr{p}_1)\sub{C} \sqcup d_{com,A}^{(n)}(\mfr{p}_2)^- \sub{C}\\
      &= (d_{com,A}^{(n)}( [n^l]\sqcup \mfr{p}_1 ) \sqcup d_{com,A}^{(n)}(\mfr{p}_2) )^-\sub{C}\\
   &= d_{com,A}^{(n)}(\mfr{p})^{-}\sub{C}.
\end{align*}
Next, suppose $n$ is odd and $l$ is even. In this case, $\mfr{p}_1$ (resp. $\mfr{p}_2$) is of type $B$ (resp. type $C$). By Part (ii) for $\mfr{p}_2$, we have
\begin{align*}
     d_{com,A}^{(n)}( (([n^l]\sqcup \mfr{p}_1) +\mfr{p}_2)^- \sub{C}   )\sub{C}&= d_{com,A}^{(n)}( ([n^l]\sqcup \mfr{p}_1{}^-)\sub{C} +(\mfr{p}_2)\sub{D}   )\sub{C}\\
     &=(d_{com,A}^{(n)}( [n^l]\sqcup \mfr{p}_1{}^-\sub{C}) \sqcup d_{com,A}^{(n)}((\mfr{p}_2)\sub{D})   )\sub{C}\\
     &= [nl+b-1] \sqcup d_{com,A}^{(n)}((\mfr{p}_1)\sub{D})\sub{C}\\
     &=d_{com,A}^{(n)}( [n^l]\sqcup \mfr{p}_1)^{-}\sub{C} \sqcup d_{com,A}^{(n)}(\mfr{p}_2)^{+-}\sub{C}\\
     &=( d_{com,A}^{(n)}( [n^l]\sqcup \mfr{p}_1) \sqcup d_{com,A}^{(n)}(\mfr{p}_2)^{-})\sub{C}\\
      &= d_{com,A}^{(n)}(\mfr{p})^{-}\sub{C}.
\end{align*}
Finally, suppose both $n$ and $l$ are odd.  In this case, $\mfr{p}_1$ (resp. $\mfr{p}_2$) is of type $D$ (resp. type $C$). Note that $(\mfr{p}_2)\sub{D}=(\mfr{p}_{2}{}^- \sub{B}) \sqcup [1]$. If $\mfr{p}_1=[n] \sqcup \mfr{p}_1'$, then
\[(([n^l]\sqcup \mfr{p}_1) +\mfr{p}_2)^- \sub{C}= ( (\mfr{p}_2)\sub{D}+ [n^{l+1}] ) \sqcup (\mfr{p}_1')^-\sub{C}=( [n^{l+1}]  \sqcup (\mfr{p}_1')^-\sub{C}) + (\mfr{p}_2)\sub{D}, \]
and similar argument in the previous case works. If the largest part of $\mfr{p}_1$ is less than $n$, then
\[(([n^l]\sqcup \mfr{p}_1) +\mfr{p}_2)^- \sub{C} = (([n^l]+\mfr{p}_2) \sqcup \mfr{p}_1{}^-)\sub{C} = ([n^l]+\mfr{p}_2)^{-}\sub{C} \sqcup \mfr{p}_1{}^{+-}\sub{C} =([n^l]\sqcup \mfr{p}_1{}^{+-}\sub{C})+ \mfr{p}_2{}^- \sub{B}. \]
Also, observe that since $l$ is odd, we have $(\mfr{p}_2)\sub{D}= \mfr{p}_{2}{}^- \sub{B} \sqcup [1]$, and hence  $d_{com,A}^{(n)}((\mfr{p}_2)\sub{D})=  d_{com,A}^{(n)}(\mfr{p}_2{}^-\sub{B})^+.$
Thus, by Part (ii) for $\mfr{p}_2$, we have (the non-emptiness of $\mfr{p}_1$ is used here)
\begin{align*}
     d_{com,A}^{(n)}( (([n^l]\sqcup \mfr{p}_1) +\mfr{p}_2)^- \sub{C}   )\sub{C}&= d_{com,A}^{(n)}( ([n^l]\sqcup \mfr{p}_1{}^{+-}\sub{C}) +\mfr{p}_2{}^-\sub{B}   )\sub{C}\\
     &= (  d_{com,A}^{(n)}([n^l]\sqcup \mfr{p}_1{}^{+-}\sub{C}) \sqcup   d_{com,A}^{(n)}(\mfr{p}_2{}^-\sub{B})  )\sub{C}\\
     &=   [nl+b-1] \sqcup   d_{com,A}^{(n)}(\mfr{p}_2{}^-\sub{B})^+\sub{C}\\
     &=   [nl+b-1] \sqcup  d_{com,A}^{(n)}((\mfr{p}_2)\sub{D})\sub{C}\\
      &= d_{com,A}^{(n)}(\mfr{p})^{-}\sub{C}.
\end{align*}
This completes the reduction process for Part (i).

By similar argument, one can show that Parts (ii) and (iii) (resp. Parts (iv)--(vi)) for $\mfr{p}$ can be reduced to Parts (i)--(iii) (resp. Parts (iv)--(vi)) for $\mfr{p}_2$ or some related $\mfr{p}_2'$ such that $A(\mfr{p}_2)=A(\mfr{p}_2')$. Below we give details for each case.

\textbf{Part (ii):} First, we deal with the case that $\mfr{p}_1$ is empty. If $n$ is even, then $\mfr{p}_2$ is of type $C$. By Part (ii) for $\mfr{p}_2$, we have
\begin{align*}
    d_{com,A}^{(n)}( ([n^l]+ \mfr{p}_2)\sub{D} ) \sub{C}&=d_{com,A}^{(n)}( [n^l] + (\mfr{p}_2)\sub{D} ) \sub{C}\\
    &= [nl] \sqcup d_{com,A}^{(n)}((\mfr{p}_2)\sub{D}) )\sub{C}\\
    &= [nl] \sqcup d_{com,A}^{(n)}(\mfr{p}_2)^{+-} \sub{C}\\
    &= ( [nl+1] \sqcup d_{com,A}^{(n)}(\mfr{p}_2)^{-}) \sub{C}\\
    &= d_{com,A}^{(n)}(\mfr{p})^{+-}\sub{C}.
\end{align*}
If $n$ and $l$ are both odd, then $\mfr{p}_2$ is of type $B$. By Part (i) for $\mfr{p}_2$, we have
\begin{align*}
    d_{com,A}^{(n)}( ([n^l]+ \mfr{p}_2)\sub{D} ) \sub{C}&=d_{com,A}^{(n)}( ([n^l]\sqcup [1]) + \mfr{p}_2{}^-\sub{C} ) \sub{C}\\
    &= [nl+1] \sqcup d_{com,A}^{(n)}(\mfr{p}_2{}^-\sub{C}) )\sub{C}\\
    &= [nl+1] \sqcup d_{com,A}^{(n)}(\mfr{p}_2)^{-} \sub{C}\\
    &= ( [nl+1] \sqcup d_{com,A}^{(n)}(\mfr{p}_2)^{-}) \sub{C}\\
    &= d_{com,A}^{(n)}(\mfr{p})^{+-}\sub{C}.
\end{align*}
If $n$ is odd and $l$ is even, then $\mfr{p}_2$ is of type $D$. By Part (i) for $\mfr{p}_2 \sqcup [1]$, we have
\begin{align*}
    d_{com,A}^{(n)}( ([n^l]+ \mfr{p}_2)\sub{D} ) \sub{C}&=d_{com,A}^{(n)}( [n^l] + (\mfr{p}_2)\sub{C} ) \sub{C}\\
    &= [nl] \sqcup d_{com,A}^{(n)}( (\mfr{p}_2 \sqcup [1])^-\sub{C})\sub{C}\\
    &= [nl] \sqcup d_{com,A}^{(n)}(\mfr{p}_2\sqcup [1])^{-} \sub{C}\\
    &= ( [nl+1] \sqcup d_{com,A}^{(n)}(\mfr{p}_2)^{-}) \sub{C}\\
    &= d_{com,A}^{(n)}(\mfr{p})^{+-}\sub{C}.
\end{align*}
This completes the verification of this case. In the following, we assume $\mfr{p}_1$ is non-empty.

Suppose that $n$ is even. In this case, $\mfr{p}_1$ and $\mfr{p}_2$ are both of type $C$. By Part (ii) for $\mfr{p}_2$, we have
\begin{align*}
     d_{com,A}^{(n)}( (([n^l]\sqcup \mfr{p}_1) + \mfr{p}_2)\sub{D} )\sub{C}&= d_{com,A}^{(n)} (([n^l]\sqcup \mfr{p}_1) \sub{D} + (\mfr{p}_2 )\sub{D} )\sub{C}\\
     &= ([nl+b] \sqcup d_{com,A}^{(n)}((\mfr{p}_2)\sub{D}))\sub{C}\\
     &=[nl+b] \sqcup d_{com,A}^{(n)}((\mfr{p}_2)\sub{D})\sub{C}\\
     &=[nl+b] \sqcup d_{com,A}^{(n)}(\mfr{p}_2)^{+-}\sub{C}\\
     &= ( [nl+b+1] \sqcup d_{com,A}^{(n)}(\mfr{p}_2)^{-})\sub{C}\\
     &= (d_{com,A}^{(n)}([n^l]\sqcup \mfr{p}_1) \sqcup d_{com,A}^{(n)}(\mfr{p}_2))^{+-}\sub{C} \\
     &= d_{com,A}^{(n)}(\mfr{p})^{+-}\sub{C}.
\end{align*}
This completes the verification of this case. Next, suppose that both $n$ and $l$ are odd. In this case, $\mfr{p}_1$ (resp. $\mfr{p}_2$) is of type $C$ (resp. type $B$). By Part (i) for $\mfr{p}_2$,we have 
\begin{align*}
     d_{com,A}^{(n)}( (([n^l]\sqcup \mfr{p}_1) + \mfr{p}_2)\sub{D} )\sub{C}&= d_{com,A}^{(n)} (([n^l]\sqcup \mfr{p}_1)^+ \sub{D} + \mfr{p}_2 {}^-\sub{C} )\sub{C}\\
     &= ([nl+b+1] \sqcup d_{com,A}^{(n)}(\mfr{p}_2{}^{-}\sub{C}))\sub{C}\\
     &=[nl+b+1] \sqcup d_{com,A}^{(n)}(\mfr{p}_2{}^{-}\sub{C})\sub{C}\\
     &=[nl+b+1] \sqcup d_{com,A}^{(n)}(\mfr{p}_2)^{-}\sub{C}\\
     &= ( [nl+b+1] \sqcup d_{com,A}^{(n)}(\mfr{p}_2)^{-})\sub{C}\\
     &= d_{com,A}^{(n)}(\mfr{p})^{+-}\sub{C}.
\end{align*}
This completes the verification of this case. Finally, suppose that $n$ is odd and $l$ is even. In this case, $\mfr{p}_1$ (resp. $\mfr{p}_2$) is of type $C$ (resp. type $D$). Let $\mfr{p}_2':= \mfr{p}_2 \sqcup [1]$, which is of type $B$. By Part (i) for $\mfr{p}_2'$, we have
\begin{align*}
     d_{com,A}^{(n)}( (([n^l]\sqcup \mfr{p}_1) + \mfr{p}_2)\sub{D} )\sub{C}&= d_{com,A}^{(n)} (([n^l]\sqcup \mfr{p}_1)\sub{D} + (\mfr{p}_2)\sub{C} )\sub{C}\\
     &= ([nl+b] \sqcup d_{com,A}^{(n)}((\mfr{p}_2')^-\sub{C}))\sub{C}\\
     &=[nl+b] \sqcup d_{com,A}^{(n)}((\mfr{p}_2')^-\sub{C})\sub{C}\\
     &=[nl+b] \sqcup d_{com,A}^{(n)}(\mfr{p}_2')^{-}\sub{C}\\
     &=[nl+b] \sqcup d_{com,A}^{(n)}(\mfr{p}_2)^{+-}\sub{C}\\
     &= ( [nl+b+1] \sqcup d_{com,A}^{(n)}(\mfr{p}_2)^{-})\sub{C}\\
     &= d_{com,A}^{(n)}(\mfr{p})^{+-}\sub{C}.
\end{align*}
This completes the verification of this case and the reduction for Part (ii).

\textbf{Part (iii):} First, we deal with the case that $\mfr{p}_1$ is empty. Note that both $n$ and $l=l(\mfr{p}_2)= l(\mfr{p})$ are odd. In this case, $\mfr{p}_2$ is of type $B$. By Part (i) for $\mfr{p}_2$, we have
\begin{align*}
    d_{com,A}^{(n)}( ([n^l]+ \mfr{p}_2)^{-}\sub{B}  )^+ \sub{C}&=  d_{com,A}^{(n)}( ([n^l]+ \mfr{p}_2{}^{-}\sub{C}  )^+ \sub{C}\\
    &= ([nl+1] \sqcup d_{com,A}^{(n)}(\mfr{p}_2{}^{-}\sub{C} )) \sub{C}\\
    &=[nl+1] \sqcup d_{com,A}^{(n)}(\mfr{p}_2{}^{-}\sub{C} ) \sub{C}\\
    &= [nl+1] \sqcup d_{com,A}^{(n)}(\mfr{p}_2 )^{-}\sub{C}\\
    &= ([nl+1] \sqcup d_{com,A}^{(n)}(\mfr{p}_2 ))^{-}\sub{C}\\
    &=  d_{com,A}^{(n)}(\mfr{p})^{+-}\sub{C}.
\end{align*}
This completes the verification of this case. In the following, we assume $\mfr{p}_1$ is non-empty.

Suppose $l$ is even. In this case, $\mfr{p}_1$ (resp. $\mfr{p}_2$) is of type $C$ (resp. type $D$). By Part (i) for $\mfr{p}_2':=\mfr{p}_2\sqcup [1]$, we have
\begin{align*}
    d_{com,A}^{(n)}( (([n^l] \sqcup \mfr{p}_1)+ \mfr{p}_2)^{-}\sub{B}  )^+ \sub{C}&= d_{com,A}^{(n)}( (([n^l] \sqcup \mfr{p}_1{}^-)+ \mfr{p}_2)\sub{B}  )^+ \sub{C}\\
    &= d_{com,A}^{(n)}( (([n^l] \sqcup \mfr{p}_1{}^-)\sub{B}+ (\mfr{p}_2)\sub{C}  )^+ \sub{C}\\
    &=( [nl+b] \sqcup d_{com,A}^{(n)}( \mfr{p}_2' {}^-\sub{C}) )\sub{C}\\
    &= [nl+b] \sqcup d_{com,A}^{(n)}(\mfr{p}_2')^-\sub{C}\\
    &=([nl+b+1] \sqcup d_{com,A}^{(n)}(\mfr{p}_2))^{-}\sub{C}\\
    &=  d_{com,A}^{(n)}(\mfr{p})^{+-}\sub{C}.
\end{align*}
Suppose $l$ is odd. In this case, $\mfr{p}_1$ (resp. $\mfr{p}_2$) is of type $C$ (resp. type $B$). By Part (i) for $\mfr{p}_2$, we have
\begin{align*}
    d_{com,A}^{(n)}( (([n^l] \sqcup \mfr{p}_1)+ \mfr{p}_2)^{-}\sub{B}  )^+ \sub{C}&= d_{com,A}^{(n)}( (([n^l] \sqcup \mfr{p}_1{}^-)+ \mfr{p}_2)\sub{B}  )^+ \sub{C}\\
    &= d_{com,A}^{(n)}( (([n^l] \sqcup \mfr{p}_1{}^{+-})\sub{B}+ \mfr{p}_2{}^-\sub{C}  )^+ \sub{C}\\
    &=( [nl+b+1] \sqcup d_{com,A}^{(n)}( \mfr{p}_2 {}^-\sub{C}) )\sub{C}\\
    &= [nl+b+1] \sqcup d_{com,A}^{(n)}(\mfr{p}_2)^-\sub{C}\\
    &=  d_{com,A}^{(n)}(\mfr{p})^{+-}\sub{C}.
\end{align*}
This completes the verification of this case and the reduction for Part (iii).

\textbf{Part (iv):} 
First, we deal with the case that $\mfr{p}_1$ is empty. If $l$ is even, then $\mfr{p}_2$ is of type $D$. By Part (vi) for $\mfr{p}_2$, we have
\begin{align*}
    d_{com,A}^{(n)}( ([n^l]\sqcup \mfr{p}_2)^{+}\sub{B} )\sub{B}&= d_{com,A}^{(n)}( ([n^l]\sqcup [1]) + \mfr{p}_2{}^{+-}\sub{C} )\sub{B}\\
    &=  [nl+1] \sqcup d_{com,A}^{(n)}(\mfr{p}_2{}^{+-}\sub{C}) \sub{D}\\
    &= [nl+1]\sqcup d_{com,A}^{(n)}(\mfr{p}_2)\sub{D}\\
    &= ([nl+1]\sqcup d_{com,A}^{(n)}(\mfr{p}_2))\sub{B}\\
    &= d_{com,A}^{(n)}( \mfr{p})^{+}\sub{B}.
\end{align*}
If $l$ is odd, then $\mfr{p}_2$ is of type $B$. By Part (vi) for $\mfr{p}_2':=\mfr{p}_2\sqcup [1]$, we have
\begin{align*}
    d_{com,A}^{(n)}( ([n^l]\sqcup \mfr{p}_2)^{+}\sub{B} )\sub{B}&= d_{com,A}^{(n)}( [n^l] + \mfr{p}_2{}^{+}\sub{C} )\sub{B}\\
    &=  [nl] \sqcup d_{com,A}^{(n)}((\mfr{p}_2')^{+-}\sub{C}) \sub{D}\\
    &= [nl]\sqcup d_{com,A}^{(n)}(\mfr{p}_2')\sub{D}\\
    &= [nl+1]^{-} \sqcup d_{com,A}^{(n)}(\mfr{p}_2)^{-+}\sub{D}\\
    &= ([nl+1] \sqcup d_{com,A}^{(n)}(\mfr{p}_2))\sub{B}\\
    &= d_{com,A}^{(n)}( \mfr{p})^{+}\sub{B}.
\end{align*}
This completes the verification of this case. In the following, we assume $\mfr{p}_1$ is non-empty.

Suppose that $l$ is even. In this case $\mfr{p}_1$ (resp. $\mfr{p}_2$) is of type $C$ (resp. type $D$). By Part (vi) for $\mfr{p}_2$, we have
\begin{align*}
    d_{com,A}^{(n)}( ([n^l]\sqcup \mfr{p}_1) + \mfr{p}_2)^{+}\sub{B} )\sub{B}&= d_{com,A}^{(n)}( (([n^l]\sqcup \mfr{p}_1)^+ \sub{B} + \mfr{p}_2 {}^{+-}\sub{C} )\sub{B}\\
    &=( d_{com,A}^{(n)}( [n^l] \sqcup \mfr{p}_1{}^+ \sub{B} ) \sqcup d_{com,A}^{(n)}(\mfr{p}_2{}^{+-} \sub{C})  )\sub{B}\\
    &= [nl+b+1] \sqcup d_{com,A}^{(n)}(\mfr{p}_2{}^{+-} \sub{C})  )  \sub{D}\\
    &= [nl+b+1] \sqcup d_{com,A}^{(n)}(\mfr{p}_2)  \sub{D}\\
     &= ([nl+b+1]  \sqcup d_{com,A}^{(n)}(\mfr{p}_2) )\sub{B}\\
    &= (  d_{com,A}^{(n)}( [n^l] \sqcup \mfr{p}_1{})  \sqcup d_{com,A}^{(n)}(\mfr{p}_2) )^{+}\sub{B}\\
    &= d_{com,A}^{(n)}(\mfr{p})^+\sub{B}.
\end{align*}
This completes the verification of this case. Next, suppose $l$ is odd.  In this case $\mfr{p}_1$ (resp. $\mfr{p}_2$) is of type $C$ (resp. type $B$). Let $\mfr{p}_2':=\mfr{p}_2 \sqcup [1]$, which is of type $D$. By Part (vi) for $\mfr{p}_2'$, we have
\begin{align*}
     d_{com,A}^{(n)}( (([n^l]\sqcup \mfr{p}_1) + \mfr{p}_2)^{+}\sub{B} )\sub{B}&= d_{com,A}^{(n)} (([n^l]\sqcup \mfr{p}_1) \sub{B} + \mfr{p}_2 {}^{+}\sub{C} )\sub{B}\\
    &=( d_{com,A}^{(n)}( [n^l] \sqcup (\mfr{p}_1)\sub{D} ) \sqcup d_{com,A}^{(n)}(\mfr{p}_2{}^{+} \sub{C})  )\sub{B}\\
    &= [nl+b] \sqcup d_{com,A}^{(n)}(\mfr{p}_2'{}^{+-} \sub{C})    \sub{D}\\
    &= [nl+b] \sqcup d_{com,A}^{(n)}(\mfr{p}_2')  \sub{D}\\
     &= [nl+b+1]^{-}\sub{B}  \sqcup d_{com,A}^{(n)}(\mfr{p}_2)^+\sub{D}\\
    &= (  [nl+b+1] \sqcup d_{com,A}^{(n)}(\mfr{p}_2) )\sub{B}\\
    &= d_{com,A}^{(n)}(\mfr{p})^+\sub{B}.
\end{align*}
This completes the verification of this case and the reduction for Part (iv).

\textbf{Part (v):} First, we deal with the case that $\mfr{p}_1$ is empty. Note that $n$ is odd and $l=l(\mfr{p}_2)= l(\mfr{p})$ is even. In this case, $\mfr{p}_2$ is of type $D$. By Part (vi) for $\mfr{p}_2$, we have
\begin{align*}
    d_{com,A}^{(n)}( ([n^l] +\mfr{p}_2 )^{+-} \sub{D}  )^{+} \sub{B} &= d_{com,A}^{(n)}( ([n^l] +\mfr{p}_2{}^{+-} )\sub{D}  )^{+} \sub{B}\\
    &= d_{com,A}^{(n)}( [n^l] +\mfr{p}_2{}^{+-}\sub{C}  )^{+} \sub{B}\\
    &= ([nl+1] \sqcup d_{com,A}^{(n)}(\mfr{p}_2{}^{+-}\sub{C})) \sub{B}\\
    &=[nl+1] \sqcup d_{com,A}^{(n)}(\mfr{p}_2{}^{+-}\sub{C})\sub{D}\\
    &=[nl+1] \sqcup d_{com,A}^{(n)}(\mfr{p}_2)\sub{D}\\
    &=( [nl] \sqcup d_{com,A}^{(n)}(\mfr{p}_2))^{+} \sub{B}\\
    &= d_{com,A}^{(n)}(\mfr{p})^{+}\sub{B}.
\end{align*}
This completes the verification of this case. In the following, we assume $\mfr{p}_1$ is non-empty.

Suppose $l$ is even. In this case, $\mfr{p}_1$ (resp. $\mfr{p}_2$) is of type $C$ (resp. type $D$). By Part (vi) for $\mfr{p}_2$, we have
\begin{align*}
    d_{com,A}^{(n)}( (([n^l]\sqcup \mfr{p}_1) +\mfr{p}_2 )^{+-} \sub{D}  )^{+} \sub{B} &= d_{com,A}^{(n)}( ([n^l]\sqcup \mfr{p}_1{}^{-}) +\mfr{p}_2{}^{+} )\sub{D}  )^{+} \sub{B}\\
    &= d_{com,A}^{(n)}( ([n^l]\sqcup \mfr{p}_1{}^{+-}\sub{D} ) +\mfr{p}_2{}^{+-}\sub{C}  )^{+} \sub{B}\\
    &= ([nl+b+1] \sqcup d_{com,A}^{(n)}(\mfr{p}_2{}^{+-}\sub{C}) )\sub{B}\\
    &= [nl+b+1] \sqcup d_{com,A}^{(n)}(\mfr{p}_2{}^{+-}\sub{C})\sub{D}\\
    &=[nl+b+1] \sqcup d_{com,A}^{(n)}(\mfr{p}_2)\sub{D}\\
    &=([nl+b] \sqcup d_{com,A}^{(n)}(\mfr{p}_2))^{+}\sub{B}\\
    &= d_{com,A}^{(n)}(\mfr{p})^{+}\sub{B}.
\end{align*}
Suppose $l$ is odd. In this case, $\mfr{p}_1$ (resp. $\mfr{p}_2$) is of type $C$ (resp. type $B$). By Part (vi) for $\mfr{p}_2 \sqcup [1]$, we have
\begin{align*}
    d_{com,A}^{(n)}( (([n^l]\sqcup \mfr{p}_1) +\mfr{p}_2 )^{+-} \sub{D}  )^{+} \sub{B} &= d_{com,A}^{(n)}( ([n^l]\sqcup \mfr{p}_1{}^{-}) +\mfr{p}_2{}^{+} )\sub{D}  )^{+} \sub{B}\\
    &= d_{com,A}^{(n)}( ([n^l]\sqcup \mfr{p}_1{}^{-}\sub{B} ) +\mfr{p}_2{}^{+}\sub{C}  )^{+} \sub{B}\\
    &= ([nl+b] \sqcup d_{com,A}^{(n)}((\mfr{p}_2\sqcup [1]){}^{+-}\sub{C}) )\sub{B}\\
    &= [nl+b] \sqcup d_{com,A}^{(n)}((\mfr{p}_2\sqcup[1]){}^{+-}\sub{C})\sub{D}\\
    &=[nl+b] \sqcup d_{com,A}^{(n)}(\mfr{p}_2 \sqcup[1])\sub{D}\\
    &=[nl+b+1]^- \sqcup d_{com,A}^{(n)}(\mfr{p}_2)^+\sub{D}\\
    &=([nl+b+1] \sqcup d_{com,A}^{(n)}(\mfr{p}_2))\sub{B}\\
    &= d_{com,A}^{(n)}(\mfr{p})^{+}\sub{B}.
\end{align*}

This completes the verification of this case and the reduction for Part (v).

\textbf{Part (vi):} First, we deal with the case that $\mfr{p}_1$ is empty. Note that $l=l(\mfr{p}_2)= l(\mfr{p})$ is even since $\mfr{p}$ is of type $D.$ Suppose that the smallest part of $\mfr{p}_2$ is not less than $n$, so that there is a partition $\mfr{p}_2'$ of type $D$ such that $\mfr{p}=[(2n)^l] + \mfr{p}_2'$. Then by Part (vi) for $\mfr{p}_2$, we have
\begin{align*}
    d_{com,A}^{(n)}( ([(2n)^l] + \mfr{p}_2')^{+-}\sub{C})\sub{D}&= d_{com,A}^{(n)}( ([(2n)^l] + (\mfr{p}_2')^{+-})\sub{C})\sub{D}\\
    &= d_{com,A}^{(n)}( [(2n)^l] +  (\mfr{p}_2')^{+-}\sub{C} )\sub{D}\\
    &= ([nl,nl] \sqcup d_{com,A}^{(n)}((\mfr{p}_2')^{+-}\sub{C}) )\sub{D}\\
    &= [nl,nl] \sqcup d_{com,A}^{(n)}((\mfr{p}_2')^{+-}\sub{C})\sub{D}\\
    &=[nl,nl]\sqcup d_{com,A}^{(n)}(\mfr{p}_2')\sub{D}\\
    &=(d_{com,A}^{(n)}([(2n)^l]) \sqcup d_{com,A}^{(n)}(\mfr{p}_2')) \sub{D}\\
    &= d_{com,A}^{(n)}(\mfr{p})\sub{D}.
\end{align*}
Next, suppose the smallest part of $\mfr{p}_2$ is less than $n$. Then the largest part of $d_{com,A}^{(n)}(\mfr{p}_2{}^{+-}\sub{C})$ is less than $nl$. Thus, by Part (v) for $\mfr{p}_2$, we have
\begin{align*}
    d_{com,A}^{(n)}(([n^l] + \mfr{p}_2)^{+-}\sub{C})\sub{D}&=d_{com,A}^{(n)}(([n^l]+  \mfr{p}_2{}^{+-})\sub{C}) \sub{D}\\
    &=d_{com,A}^{(n)}([n^l]+  \mfr{p}_2{}^{+-}\sub{C}) \sub{D}\\
    &= ([nl] \sqcup d_{com,A}^{(n)}(\mfr{p}_2{}^{+-}\sub{D}))\sub{D}\\
    &= [nl-1] \sqcup d_{com,A}^{(n)}(\mfr{p}_2{}^{+-}\sub{D})^+ \sub{B}\\
    &= [nl-1] \sqcup d_{com,A}^{(n)}(\mfr{p}_2)^+ \sub{B}\\
    &=([nl]^{-} \sqcup d_{com,A}^{(n)}(\mfr{p}_2)^+ \sub{B})\\
    &= ([nl]\sqcup  d_{com,A}^{(n)}(\mfr{p}_2))\sub{D}\\
    &=  d_{com,A}^{(n)}(\mfr{p})\sub{D}.
\end{align*}
This completes the verification of this case. In the following, we assume $\mfr{p}_1$ is non-empty.

Suppose that $l$ is odd. In this case, $\mfr{p}_1$ (resp. $\mfr{p}_2$) is of type $B$ (resp. type $C$). By Part (iv) for $\mfr{p}_2$,  we have
\begin{align*}
    d_{com,A}^{(n)}( (([n^l]\sqcup \mfr{p}_1) + \mfr{p}_2)^{+-}\sub{C} ) \sub{D}&=    d_{com,A}^{(n)}( (([n^l]\sqcup \mfr{p}_1{}^{-}) + \mfr{p}_2 {}^{+} )\sub{C} ) \sub{D}\\
    &= d_{com,A}^{(n)}( ([n^l]\sqcup \mfr{p}_1{}^{-}\sub{C}) + \mfr{p}_2 {}^{+} \sub{B} ) \sub{D}\\
    &= ([nl+b-1] \sqcup d_{com,A}^{(n)}( \mfr{p}_2 {}^{+} \sub{B} ) )\sub{D}\\
    &= [nl+b-1]\sqcup d_{com,A}^{(n)}( \mfr{p}_2 {}^{+} \sub{B} )\sub{B}\\
    &= d_{com,A}^{(n)}([n^l]\sqcup \mfr{p}_1)^{-} \sub{B}\sqcup d_{com,A}^{(n)}( \mfr{p}_2) {}^{+} \sub{B} \\
    &= (d_{com,A}^{(n)}([n^l]\sqcup \mfr{p}_1{})\sqcup d_{com,A}^{(n)}( \mfr{p}_2) )\sub{D}\\
    &= d_{com,A}^{(n)}(\mfr{p}) \sub{D}
\end{align*}
This completes the verification of this case. Next, suppose that $l$ is even. In this case, $\mfr{p}_1$ (resp. $\mfr{p}_2$) is of type $D$ (resp. type $C$). If $\mfr{p}_1= [n]\sqcup \mfr{p}_1'$, then 
\[ (([n^l]\sqcup \mfr{p}_1{}^{-}) + \mfr{p}_2 {}^{+} )\sub{C}=(( [n^{l+1}]+ \mfr{p}_2{}^+) \sqcup (\mfr{p}_1')^{-})\sub{C}=( [n^{l+1}]+ \mfr{p}_2{}^+)\sub{C} \sqcup (\mfr{p}_1')^{-}\sub{C} = ([n^{l+1}]\sqcup (\mfr{p}_1')^{-}\sub{C})+\mfr{p}_2{}^+\sub{B},  \]
and similar argument in the previous case works. If the largest part of $\mfr{p}_1$ is less than $n$, then 
\[(([n^l]\sqcup \mfr{p}_1{}^{-}) + \mfr{p}_2 {}^{+} )\sub{C} = (([n^l]+\mfr{p}_2{}^+)\sqcup \mfr{p}_1{}^-)\sub{C}=([n^l]+\mfr{p}_2{}^{+-})\sub{C} \sqcup \mfr{p}_1{}^{+-}\sub{C}= ([n^l]\sqcup \mfr{p}_1{}^{+-}\sub{C}) + \mfr{p}_2 {}^{+-} \sub{D}. \]
Thus, by Part (v) for $\mfr{p}_2$,  we have
\begin{align*}
    d_{com,A}^{(n)}( (([n^l]\sqcup \mfr{p}_1) + \mfr{p}_2)^{+-}\sub{C} ) \sub{D}&=    d_{com,A}^{(n)}( (([n^l]\sqcup \mfr{p}_1{}^{-}) + \mfr{p}_2 {}^{+} )\sub{C} ) \sub{D}\\
    &= d_{com,A}^{(n)}( ([n^l]\sqcup \mfr{p}_1{}^{+-}\sub{C}) + \mfr{p}_2 {}^{+-} \sub{D} ) \sub{D}\\
    &= ([nl+b]\sqcup d_{com,A}^{(n)}(\mfr{p}_2{}^{+-}\sub{D})) \sub{D}\\
    &= [nl+b-1] \sqcup d_{com,A}^{(n)}(\mfr{p}_2{}^{+-}\sub{D})^{+} \sub{B}\\
    &=[nl+b-1] \sqcup d_{com,A}^{(n)}(\mfr{p}_2)^{+} \sub{B}\\
    &=([nl+b] \sqcup d_{com,A}^{(n)}(\mfr{p}_2))\sub{D}\\
    &= d_{com,A}^{(n)}(\mfr{p})\sub{D}.
\end{align*}
This completes the verification of this case and the reduction for Part (vi).

\end{proof}

\subsection{\texorpdfstring{Compatibility with induction for $d_{com,X}^{(n)}$}{}}
In this subsection, we prove the compatibility of $d_{com,X}^{(n)}(\mfr{p})$. We start with a useful identity.
\begin{lemma}\label{lem trans ind}
    Let $\mfr{q}_0, \mfr{q}_1, \mfr{q}_2$ be three partitions. Let $X \in \{C,D\}$ if $|\mfr{q}_0|$ is even and $X=B$ otherwise. We have the following identity
    \begin{align}\label{eq trans ind CM93}
        ((\mfr{q}_0+ 2 \mfr{q}_1)\sub{X} + 2 \mfr{q}_2)\sub{X} = ( \mfr{q}_0 + 2( \mfr{q}_1+\mfr{q}_2 ) )\sub{X}.
    \end{align}
\end{lemma}
\begin{proof}
To make the argument clear, we write $E(\mfr{q}_0, \mfr{q}_1, \mfr{q}_2)=1$ if \eqref{eq trans ind CM93} holds.

    First, observe that $E(\mfr{q}_0, \mfr{q}_1, \mfr{q}_2)=1$ if $\mfr{q}_0$ is already of type $X$. Indeed, the identity is a consequence of the transitivity of induction of nilpotent orbits (see \cite{CM93}*{Proposition 7.1.4(ii)}). 
    Next, we show that it suffices to show that $E(\mfr{p}_0, \emptyset, \mfr{p}_2)=1$ for arbitrary $\mfr{p}_0$ and $\mfr{p}_2$. Indeed,
    \begin{align*}
        ((\mfr{q}_0+ 2 \mfr{q}_1)\sub{X} + 2 \mfr{q}_2)\sub{X}&= (((\mfr{q}_0)\sub{X}+ 2 \mfr{q}_1)\sub{X} + 2 \mfr{q}_2)\sub{X}\\
        &= ( (\mfr{q}_0)\sub{X} + 2( \mfr{q}_1+\mfr{q}_2 ) )\sub{X}\\
        &= ( \mfr{q}_0 + 2( \mfr{q}_1+\mfr{q}_2 ) )\sub{X},
    \end{align*}
    where the first and the third equalities follows from  $E(\mfr{q}_0, \emptyset, \mfr{q}_1)=1$ and $E(\mfr{q}_0, \emptyset, \mfr{q}_1+\mfr{q}_2)=1$, and the second equality holds since $E((\mfr{q}_0)_X, \mfr{q}_1, \mfr{q}_2)=1$ as   
    $(\mfr{q}_0)_X$ is of type $X$.

    Thirdly, if $\mfr{q}_2= \mfr{q}_{21}+ \mfr{q}_{22}$, then we show that $E(\mfr{p}, \emptyset, \mfr{q}_2)=1$ for arbitrary $\mfr{p}$ if $E(\mfr{p}', \emptyset, \mfr{q}_{21})=1$ and $E(\mfr{p}'', \emptyset, \mfr{q}_{22})=1$ for arbitrary $\mfr{p}', \mfr{p}''$. Indeed, we have 
    \begin{align*}
        (\mfr{p}\sub{X}+ 2(\mfr{q}_{21}+\mfr{q}_{22}))&= ((\mfr{p}\sub{X}+2 \mfr{q}_{21})_X +2 \mfr{q}_{22})\sub{X}\\
        &= ((\mfr{p}+2 \mfr{q}_{21})_X +2 \mfr{q}_{22})\sub{X}\\
        &=(\mfr{p}+2 \mfr{q}_{21} +2 \mfr{q}_{22})\sub{X},
    \end{align*}
where the equalities hold since $E(\mfr{p}_X, \mfr{q}_{21}, \mfr{q}_{22})=1$, $E(\mfr{p},\emptyset, \mfr{q}_{21})=1$ and $E(\mfr{p}+2 \mfr{q}_{21}, \emptyset, \mfr{q}_{22})=1$ respectively.

Finally, we reduce to show that $E(\mfr{q}, \emptyset, \mfr{q}_2)=1$ under the assumption that $\mfr{q}_2$ can not be written as a sum of two non-empty partitions. In other words, $\mfr{q}_2=[1^d]$ for some $d\gest 1$.
We rewrite the desired equality as follows
    \[ (\mfr{q}\sub{X} + [2^d])\sub{X} = ( \mfr{q} + [2^d] )\sub{X}.\]    
    If $l(\mfr{q})\lest d$, then the identity clearly holds by the process of computing $X$-collapse. If  $l(\mfr{q})>d$, let $a$ denote the $d$-th largest part of $\mfr{q}$ and write $\mfr{q}= \mfr{q}_{>a} \sqcup [a^{r}] \sqcup \mfr{q}_{<a}$, where every part of $\mfr{q}_{>a}$ (resp. $\mfr{q}_{<a}$) is strictly greater than (resp. less than) $a$. Applying \cite{HLLS24}*{Lemma 12.2} twice, we may further assume that $ \mfr{q}_{>a}$ is either empty, or equal to $[a+1]$, and that $ \mfr{q}_{<a}$ is either empty, or equal to $[a-1]$. Then the desired equality follows from a case by case direct computation.
\end{proof}

We need an analogue of \cite{HLLS24}*{Lemmas 12.4, 12.7, 12.10} for $D$-collapse.
\begin{lemma}\label{lem HLLS24 D-collapse}
   Let $\mathfrak{p}$ be a partition of $2n$ and write $\mfr{p}=\mfr{p}_{>b} \sqcup \mfr{p}_{=b} \sqcup \mfr{p}_{<b}$, where every part in $\mfr{p}_{>b}$ (resp. $\mfr{p}_{=b}$, $\mfr{p}_{<b}$) is strictly greater than (resp. equal to, strictly smaller than) $b$.
    Then
    \begin{align}\label{eq type D collapse [b,b] 1}
        (\mathfrak{p}\sqcup [b^2])\sub{D}= \mfr{p}\sub{D} \sqcup [b^2]
    \end{align}
    unless all of the following conditions hold:
    \begin{enumerate}
        \item [(i)] $b$ is even,
        \item [(ii)] $l(\mfr{p}_{>b})+ |\mfr{p}_{>b}|$ is odd, and
        \item [(iii)] $\mfr{p}_{=b}$ is empty.
    \end{enumerate}
    If \eqref{eq type D collapse [b,b] 1} fails, then
    \begin{align}\label{eq type D collapse [b,b] 2}
        (\mathfrak{p}\sqcup [b^2])\sub{D}= \mfr{p}\sub{D} \sqcup [b+1, b-1].
    \end{align}
\end{lemma}
\begin{proof}
If $b$ is odd, then \eqref{eq type D collapse [b,b] 1} follows from \cite{HLLS24}*{Corollary 12.3(iii)}. We assume that $b$ is even from now on.

If $l(\mfr{p}_{>b})+ |\mfr{p}_{>b}|$ is even, then by \cite{HLLS24}*{Lemma 12.2}, we have
\[ (\mfr{p}_{>b} \sqcup (\mfr{p}_{=b}\sqcup \mfr{p}_{<b}))\sub{D}= (\mfr{p}_{>b}) \sub{O}  \sqcup (\mfr{p}_{=b}\sqcup \mfr{p}_{<b}))\sub{O},\]
and
\begin{align*}
    (\mfr{p}_{>b} \sqcup (\mfr{p}_{=b}\sqcup [b^2]\sqcup \mfr{p}_{<b}))\sub{D}=  (\mfr{p}_{>b}) \sub{O} \sqcup (\mfr{p}_{=b}\sqcup [b^2]\sqcup \mfr{p}_{<b}))\sub{O} =  (\mfr{p}_{>b}) \sub{O} \sqcup  [b^2] \sqcup (\mfr{p}_{=b}\sqcup \mfr{p}_{<b}))\sub{O},
\end{align*}
where $O \in \{B,D\}$. We again obtain \eqref{eq type D collapse [b,b] 1}. We assume that $l(\mfr{p}_{>b})+ |\mfr{p}_{>b}|$ is odd from now on. 

Now we assume that $\mfr{p}_{=b}$ is non-empty and write $\mfr{p}_{=b}= [b^r]$. Then for some $O \in \{B,D\},$
Thus,
\[ (\mfr{p}_{>b} \sqcup ([b^r] \sqcup \mfr{p}_{<b}))\sub{D}= (\mfr{p}_{>b})^- \sub{O}  \sqcup [b+1]\sqcup ([b^{r-1}]\sqcup \mfr{p}_{<b})) \sub{O},\]
and
\begin{align*}
    (\mfr{p}_{>b} \sqcup (\mfr{p}_{=b}\sqcup [b^2]\sqcup \mfr{p}_{<b}))\sub{D}&=  (\mfr{p}_{>b})^- \sub{O} \sqcup (\mfr{p}_{=b}\sqcup [b^2]\sqcup \mfr{p}_{<b}))^+\sub{O} \\
    &=  (\mfr{p}_{>b})^{-} \sub{O} \sqcup  [b+1] \sqcup [b^2]  \sqcup ([b^{r-1}] \sqcup \mfr{p}_{<b})) \sub{O},
\end{align*}
 which verifies \eqref{eq type D collapse [b,b] 1} again.

Finally, assume that Conditions (i), (ii) and (iii) all hold. Then 
\[ (\mfr{p}_{>b} \sqcup ( \mfr{p}_{<b}))\sub{D}= (\mfr{p}_{>b})^- \sub{O} \sqcup (\mfr{p}_{<b})^+ \sub{O},\]
and
\begin{align*}
    (\mfr{p}_{>b} \sqcup ([b^2]\sqcup \mfr{p}_{<b}))\sub{D}&=  (\mfr{p}_{>b})^- \sub{O} \sqcup ( [b^2]\sqcup \mfr{p}_{<b}))^+\sub{O} \\
    &=  (\mfr{p}_{>b})^{-} \sub{O} \sqcup  [b+1] \sqcup  ([b]  \sqcup \mfr{p}_{<b})  \sub{O'}\\
    &= (\mfr{p}_{>b})^{-} \sub{O} \sqcup  [b+1,b-1]  \sqcup (\mfr{p}_{<b})^{+}  \sub{O'}.
\end{align*}
This proves \eqref{eq type D collapse [b,b] 2} in this case and completes the proof of the lemma.
\end{proof}

\begin{prop}\label{P:ind BCD}
    Let $X,X' \in \{B,C\}$ or $X=X'=D$. Take any $n \in \Z_{>0}$ and set $n^{\ast}:=n/\gcd(n,2)$.    For any partition $\mathfrak{p}$ of type $X'$ and $\mathfrak{q}$, such that $d_{com,X}^{(n)} ( \mathfrak{p})$ is well-defined, we have 
\begin{align}\label{eq ind sat com}
     d_{com,X}^{(n)} ( \mathfrak{p} \sqcup \mathfrak{q}\sqcup \mathfrak{q} )= ( d_{com,X}^{(n)}(\mathfrak{p}) + 2 d_{com,A}^{(n^{\ast})}(\mathfrak{q}))_{X}.
\end{align}
\end{prop}

\begin{proof}
    By Lemmas \ref{lem ind A}(a), \ref{lem trans ind}, it is not hard to reduce the problem to the case that $\mfr{q}=[b]$, which we assume from now on. We start with an easier case that $X \in \{B,C\}$ with  $n^{\ast}$ even or $X=D$. In this case, \eqref{eq ind sat com} follows directly from Lemmas \ref{lem ind A}(a) and \ref{lem trans ind}:
\begin{align*}
    d_{com,X}^{(n)} ( \mfr{p} \sqcup [b^2])&=d_{com,A}^{(n^{\ast})} (\mfr{p} \sqcup [b^2])_{X}\\
    &= ( d_{com,A}^{(n^{\ast})} (\mfr{p}) + 2d_{com,A}^{(n^{\ast})}([b]))_{X}\\
    &= ( d_{com,A}^{(n^{\ast})} (\mfr{p})_X + 2d_{com,A}^{(n^{\ast})}([b]))_{X}\\
    &= (d_{com,X}^{(n)} ( \mfr{p}) + 2d_{com,A}^{(n^{\ast})}([b]))_{X}.
\end{align*}
Now we deal with the rest of the cases.\\

    \noindent\underline{Case B: $X=B$ and $n^{\ast}$ is odd:}

    Note that $\mfr{p}$ must be of type $C$ in this case. By Lemma \ref{lem Achar identities}(iv), the left hand side of \eqref{eq ind sat com} can be rewritten as $ d_{com,A}^{(n^{\ast})} ( (\mfr{p} \sqcup [b^2] )^+ \sub{B} )\sub{B}$. If 
    \begin{align}\label{eq prop ind B 1}
        (\mfr{p} \sqcup[b^2])^+\sub{B}= \mfr{p}^{+}\sub{B} \sqcup[b^2]
    \end{align}
    holds,    then again by Lemmas Lemmas \ref{lem ind A}(a), \ref{lem trans ind}, we have
    \[ d_{com,A}^{(n^{\ast})} ( (\mfr{p} \sqcup [b^2] )^+ \sub{B} )\sub{B}=( d_{com,A}^{(n^{\ast})} ( \mfr{p}^{+}\sub{B})+ 2 d_{com,A}^{(n^{\ast})}([b]))_{B}= ( d_{com,B}^{(n)}(\mathfrak{p}) + 2 d_{com,A}^{(n^{\ast})}([b]))_{B},  \]
    which verifies \eqref{eq ind sat com}.
    On the other hand, if \eqref{eq prop ind B 1} fails, then by \cite{HLLS24}*{Lemma 12.7}, we must have $b$ is even and
    \[  (\mfr{p} \sqcup[b^2])^+\sub{B}= \mfr{p}^{+}\sub{B} \sqcup[b+1, b-1]. \]
    Moreover, we may write $\mfr{p}=\mfr{p}_1 \sqcup \mfr{p}_2$, where every part in $\mfr{p}_1$ (resp. $\mfr{p}_2$) is strictly greater (resp. smaller) than $b$, and $l(\mfr{p}_1)$ is even ($|\mfr{p}_1|$ is even since $\mfr{p}$ is of type $C$). Thus, by \cite{HLLS24}*{Lemma 12.2}, we have
    \[\mfr{p}^{+}\sub{B} = \mfr{p}_1{}^{+-} \sub{D} \sqcup \mfr{p}_2{}^+ \sub{B}.  \]
    Now write $b=an^{\ast}+c$. If $c \neq 0$, then 
    \[ d_{com,A}^{(n^{\ast})}([b+1,b-1])= [(n^{\ast})^a, c+1]+ [(n^{\ast})^a,c-1]= 2 [(n^{\ast})^a,c]= d_{com,A}^{(n^{\ast})}([b,b]), \]
    and \eqref{eq ind sat com} holds in this case by the same argument above. Thus in the following, we assume that $c=0$, i.e. $b=an^{\ast}$. Note that then $a$ must be even. 

    Since every part in $\mfr{p}_1{}^{+-} \sub{D}$ is greater than or equal to $b$, we may write
    \[ d_{com,A}^{(n^{\ast})}(\mfr{p}_1{}^{+-} {}_{D})= [ (l(\mfr{p}_1) n^{\ast})^a] \sqcup \mfr{p}_1', \]
where every part of $\mfr{p}_1'$ is smaller than or equal to $l(\mfr{p}_1) n^{\ast}$. On the other hand, let $\mfr{p}_2':= d_{com,A}^{(n^{\ast})}([b^2]\sqcup \mfr{p}_2{}^+ {}_{B})$. Then $l(\mfr{p}_2')=l(\mfr{p}_2'{}^-)=a$ and $|\mfr{p}_2'|$ is odd. 
Now we compute the right hand side of \eqref{eq ind sat com}:
\begin{align*}
    (d_{com,B}^{(n)}(\mfr{p})+d_{com,A}^{(n^{\ast})}([b^2]))_{B}&= (d_{com,A}^{(n^{\ast})}(\mfr{p}^{+}\sub{B}) + d_{com,A}^{(n^{\ast})}([b^2]))_{B}\\
    &=( d_{com,A}^{(n^{\ast})}(\mfr{p}_1{}^{+-}\sub{D})+ d_{com,A}^{(n^{\ast})}(\mfr{p}_2{}^+ \sub{B} \sqcup[b^2]) )_{B}\\
    &= ((\mfr{p}_2'+ [(l(\mfr{p}_1) n^{\ast})^{a}]) \sqcup \mfr{p}_1' )_{B}\\
    &= (\mfr{p}_2'+ [(l(\mfr{p}_1) n^{\ast})^{a}])^- \sub{D} \sqcup \mfr{p}_1'{}^+ \sub{B}
\end{align*}    
On the other hand, the left hand side of \eqref{eq ind sat com} is
\begin{align*}
    d_{com,A}^{(n^{\ast})}( \mfr{p}^+ \sub{B} \sqcup [b+1, b-1] )_B&=  ( d_{com,A}^{(n^{\ast})}(\mfr{p}_1{}^{+-}\sub{D})+ d_{com,A}^{(n^{\ast})}(\mfr{p}_2{}^+ \sub{B})+ [ (n^{\ast})^{a-1}, n^{\ast}-1, 1 ] )_B\\
    &= (  ([(l(\mfr{p}_1) n^{\ast})^{a}] \sqcup \mfr{p}_1') + ( \mfr{p}_2' {}^- \sqcup [1]) )_B\\
    &= ( (\mfr{p}_2'+ [(l(\mfr{p}_1) n^{\ast})^{a}])^-  \sqcup \mfr{p}_1'{}^+ )\sub{B}\\
     &= (\mfr{p}_2'+ [(l(\mfr{p}_1) n^{\ast})^{a}])^- \sub{D} \sqcup \mfr{p}_1'{}^+ \sub{B}.
\end{align*}
This completes the proof of \eqref{eq ind sat com} in this case.

The rest of the cases that $X=C$ and $n$ is odd or $n^{\ast}$ is odd follow from the same method, but using \cite{HLLS24}*{Lemma 12.4} and Lemma \ref{lem HLLS24 D-collapse} instead. We detail below.\\

\noindent\underline{Case C1: $X=C$ and $n$ is odd:}

In this case, $\mfr{p}$ is of type $B$. By Lemma \ref{lem Achar identities}(i), we have 
\[d_{com,A}^{(n^{\ast})}(\mfr{p} \sqcup [b^2])^{-}\sub{C} = d_{com,A}^{(n^{\ast})}((\mfr{p}\sqcup[b^2])^{-}\sub{C})\sub{C}.\]
We shall only deal with the hard case that $b=a n^{\ast}$ is odd, and 
\[  (\mfr{p}\sqcup[b^2])^{-}\sub{C} = \mfr{p}^{-}\sub{C} \sqcup [b+1, b-1],\]
where $\mfr{p}= \mfr{p}_1 \sqcup \mfr{p}_2$ with every part in $\mfr{p}_1$ (resp. $\mfr{p}_2$) is strictly greater (resp. smaller) than $b$ and $|\mfr{p}_1|$ is odd (see \cite{HLLS24}*{Lemma 12.4}). Note that $l(\mfr{p}_1)$ is also odd since $\mfr{p}$ is of type $B$. Thus, by \cite{HLLS24}*{Lemma 12.2}, we have 
\[ \mfr{p}^{-}\sub{C}= \mfr{p}_1{}^-\sub{C} \sqcup \mfr{p}_2{}^{+-}\sub{C}. \]
Now we write 
\[d_{com,A}^{(n^{\ast})}(\mfr{p}_1{}^- \sub{C})= [(l(\mfr{p}_1) n^{\ast})^a] \sqcup \mfr{p}_1',\ \   d_{com,A}^{(n^{\ast})}(\mfr{p}_2{}^{+-}\sub{C} \sqcup[b^2])= \mfr{p}_2', \]
where every part of $\mfr{p}_1'$ is smaller than or equal to $l(\mfr{p}_1) n^{\ast}$, and $l(\mfr{p}_2')=l(\mfr{p}_2' {}^-)=a$ and $|\mfr{p}_2'|$ is even. Now we compute the right hand side of \eqref{eq ind sat com}:

\begin{align*}
    (d_{com,C}^{(n)}(\mfr{p})+d_{com,A}^{(n^{\ast})}([b^2]))_{C}&= (d_{com,A}^{(n^{\ast})}(\mfr{p}^{-}\sub{C}) + d_{com,A}^{(n^{\ast})}([b^2]))_{C}\\
    &=( d_{com,A}^{(n^{\ast})}(\mfr{p}_1{}^{-}\sub{C})+ d_{com,A}^{(n^{\ast})}(\mfr{p}_2{}^{+-} \sub{C} \sqcup[b^2]) )_{C}\\
    &= ((\mfr{p}_2'+ [(l(\mfr{p}_1) n^{\ast})^{a}]) \sqcup \mfr{p}_1' )_{C}\\
    &= (\mfr{p}_2'+ [(l(\mfr{p}_1) n^{\ast})^{a}])^- \sub{C} \sqcup \mfr{p}_1'{}^+ \sub{C}.
\end{align*}  
On the other hand, the left hand side of \eqref{eq ind sat com} is
\begin{align*}
    d_{com,A}^{(n^{\ast})}( \mfr{p}^- \sub{C} \sqcup [b+1, b-1] )_C&=  ( d_{com,A}^{(n^{\ast})}(\mfr{p}_1{}^{-}\sub{C})+ d_{com,A}^{(n^{\ast})}(\mfr{p}_2{}^{+-} \sub{C})+ [ (n^{\ast})^{a-1}, n^{\ast}-1, 1 ] )_C\\
    &= (  ([(l(\mfr{p}_1) n^{\ast})^{a}] \sqcup \mfr{p}_1') + ( \mfr{p}_2' {}^{-} \sqcup [1]) )_C\\
    &= ( (\mfr{p}_2'+ [(l(\mfr{p}_1) n^{\ast})^{a}])^-  \sqcup \mfr{p}_1'{}^+ )\sub{C}\\
     &= (\mfr{p}_2'+ [(l(\mfr{p}_1) n^{\ast})^{a}])^- \sub{C} \sqcup \mfr{p}_1'{}^+ \sub{C}.
\end{align*}
This completes the proof in this case.\\

 \noindent\underline{Case C2: $X=C$ and $n$ is even but $n^{\ast}$ is odd:}

In this case, $\mfr{p}$ is of type $C$. By Lemma \ref{lem Achar identities})(ii), we have 
\[ d_{com,A}^{(n^{\ast})}(\mfr{p}\sqcup[b^2])^{+-}\sub{C}= d_{com,A}^{(n^{\ast})} ((\mfr{p} \sqcup[b^2] )\sub{D})\sub{C}. \]
We shall only deal with the case that the hard case that $b=a n^{\ast}$ and 
\[ (\mfr{p} \sqcup[b^2]) \sub{D}=\mfr{p} \sub{D} \sqcup [b+1, b-1],\]
in which case we have $b$ is even, and $\mfr{p}= \mfr{p}_1 \sqcup \mfr{p}_2$ with every part in $\mfr{p}_1$ (resp. $\mfr{p}_2$) is strictly greater (resp. smaller) than $b$ and $l(\mfr{p}_1)+|\mfr{p}_1|$ is odd by Lemma \ref{lem HLLS24 D-collapse}. By \cite{HLLS24}*{Lemma 12.2}, we have for some $O \in \{B,D\}$,
\[ (\mfr{p_1} \sqcup \mfr{p}_2 )\sub{D}= \mfr{p}_1{}^{-} \sub{O} \sqcup \mfr{p}_2{}^{+}\sub{O},\]
and every part of $\mfr{p}_1{}^{-} \sub{O}$ is still strictly greater than $b$. Note that both $|\mfr{p}_1|$ and $|\mfr{p}_2|$ are even since $\mfr{p}$ is of type $C$.

Now we write 
\[d_{com,A}^{(n^{\ast})}(\mfr{p}_1 {}^-\sub{O})= [(l(\mfr{p}_1) n^{\ast})^a] \sqcup \mfr{p}_1',\ \   d_{com,A}^{(n^{\ast})}(\mfr{p}_2{}^{+}\sub{O} \sqcup[b^2])= \mfr{p}_2', \]
where every part of $\mfr{p}_1'$ is smaller than or equal to $l(\mfr{p}_1) n^{\ast}$, and $l(\mfr{p}_2')=l(\mfr{p}_2')=a$ and $|\mfr{p}_2'{}^-|$ is even. Now we compute the right hand side of \eqref{eq ind sat com}:
\begin{align*}
    (d_{com,C}^{(n)}(\mfr{p})+d_{com,A}^{(n^{\ast})}([b^2]))_{C}&= (d_{com,A}^{(n^{\ast})}(\mfr{p}\sub{D}) + d_{com,A}^{(n^{\ast})}([b^2]))_{C}\\
    &=( d_{com,A}^{(n^{\ast})}(\mfr{p}_1{}^{-}\sub{O})+ d_{com,A}^{(n^{\ast})}(\mfr{p}_2{}^{+} \sub{O} \sqcup[b^2]) )_{C}\\
    &= ((\mfr{p}_2'+ [(l(\mfr{p}_1) n^{\ast})^{a}]) \sqcup \mfr{p}_1' )_{C}\\
    &= (\mfr{p}_2'+ [(l(\mfr{p}_1) n^{\ast})^{a}])^- \sub{C} \sqcup \mfr{p}_1'{}^+ \sub{C}.
\end{align*}  
On the other hand, the left hand side of \eqref{eq ind sat com} is
\begin{align*}
    d_{com,A}^{(n^{\ast})}( \mfr{p} \sub{D} \sqcup [b+1, b-1] )_C&=  ( d_{com,A}^{(n^{\ast})}(\mfr{p}_1{}^{-}\sub{O})+ d_{com,A}^{(n^{\ast})}(\mfr{p}_2{}^{+} \sub{O})+ [ (n^{\ast})^{a-1}, n^{\ast}-1, 1 ] )_C\\
    &= (  ([(l(\mfr{p}_1) n^{\ast})^{a}] \sqcup \mfr{p}_1') + ( \mfr{p}_2' {}^{-} \sqcup [1]) )_C\\
    &= ( (\mfr{p}_2'+ [(l(\mfr{p}_1) n^{\ast})^{a}])^-  \sqcup \mfr{p}_1'{}^+ )\sub{C}\\
     &= (\mfr{p}_2'+ [(l(\mfr{p}_1) n^{\ast})^{a}])^- \sub{C} \sqcup \mfr{p}_1'{}^+ \sub{C}.
\end{align*}
This completes the proof in this case.
\end{proof}

\subsection{\texorpdfstring{Compatibility with induction for $\AP{X}(\lambda_{X'}^{(n)}(\mfr{p}))$ for type $B,C$}{}}

\begin{prop}\label{prop ind BCD Ann}
     Let $X,X' \in \{B,C\}$. Take any $n \in \Z_{>0}$ and set $n^{\ast}:=n/\gcd(n,2)$.    For any partition $\mathfrak{p}$ of type $X'$ and $\mathfrak{q}$,  we have 
        \[ \AP{X}( \lambda_{X'}^{(n)} (\mfr{p}\sqcup \mfr{q}\sqcup \mfr{q}) )=(\AP{X}( \lambda_{X'}^{(n)} (\mfr{p}) ) +2 \AP{A}( \lambda_{A}^{(n^{\ast})} (\mfr{q}) ) )_X\]
\end{prop}
\begin{proof}
Let $\lambda:= \lambda_{X'}^{(n)}(\mfr{p}\sqcup \mfr{q} \sqcup \mfr{q})$, $\lambda':= \lambda^{(n)}_{X'}(\mfr{p})$ and $\lambda'':= \lambda^{(n^{\ast})}_{A}(\mfr{q})$. Consider decompositions  
\begin{align*}
    \lambda_{A}^{(n^{\ast})}(\mfr{p} \sqcup \mfr{q}\sqcup \mfr{q})&= \lambda_{X',0} + \lambda_{X',\half{1}}+ \sum_{i=1}^{l_{X'}} \lambda_{X',i},\\
    \lambda_{A}^{(n^{\ast})}(\mfr{p})&= (\lambda')_{X',0} + (\lambda')_{X',\half{1}}+ \sum_{i=1}^{l_{X'}} (\lambda')_{X',i},\\
    \lambda_{A}^{(n^{\ast})}( \mfr{q})&= (\lambda'')_{A,0} + (\lambda'')_{A,\half{1}}+ \sum_{i=1}^{l_{X'}} (\lambda'')_{A,i},
\end{align*} 
so that for $i \in \{0, \half{1},1,\ldots, l_{X'}\}$, any two elements in $ \lambda_{X',i}\sqcup (\lambda')_{X',i}\sqcup (\lambda'')_{A,i}$ have integral difference. Note that $(\lambda')_{X',i}$ or $ (\lambda'')_{A,i}$ may possibly be empty. Then by definition \eqref{eq decomp lambds BCD}, we have 
\[\lambda_{X',i}= (\lambda')_{X',i}\sqcup (\lambda'')_{A,i}\sqcup (\lambda'')_{A,i}.\]
Also note that Lemma \ref{lem partition of wt} implies that for any $x,y \in \R$ such that $x-y \in \Z$, if $|x|>|y|$, then the multiplicity of $x$ in $(\lambda')_{X',i}$ (resp. $(\lambda'')_{A,i}$) is not greater than that of $y$. Therefore, according to the definition \eqref{eq partition lambda sep} and Lemma \ref{lem RS}, we have
\[ \mfr{p}_{X',i}(\lambda)= \mfr{p}_{X',i}(\lambda') \sqcup \mfr{p}_{A,i}(\lambda'')\sqcup \mfr{p}_{A,i}(\lambda'').\]

If $X=B$, using the formula in Theorem \ref{thm BMW} and Lemma \ref{lem trans ind}, we have
\begin{align*}
    &\AP{B}(\lambda) \\
    =& \left( \mfr{p}_{X',0}(\lambda)^\ast   +\mfr{p}_{X',\half{1}}(\lambda)^{\ast} + \sum_{i=1}^{l_{X'}} \mfr{p}_{X',i}(\lambda)^{\ast}  \right)_{B}\\
    =& \left( (\mfr{p}_{X',0}(\lambda')^\ast + 2 \mfr{p}_{A,0}(\lambda'')^{\ast})   +(\mfr{p}_{X',\half{1}}(\lambda')^{\ast} + 2 \mfr{p}_{A,\half{1}}(\lambda'')^{\ast}) + \sum_{i=1}^{l_{X'}} (\mfr{p}_{X',i}(\lambda')^{\ast}+  2 \mfr{p}_{A,i}(\lambda'')^{\ast})  \right)_{B} \\
    =& \left( \mfr{p}_{X',0}(\lambda')^\ast   +\mfr{p}_{X',\half{1}}(\lambda')^{\ast} + \sum_{i=1}^{l_{X'}} \mfr{p}_{X',i}(\lambda')^{\ast}  + 2 \AP{A}(\lambda_A^{(n^{\ast})}(\mfr{q}))\right)_{B}\\
    =& \left(\left( \mfr{p}_{X',0}(\lambda')^\ast   +\mfr{p}_{X',\half{1}}(\lambda')^{\ast} + \sum_{i=1}^{l_{X'}} \mfr{p}_{X',i}(\lambda')^{\ast}  \right)_{B} + 2 \AP{A}(\lambda_A^{(n^{\ast})}(\mfr{q})) \right)_{B}\\
    =& (\AP{B}(\lambda') + 2 \AP{A}(\lambda''))\sub{B}.
    \end{align*}
    This proves the case $X=B$.
    For $X=C$, by the same computation above, it suffices to show the following identities for any $\mfr{q}$.
    \begin{enumerate}
        \item [(C1)] $(\mfr{p}\sqcup \mfr{q} \sqcup \mfr{q})^{-}\sub{C}{}^{\ast}=d_{com,C}^{(1)}(\mfr{p}\sqcup \mfr{q} \sqcup \mfr{q})=(d_{com,C}^{(1)}(\mfr{p}) + 2 \mfr{q}^{\ast})\sub{C}$ for $\mfr{p}$ of type $B$.
        \item [(C2)] $(\mfr{p}\sqcup \mfr{q} \sqcup \mfr{q})^{\ast}\sub{D}{}^{+-}\sub{C}=\widetilde{d}_{com,C}(\mfr{p}\sqcup \mfr{q} \sqcup \mfr{q})=(\widetilde{d}_{com,C}(\mfr{p}) + 2 \mfr{q}^{\ast})\sub{C}$ for $\mfr{p}$ of type $C$.
    \end{enumerate}
Here 
\begin{equation} \label{d-meta}
\widetilde{d}_{com,C}(\mfr{p}):= \mfr{p}^{\ast}\sub{D}{}^{+-}\sub{C}
\end{equation}
is the metaplectic Barbasch--Vogan duality defined in \cite{BMSZ23}. Later we will show that $\widetilde{d}_{com,C}(\mfr{p})=d_{com,C}^{(2)}(\mfr{p}) $, see Corollary \ref{cor eq *D*D*=*D*}.

    The identity (C1) is a consequence of \cite{HLLS24}*{Lemma 3.8} and Lemma \ref{lem trans ind}. Moreover, we have another identity 
      \begin{enumerate}
        \item [(B1)] $d_{com,B}^{(1)}(\mfr{p}\sqcup \mfr{q} \sqcup \mfr{q})=(d_{com,B}^{(1)}(\mfr{p})+ 2 \mfr{q}^{\ast})\sub{C}$ for $\mfr{p}$ of type $C$.
    \end{enumerate}
    For (C2), \cite{BMSZ23}*{Proposition 3.8} states that $\widetilde{d}_{BV,C}(\mfr{p})$ is characterized by the following identity
    \[ [1^{2a+1}] + \widetilde{d}_{com,C}(\mfr{p})= d_{com,B}^{(1)}( [2a] \sqcup \mfr{p}) \]
    for any integer $a \gest |\mfr{p}|/2$. Thus, taking any $a \gest |\mfr{p}|/2+ |\mfr{q}|$, an application of the equality (B1) gives 
    \begin{align*}
        [1^{2a+1}] + \widetilde{d}_{com,C}(\mfr{p}\sqcup \mfr{q}\sqcup \mfr{q})&= d_{com,B}^{(1)}( [2a] \sqcup \mfr{p}\sqcup \mfr{q}\sqcup \mfr{q} )\\
        &=(d_{com,B}^{(1)}( [2a] \sqcup \mfr{p} )+2 \mfr{q}^{\ast})\sub{B}\\
        &= ([1^{2a+1}] + \widetilde{d}_{com,C}(\mfr{p})+ 2 \mfr{q}^{\ast} )\sub{B}\\
        &= [1^{2a+1}]+ (\widetilde{d}_{com,C}(\mfr{p})+ 2 \mfr{q}^{\ast} )\sub{C}.
    \end{align*}
    Here the last equality follows from \cite{Ach03}*{Lemma 3.1}. This proves the identity (C2); hence the case $X=C$, which completes the proof of the proposition.
\end{proof}

\begin{cor}\label{cor eq *D*D*=*D*}
    Let $\mfr{p}$ be a partition of type $C$. We have 
    \[ \mfr{p}\sub{D}{}^{\ast}= \mfr{p}^{\ast+-}\sub{C}= \mfr{p}^{\ast} \sub{D}{}^{+-}\sub{C},  \]
 where the second term is $d_{com,C}^{(2)}(\mfr{p})$ and the last term is the metaplectic Barbasch--Vogan duality $\widetilde{d}_{com,C}(\mfr{p})$ as in \eqref{d-meta}.
\end{cor}
\begin{proof}
    The first equality follows from \cite{Ach03}*{Lemma 3.3} (or see Lemma \ref{lem Achar identities}(ii)), so it remains to show the second equality. Note that  $\mfr{p}^{\ast+-}\sub{C}= d_{com,C}^{(2)}(\mfr{p})$ and $\mfr{p}^{\ast} \sub{D}{}^{+-}\sub{C}= \AP{C}(\lambda_{C}^{(2)}(\mfr{p})).$ Write $ \mfr{p}= \mfr{p}' \sqcup \mfr{q} \sqcup \mfr{q}$, where $\mfr{p}'=[2p_1, 2p_2,\ldots, 2p_s]$ with $p_1>\cdots > p_s$. Then $(\mfr{p}')^{\ast}$ is already of type $D$. Therefore, by Propositions \ref{prop comparison A}, \ref{prop ind BCD Ann} and \ref{P:ind BCD}, we have 
\begin{align*}
    \mfr{p}^{\ast+-}\sub{C}&=d_{com,C}^{(2)}(\mfr{p}' \sqcup \mfr{q }\sqcup \mfr{q}) \\
    &= ( d_{com,C}^{(2)}(\mfr{p}') + 2 d_{com,A}^{(1)}(\mfr{q}))\sub{C}\\
    &= ( (\mfr{p}')^{\ast+-} \sub{C} + 2 \AP{A}(\lambda_A^{(1)}(\mfr{q}))\sub{C}\\
    &= ( (\mfr{p}') {}^{\ast} \sub{D}{}^{+-} \sub{C} + 2 \AP{A}(\lambda_A^{(1)}(\mfr{q}))\sub{C}\\
    &=  (\AP{C}(\lambda_{C}^{(2)}(\mfr{p}')+ 2 \AP{A}(\lambda_A^{(1)}(\mfr{q}))\sub{C}\\
    &=\AP{C}(\lambda_{C}^{(2)}(\mfr{p}' \sqcup \mfr{q }\sqcup \mfr{q})\\
    &= \mfr{p}^{\ast} \sub{D}{}^{+-}\sub{C}.
\end{align*}
    This completes the proof of the corollary.
\end{proof}

\subsection{\texorpdfstring{Proof of Theorem \ref{T:comparison} for type $B,C$}{}}

In this subsection, we prove Theorem \ref{T:comparison} based on Lemma \ref{lem ind A} and Proposition \ref{P:ind BCD}.

\begin{prop}
    Theorem \ref{T:comparison} holds when $X=B, C$.
\end{prop}
\begin{proof}
We prove the case $X=C$ in details. The case $X=B$ follows by similar (and simpler) argument, which we omit.

Suppose $ d_{com,C}^{(n)}(\mfr{q})$ is well-defined, where $\mfr{q}$ is a partition of type $X'$. More explicitly, if $X'=B$ (resp. $X'=C$), then $n$ must be odd (resp. even). Write $\lambda= \lambda_{X'}^{(n)}(\mfr{q})$ for short.  By Propositions \ref{P:ind BCD}, \ref{prop ind BCD Ann} and \ref{prop comparison A}, to show $\AP{C}(\lambda_{X'}^{(n)}(\mfr{q}))=d_{com,C}^{(n)}(\mfr{q})$,
we may assume $\mfr{q}=[q_1,\ldots, q_s]$, where $q_1>\cdots > q_s$ and $q_i$ are all odd (resp. even) if $X'=B$ (resp. $X'=C$).

 First, we deal with the case that $X'=B$. Note that $n$ must be odd in this case.  Under the assumption that $\mfr{q}$ consists of odd integers, $\mfr{p}_{B,\half{1}}(\lambda)$ is the empty partition. Therefore, by the compatibility with induction, we have 
\begin{align*}
    \AP{C}(\lambda)&= \left(\mfr{p}_{B,0}(\lambda)^{-} \sub{C} {}^\ast + \sum_{i=1}^{l_{B}} \mfr{p}_{B,i}(\lambda)^{\ast}\right)\sub{C}\\
    &= \left(d_{com,C}^{(1)}( \mfr{p}_{B,0}(\lambda) )+ \sum_{i=1}^{l_{B}/2} 2 d_{com,A}^{(1)}(\mfr{p}_{B,i}(\lambda)) \right)\sub{C}\\
    &= \left(d_{com,C}^{(1)}( \mfr{p}_{B,0}(\lambda)) + 2 d_{com,A}^{(1)}\left(  \bigsqcup_{i=1}^{l_{B}/2} \mfr{p}_{B,i}(\lambda)\right) \right)\sub{C}\\
    &= d_{com,C}^{(1)} \left( \mfr{p}_{B,0}(\lambda) \sqcup \bigsqcup_{i=1}^{l_{B}} \mfr{p}_{B,i}(\lambda) \right)\\
    &= \left( \mfr{p}_{B,0}(\lambda) \sqcup \bigsqcup_{i=1}^{l_{B}} \mfr{p}_{B,i}(\lambda) \right)^{\ast-} \sub{C}\\
    &= d_{com,A}^{(n)}(\mfr{q})^{-}\sub{C}\\
    &= d_{com,C}^{(n)}(\mfr{q}).
\end{align*}
This completes the verification of this case.

Next, we consider the case that $X'=C$ and $n^{\ast}$ is odd. In this case, under the assumption that $\mfr{q}$ consists of even integers, we have $\lambda_{C,0}$ defined in \eqref{eq decomp lambds BCD} is empty; hence $\mfr{p}_{C,0}(\lambda)=[1]$. Therefore, by the compatibility with induction and Corollary \ref{cor eq *D*D*=*D*}, we have 
\begin{align*}
    \AP{C}(\lambda)&=  \left( \mfr{p}_{C,\half{1}}(\lambda)^{\ast}\sub{D} {}^{+-} \sub{C}+ \sum_{i=1}^{l_{C}} \mfr{p}_{C,i}(\lambda)^{\ast}  \right)_{C}\\
    &= \left( \mfr{p}_{C,\half{1}}(\lambda)^{\ast} {}^{+-} \sub{C}+  \sum_{i=1}^{l_{C}/2} 2d_{com,A}^{(1)}(\mfr{p}_{C,i}(\lambda) ) \right)_{C}\\
    &= \left(d_{com,C}^{(2)}(\mfr{p}_{C,\half{1}}(\lambda) ) + 2 d_{com,A}^{(1)}\left(\bigsqcup_{i=1}^{l_C/2} \mfr{p}_{C,i}(\lambda)\right)  \right)_{C}\\
    &= d_{com,C}^{(2)} \left( \mfr{p}_{C,\half{1}}(\lambda) \sqcup  \bigsqcup_{i=1}^{l_C/2} \mfr{p}_{C,i}(\lambda)\right)\\
    &= d_{com,A}^{(n/2)} ( \mfr{q})^{+-}\sub{C}\\
    &= d_{com,C}^{(n)}(\mfr{q}).
    \end{align*}
This completes the verification of this case.

Finally, we deal with the case that $X'=C$ and $n^{\ast}$ is even. In this case, under the assumption that $\mfr{q}$ consists of even integers, we have $\mfr{p}_{C,\half{1}}(\lambda)=[1]$ and $\mfr{p}_{C,0}(\lambda)$ is empty. Therefore,
\begin{align*}
    \AP{C}(\lambda)=  \left(  \sum_{i=1}^{l_{C}} \mfr{p}_{C,i}(\lambda)^{\ast}  \right)_{C}= d_{com,A}^{(n/2)} ( \mfr{q})_{C}= d_{com,C}^{(n)}(\mfr{q}).
\end{align*}
This completes the verification of this case and the proof of the proposition.
\end{proof}

\subsection{\texorpdfstring{Proof of Theorem \ref{T:comparison} for type $D$}{}}
In this subsection,  we prove Theorem \ref{T:comparison} for type $D$. This implies the compatibility with induction for $\AP{D}(\lambda_{D}^{(n)}(\mfr{p}))$ since the same statement holds for $d_{com,D}^{(n)}(\mfr{p})$ by Proposition \ref{P:ind BCD}.

\begin{lemma}\label{lem type D, n=4}
    Let $\mfr{q}$ be a partition of type $D$. Let $\lambda:= \lambda_{D}^{(n)}(\mfr{q})$ and let 
     $\mfr{p}_{D, 0}(\lambda), \mfr{p}_{D, \half{1}}(\lambda)$ be defined as \eqref{eq partition lambda sep}. Then
     \begin{align}\label{eq type D n=4}
         ( \mfr{p}_{D,0}(\lambda)^{-\ast} \sub{D}   +\mfr{p}_{D,\half{1}}(\lambda) \sub{D} {}^{\ast})\sub{D}= (\mfr{p}_{D,0}(\lambda)^- \sqcup \mfr{p}_{D,\half{1}}(\lambda))^{\ast} \sub{D}. 
     \end{align}
\end{lemma}
\begin{proof}
Write $\mfr{p}_{odd}:=\mfr{p}_{D,0}(\lambda)^- $, $\mfr{p}_{even}:= \mfr{p}_{D,\half{1}}(\lambda)$ and $\mfr{p}:= \mfr{p}_{odd} \sqcup \mfr{p}_{even}$. Also write $\mfr{q}=[q_1^{r_1},\ldots, q_s^{r_s}]$. Each part $q_i$ contributes to at most $2$ parts in $\mfr{p}$, which we explicate case by case now. Let $I:=\{1,\ldots, s\}$ and $I_{odd}:= \{i \in I\ | \ q_i \text{ is odd}\}$ and $I_{even}:= I \setminus I_{odd}$. For each $i \in I_{odd}$, we write $q_i=2q_i'+1$.

First, for $i \in I_{even}$, observe that $\{ \frac{q_i-1}{2n^{\ast}} ,\ldots, \frac{1-q_i}{2n^{\ast}}\}$ does not contain any integer. Hence, $[q_i^{r_i}]$ contributes $[p_i^{r_i}]$ in $\mfr{p}$, where 
\[ p_i= \#\{ 0 \lest t \lest q_i\ | \ n^{\ast} \text{ divides } q_i-1- 2t  \}.\]
Note that both $r_i$ and $p_i$ are even.

Suppose that $n^{\ast}$ is odd. Then $[q_i^{r_i} ]$ contributes $[p_i^{r_i}]$, where
\[ p_i= \#\{  -q_i' \lest t \lest  q_i' \ | \ n^{\ast} \text{ divides } t  \}.\]
Note that $p_i$ must be odd. Thus,
\[\mfr{p}_{odd}= [p_i^{r_i}]_{i \in I_{odd}},\ \mfr{p}_{even}= [p_i^{r_i}]_{i \in I_{even}}. \]
Let $\mfr{p}_{even/2}:= [p_i^{r_i/2}]_{i \in I_{even}}$. Then since $\mfr{p}_{even}$ is already of type $D$, by Proposition \ref{P:ind BCD}, we see that
\begin{align*}
    (\mfr{p}_{odd}{}^\ast \sub{D} + (\mfr{p}_{even})\sub{D}{}^\ast)\sub{D}&= (d_{com,D}^{(1)}(\mfr{p}_{odd}) + 2 d_{com,A}^{(1)}(\mfr{p}_{even/2}) )\sub{D}\\
    &= d_{com,D}^{(1)}( \mfr{p}_{odd} \sqcup \mfr{p}_{even/2} \sqcup \mfr{p}_{even/2} )\\
    &= (\mfr{p}_{odd} \sqcup \mfr{p}_{even})^{\ast} \sub{D}.
\end{align*}
This proves \eqref{eq type D n=4} in this case. 

Suppose that $n^{\ast}$ is even. Then $[q_i^{r_i} ]$ contributes $[p_i^{r_i}, (p_i+1)^{r_i}]$, where
\[ 2p_i+1 = \#\{  -q_i' \lest t \lest  q_i' \ | \ n^{\ast} \text{ divides } t  \}.\]
Let $\{ p_{i,odd}, p_{i,even}\}= \{p_i, p_i+1\}$ where $p_{i,odd}$ is odd. Then 
\[\mfr{p}_{odd}= [p_{i,odd}^{r_i}]_{i \in I_{odd}},\ \mfr{p}_{even}= [p_i^{r_i}]_{i \in I_{even}} \sqcup [p_{i,even}^{r_i}]_{i \in I_{odd}}. \]
Thus, the desired equality is a consequence of the following lemma.
\end{proof}

\begin{lemma}
Let $2p_1+1 \gest \cdots \gest 2 p_{2s}+1$ be a sequence of odd integers and let $ q_1 \gest \cdots \gest  q_r$ be a sequence of even integers. Write $\{ p_{i,odd}, p_{i,even} \}=\{p_i, p_i+1\}$ with $p_{i,odd}$ being odd and let 
\[ \mfr{p}_{odd}:= [p_{i,odd}]_{ 1 \lest i \lest 2s},\ \ \mfr{p}_{even}:= [p_{i,even}]_{1\lest i \lest 2s} \sqcup [q_1^2,\ldots, q_r^2].\]
Then 
\begin{align}\label{eq type D, n=4 2}
    ((\mfr{p}_{odd})^\ast \sub{D} + (\mfr{p}_{even})\sub{D}{}^\ast)\sub{D}= (\mfr{p}_{odd} \sqcup \mfr{p}_{even})^{\ast} \sub{D}.
\end{align}
\end{lemma}
\begin{proof}
We apply induction on $s+r$ throughout the proof.

First, we reduce to the case that $\{p_1,\ldots, p_{2s}\}$ are distinct. Suppose on the contrary that $p_i=p_{i+1}$ for some $i$ and write $\mfr{p}_{odd}= \mfr{p}_{odd}'\sqcup [p_{i,odd}^2]$. Since $\mfr{p}_{odd}$ is of type $D$, by \cite{Ach03}*{Lemma 3.3} and \cite{HLLS24}*{Lemma 12.10}, we see that
\[ (\mfr{p}_{odd})^{\ast} \sub{D}=(\mfr{p}_{odd}' \sqcup [p_{i,odd}^2] )^{+-}\sub{C}{}^{\ast}=(\mfr{p}_{odd}')^{+-}\sub{C}{}^{\ast}+ \beta^{\ast}= (\mfr{p}_{odd}')^\ast \sub{D} + \beta^{\ast}, \]
where $\beta \in \{[p_{i,odd}^2], [p_{i,odd}+1, p_{i,odd}-1] \}.$ Moreover, if $\beta=[p_{i,odd}+1, p_{i,odd}-1], $ then $p_{i-1, odd} > p_{i,odd}> p_{i+2,odd}$, and $i$ is odd (set $p_{0,odd}:=\infty$). Thus, applying Proposition \ref{P:ind BCD} if $\beta=[p_{i,odd}^2]$, or applying
\cite{Ach03}*{Lemma 3.1} twice (the formula for $\lambda_D$ with $k,p$ even in the notation there) if $\beta=[p_{i,odd}+1, p_{i,odd}-1]$, we have
\[ (\mfr{p}_{odd} \sqcup \mfr{p}_{even})^{\ast} \sub{D}= ((\mfr{p}_{odd}' \sqcup \mfr{p}_{even})^{\ast} \sub{D} + \beta^{\ast})\sub{D}.\]
Hence,
\begin{align*}
    (\mfr{p}_{odd}' \sqcup \mfr{p}_{even} \sqcup [p_{i,odd}^2])^{\ast}\sub{D}&= (( \mfr{p}_{odd}' \sqcup \mfr{p}_{even})^{\ast}\sub{D} + \beta^{\ast} )\sub{D}\\
    & = (((\mfr{p}_{odd}')^\ast \sub{D} + (\mfr{p}_{even})\sub{D}{}^\ast)\sub{D}+ \beta^{\ast})\sub{D}\\
    & \lest ((\mfr{p}_{odd}')^\ast \sub{D} + (\mfr{p}_{even})\sub{D}{}^\ast+ \beta^{\ast})\sub{D}\\
    &= ((\mfr{p}_{odd})^{\ast} \sub{D}+ (\mfr{p}_{even})\sub{D}{}^\ast) \sub{D}.
\end{align*}
Here, the second equality follows from the induction hypothesis, and the inequality follows from the fact that $ \gamma + \delta \lest \gamma + \delta'$ if $\delta \lest \delta'$. On the other hand, applying Proposition \ref{P:ind BCD} and Lemma \ref{lem trans ind}, we have
\begin{align*}
      (\mfr{p}_{odd}' \sqcup \mfr{p}_{even} \sqcup [p_{i,odd}^2])^{\ast}\sub{D}&= (( \mfr{p}_{odd}' \sqcup \mfr{p}_{even})^{\ast}\sub{D} + 2[p_{i,odd}]^{\ast} )\sub{D}\\
      &=(((\mfr{p}_{odd}')^\ast \sub{D} + (\mfr{p}_{even})\sub{D}{}^\ast)\sub{D}+ 2[p_{i,odd}]^{\ast})\sub{D}\\
      &= ((\mfr{p}_{odd}')^\ast \sub{D} + (\mfr{p}_{even})\sub{D}{}^\ast+ 2[p_{i,odd}]^{\ast})\sub{D}\\
      & \gest ((\mfr{p}_{odd}')^\ast \sub{D} + (\mfr{p}_{even})\sub{D}{}^\ast+ \beta^{\ast})\sub{D}\\
      &= ((\mfr{p}_{odd})^{\ast} \sub{D}+ (\mfr{p}_{even})\sub{D}{}^\ast) \sub{D}.
\end{align*}
This completes the induction argument. Thus, we assume that $\{p_1,\ldots, p_{2s}\}$ are distinct for the rest of the proof.

Next, we reduce to the case that $r=0$. Suppose on the contrary that $r >0$ and write $\mfr{p}_{even}= \mfr{p}_{even}' \sqcup [q_r^{2}]$. Then according to Lemma \ref{lem HLLS24 D-collapse}, we have 
    \[ (\mfr{p}_{even})\sub{D}=(\mfr{p}_{even}')\sub{D} \sqcup \beta,\ \ (\mfr{p}_{odd} \sqcup \mfr{p}_{even})\sub{D}=(\mfr{p}_{odd} \sqcup \mfr{p}_{even}')\sub{D} \sqcup \beta \]
    for the same $\beta \in \{[q_r^2], [q_r+1, q_r-1]\}.$ Note that if $\beta = [q_r+1, q_r-1]$, then  we must have $ p_{j,even} > q_{r} > p_{j+1, even}$ (set $p_{2s+1,even}= 0$) for some odd $j$ (by Lemma \ref{lem HLLS24 D-collapse}(ii),(iii)). Thus, applying Proposition \ref{P:ind BCD} if $\beta=[q_r^2]$ or applying
\cite{Ach03}*{Lemma 3.1} twice if $\beta=[q_r+1,q_r-1]$, we see that
\[ (\mfr{p}_{odd} \sqcup \mfr{p}_{even})^{\ast} \sub{D}= (\mfr{p}_{odd} \sqcup \mfr{p}_{even}')^{\ast} \sub{D} + \beta^{\ast}\]
    always holds. Then the rest of the argument is similar to the previous reduction, which we omit. We assume that $r=0$ in the rest of the proof.

    Thirdly, we reduce to the case that $p_{2j,odd}=p_{2j+1,odd}$ for $1 \lest j \lest s-1$. Suppose on the contrary that $p_{2j,odd} > p_{2j+1,odd}$. For $\ast \in \{odd, even\}$,  we decompose $\mfr{p}_{\ast}= \mfr{p}_{\ast, \lest 2j} \sqcup \mfr{p}_{\ast, > 2j}$ where $\mfr{p}_{\ast, \lest 2j}:= [p_{i,\ast}]_{i \lest 2j}$. Note that $\mfr{p}_{\ast, \lest 2j}$ is evenly superior than $\mfr{p}_{\ast, >2j}$ in the notation in \cite{Ach03}*{\S 3.1} by our assumption. Applying \cite{Ach03}*{Lemma 3.1}, we observe that
    \begin{itemize}
        \item  $(\mfr{p}_{odd})^{\ast}\sub{D}= (\mfr{p}_{odd,\lest 2j})^{\ast}\sub{D}+ (\mfr{p}_{odd,> 2j})^{\ast}\sub{D}$, and
        \item $(\mfr{p}_{even})\sub{D}=(\mfr{p}_{even,\lest 2j})\sub{D}\sqcup (\mfr{p}_{even,>2j})\sub{D}$, and
        \item $(\mfr{p}_{even} \sqcup\mfr{p}_{odd})^{\ast} \sub{D}= (\mfr{p}_{even,\lest 2j} \sqcup\mfr{p}_{odd,\lest 2j})^{\ast} \sub{D}+ (\mfr{p}_{even,> 2j} \sqcup\mfr{p}_{odd,>2j})^{\ast} \sub{D}$.
    \end{itemize}
    Finally, observe that the last part of  $((\mfr{p}_{odd,\lest 2j})^\ast \sub{D} + (\mfr{p}_{even, \lest 2j})\sub{D}{}^\ast)^{\ast}$ is not smaller than $ p_{2j,odd}-1$, and the first part of $((\mfr{p}_{odd,> 2j})^\ast \sub{D} + (\mfr{p}_{even, > 2j})\sub{D}{}^\ast)^{\ast}$ is not larger than $p_{2j+1,odd}+1$. Thus, we have
    \begin{align*}
        ((\mfr{p}_{odd,\lest 2j})^\ast \sub{D} &+ (\mfr{p}_{even, \lest 2j})\sub{D}{}^\ast + (\mfr{p}_{odd,> 2j})^\ast \sub{D} + (\mfr{p}_{even, > 2j})\sub{D}{}^\ast )\sub{D}\\
        &= ((\mfr{p}_{odd,\lest 2j})^\ast \sub{D} + (\mfr{p}_{even, \lest 2j})\sub{D}{}^\ast )\sub{D}+ ((\mfr{p}_{odd,> 2j})^\ast \sub{D} + (\mfr{p}_{even, > 2j})\sub{D}{}^\ast )\sub{D}.
    \end{align*}
    Combining these four observations, we see that the desired equality follows from induction hypothesis.

    Finally, we compute the case that $p_{2j,odd}=p_{2j+1,odd}$ for all $1 \lest j \lest s-1$ directly. By the reduction we have done so far, we obtain that 
    \[ p_{2i+1,even} \gest p_{2i+2, even}> p_{2i+2, odd}= p_{2i+3, odd}> p_{2i+3,even}\]
    for $0 \lest i \lest s-2$. Then it follows from a direct computation that
    \begin{align*}
        (\mfr{p}_{odd})^{\ast}\sub{D}&= [ p_{1,odd}+1, p_{2,odd},\ldots, p_{2s-1,odd}, p_{2s,odd}-1]^{\ast}, \\(\mfr{p}_{even})\sub{D}{}^\ast&= (\sqcup_{1 \lest i \lest s} [p_{2i-1,even}-\epsilon_{i}, p_{2i,even}+\epsilon_i])^{\ast}
    \end{align*}
where $\epsilon_i= 1$ if $p_{2i-1,even}> p_{2i,even}$ and $0$ otherwise. 
Then, by separating into 4 situations based on the signs of  $p_{1,odd} - p_{1,even}$ and $p_{2s,odd}- p_{2s, even}$, one can check that $(\mfr{p}_{odd})^{\ast}\sub{D} + (\mfr{p}_{even})\sub{D}{}^\ast$ is already of type $D$ by a direct computation. Moreover, it is equal to the $D$-collapse of  $\mfr{p}_{even} \sqcup\mfr{p}_{odd}$. We give two examples to demonstrate this computation. This completes the proof of the lemma.
\end{proof}

\begin{eg}
As the first example, consider $[p_1,p_2,p_3,p_4]=[17,11,9,3]$. We have $\mfr{p}_{odd}=[9,5,5,1]$ and $\mfr{p}_{even}=[8,6,4,2]$. Thus, 
     \begin{align*}
        (\mfr{p}_{odd})^{\ast}\sub{D} + (\mfr{p}_{even})\sub{D}{}^\ast&= ( [9+1,5,5,1-1] \sqcup [8-1,6+1,4-1,2+1] )^{\ast}\\
        &= [9+1, 8-1, 6+1,5,5,4-1, 2+1,1-1]^\ast.
    \end{align*}
     A direct computation shows that 
    \[ [9+1, 8-1, 6+1,5,5,4-1, 2+1,1-1]^\ast= [7, 7, 7, 5, 5, 3, 3, 1, 1, 1],\]
    which is of type $D$. On the other hand, we may compute the $D$-collapse of $\mfr{p}_{even} \sqcup\mfr{p}_{odd}$ step by step as follows.
    \begin{align*}
        [9,8,6,5,5,4,2,1]^{\ast}\sub{D}&= [9,8,6,5,5,4,2+1]^{\ast}\sub{D} + [1-1]^{\ast}\\
        &= [9,8,6,5+1,5]^{\ast}\sub{D}+[4-1,2+1]^{\ast} + [1-1]^{\ast} \\
        &= [9,8,6+1]^{\ast}\sub{D}+[5+1-1,5]^{\ast}+[4-1,2+1]^{\ast} + [1-1]^{\ast}\\
        &= [9+1]^{\ast}\sub{D}+[8-1,6+1]+[5,5]^{\ast}+[4-1,2+1]^{\ast} + [1-1]^{\ast}\\
        &= [9+1, 8-1, 6+1,5,5,4-1, 2+1,1-1]^\ast.
    \end{align*}
 This completes the verification of this case.
 
 As the second example, consider $[p_1,p_2,p_3,p_4]=[19,11,9,5]$. We have $\mfr{p}_{odd}=[9,5,5,3]$ and $\mfr{p}_{even}=[10,6,4,2]$. Thus, 
    \begin{align*}
        (\mfr{p}_{odd})^{\ast}\sub{D} + (\mfr{p}_{even})\sub{D}{}^\ast&= ( [9+1,5,5,3-1] \sqcup [10-1,6+1,4-1,2+1] )^{\ast}\\
        &= [10,9, 6+1,5,5,4-1, 3,2]^\ast.
    \end{align*}
    Note that here $\{ 10-1, 9+1\}=\{10,9\}$ and $\{3-1,2+1\}=\{3,2\}$. A direct computation shows that 
    \[ [10,9, 6+1,5,5,4-1, 3,2]^\ast= [8, 8, 7, 5, 5, 3, 3, 2, 2, 1],\]
    which is of type $D$. On the other hand, we may compute the $D$-collapse of $\mfr{p}_{even} \sqcup\mfr{p}_{odd}$ step by step as follows.
    \begin{align*}
        [10,9,6,5,5,4,3,2]^{\ast}\sub{D}&=[10,9,6,5+1,5]^{\ast}\sub{D}  +[4-1,3,2]^{\ast}\\
        &=[10,9,6+1]^{\ast}\sub{D}+[5+1-1,5]^{\ast}  +[4-1,3,2]^{\ast}\\
        &= [10, 9, 6+1, 5,5,4-1, 3,2]^{\ast}.
    \end{align*}
        This completes the verification of this case.
\end{eg}
Now as a corollary of Lemma \ref{lem type D, n=4}, we prove the last case of Theorem \ref{T:comparison}.
\begin{prop}
    Theorem \ref{T:comparison} holds when $X=D$.
\end{prop}
\begin{proof}
Let $\mfr{q}$ be of type $D$ and let $\lambda:= \lambda_D^{(n)}(\mfr{q})$. By \eqref{eq decomp lambds BCD}, \eqref{eq partition lambda sep}, Proposition \ref{prop comparison A}, we have
\[d_{com,A}^{(n^{\ast})}(\mfr{q})= \left(\mfr{p}_{D,0}(\lambda)^- \sqcup \mfr{p}_{D,\half{1}}(\lambda) \sqcup \bigsqcup_{i=1}^{{l_{D}}} \mfr{p}_{D,i}(\lambda) \right)^{\ast}. \]
Thus, by Lemmas \ref{lem trans ind} and \ref{lem type D, n=4},
    \begin{align*}
        d_{com,A}^{(n^{\ast})}(\mfr{q})\sub{D}&= \left((\mfr{p}_{D,0}(\lambda)^- \sqcup \mfr{p}_{D,\half{1}}(\lambda))^{\ast} \sub{D}+\sum_{i=1}^{{l_{D}}} \mfr{p}_{D,i}(\lambda)^\ast \right)\sub{D}\\
       &= \left( ( \mfr{p}_{D,0}(\lambda)^{-\ast} \sub{D}   +\mfr{p}_{D,\half{1}}(\lambda) \sub{D}{}^\ast)\sub{D}+\sum_{i=1}^{{l_{D}}} \mfr{p}_{D,i}(\lambda)^\ast \right)\sub{D}\\
        &= \left(  \mfr{p}_{D,0}(\lambda)^{-\ast} \sub{D}   +\mfr{p}_{D,\half{1}}(\lambda) \sub{D}{}^\ast+\sum_{i=1}^{{l_{D}}} \mfr{p}_{D,i}(\lambda)^\ast \right)\sub{D}\\
        &=\AP{D}(\lambda^{(n)}_D(\mfr{q})).
    \end{align*}
   This completes the proof of the proposition.
\end{proof}

For type $D_r$, we will need to consider very even orbit with the labelling $\heartsuit \in \set{{\rm I}, {\rm II}}$.

\begin{prop} \label{P:D-I/II}
Let $\wt{G}$ be cover of $\SO_{2r}$ as in \S \ref{SS:fix-cov}. Let $\mfr{p} \in \mca{N}(\wt{G}^\vee)$ be a very even orbit. Then
$$d_{BV, \SO_{2r}}^{(n)}(\mfr{p}^{\heartsuit}) = d_{com, D}^{(n)}(\mfr{p}^\heartsuit)$$ for $\heartsuit \in \set{{\rm I}, {\rm II}}$.
\end{prop}
\begin{proof}
We first note that at the partition level, Theorem  \ref{T:comparison} gives
$$d_{BV, \SO_{2r}}^{(n)}(\mfr{p}) = d_{com, D}^{(n)}(\mfr{p}).$$
Let $M_{\rm I}:=\GL_r^{\rm I} \subset \SO_{2r}$ be the Siegel Levi subgroup associated with $\set{e_i - e_{i+1}}_{1\lest i \lest r-1}$. Let $M_{\rm II}:=\GL_r^{\rm II}$ be the other Siegel Levi subgroup associated with $\set{e_i - e_{i+1}}_{1\lest i \lest r-2} \cup \set{e_{r-1} + e_r}$.
We write 
$$\mfr{p}_{\rm Lus}^\heartsuit:=\mfr{p}^\heartsuit$$ to emphasize that we are following Lusztig's convention in the usage of labeling, as adopted in \cites{BMW25, BGWX25}. Also, we write $\mfr{p}_{\rm CM}^\heartsuit$ to indicate that we use the convention of Collingwood--McGovern \cite{CM93}. Their relation is as follows: for $\set{\heartsuit, \spadesuit} = \set{\rm I, II}$, we get
\begin{equation} \label{Lus-CM}
\mfr{p}_{\rm Lus}^\heartsuit 
=\begin{cases}
\mfr{p}_{\rm CM}^\heartsuit & \text{ if  $r/2$ is even},\\
\mfr{p}_{\rm CM}^\spadesuit & \text{ if  $r/2$ is odd}.
\end{cases}
\end{equation}

Consider $\mfr{p}_{\rm CM}^{\rm I} \in \mca{N}(\wt{G}^\vee)$, then we have  $\mfr{p}_{\rm CM}^{\rm I} = {\rm Sat}_{\wt{M}_{\rm I}^\vee}^{\wt{G}^\vee} \mfr{p}_0$ for a partition $\mfr{p}_0$ of $r$ uniquely determined by $\mfr{p}$.

By definition $d_{BV, \SO_{2r}}^{(n)}(\mfr{p}_{\rm CM}^{\rm I}) = {\rm AV}_{\mfr{g}_\BC}(h_{\mfr{p}_{\rm CM}^{\rm I}}^{(n)}/2)$. Now, as argued by Barbasch--Vogan in the proof of \cite{BV85}*{Proposition A2 (c)}, one has
$${\rm AV}_{\mfr{g}_\BC}(h_{\mfr{p}_{\rm CM}^{\rm I}}^{(n)}/2) \subseteq {\rm Ind}_{M_{\rm I}}^G {\rm AV}_{\mfr{m}_{I,\BC}}(h_{\mfr{p}_{\rm CM}^{\rm I}}^{(n)}/2),$$
that is,
\begin{equation} \label{E:inc-D}
d_{BV, \SO_{2r}}^{(n)}(\mfr{p}_{\rm CM}^{\rm I}) \subseteq {\rm Ind}_{M_{\rm I}}^G d_{BV, M_{\rm I}}^{(n)}(\mfr{p}_0).
\end{equation}
The underlying partitions of two sides of \eqref{E:inc-D} are equal by Proposition  \ref{P:ind BCD}; in particular, the two orbits in \eqref{E:inc-D} have the same dimension. This gives 
\begin{equation} \label{E:chain1}
d_{BV, \SO_{2r}}^{(n)}(\mfr{p}_{\rm CM}^{\rm I}) = {\rm Ind}_{M_{\rm I}}^G d_{BV, M_{\rm I}}^{(n)}(\mfr{p}_0) = {\rm Ind}_{M_{\rm I}}^G(\mfr{p}_0^*) = d_{BV, \SO_{2r}}^{(n)}(\mfr{p})_{\rm Lus}^{\rm I},
\end{equation}
where the last equality follows from \cite{BMW25}*{Proposition 7.5 (2)}.
Similarly, one has
\begin{equation} \label{E:chain2}
d_{BV, \SO_{2r}}^{(n)}(\mfr{p}_{\rm CM}^{\rm II}) = {\rm Ind}_{M_{\rm II}}^G d_{BV, M_{\rm I}}^{(n)}(\mfr{p}_0) = {\rm Ind}_{M_{\rm II}}^G(\mfr{p}_0^*) = d_{BV, \SO_{2r}}^{(n)}(\mfr{p})_{\rm Lus}^{\rm II},
\end{equation}
Combining \eqref{Lus-CM}, \eqref{E:chain1} and \eqref{E:chain2} gives us
\begin{equation}
d_{BV,\SO_{2r}}^{(n)}(\mfr{p}^\heartsuit) := 
\begin{cases}
d_{BV,\SO_{2r}}^{(n)}(\mfr{p})^\heartsuit & \text{ if  $r/2$ is even};\\
d_{BV,\SO_{2r}}^{(n)}(\mfr{p})^{\spadesuit} & \text{ if  $r/2$ is odd}.
\end{cases}
\end{equation}
In particular, it is identical to the rule of labellings in \ref{E:dBV-D 2}. This completes the proof.
\end{proof}

\subsection{Order reversing}
\begin{prop} \label{P:ord-rev}
For $X\in \{A, B,C,D\}$ and $n \in \Z_{>0}$, $d_{com,X}^{(n)}$ is order-reversing, i.e., if $d_{com, X}^{(n)}(\mfr{q}) \gest d_{com, X}^{(n)}(\mfr{q}')$ whenever $\mfr{q} \lest \mfr{q}'$.  
\end{prop}
\begin{proof}
Since the maps $\mfr{q} \mapsto \mfr{q}^+, \mfr{q}^-$ and $\mfr{q}_{X}$ are all order-preserving, it suffices to show that $d_{com,A}^{(n)}$ is order-reversing, or equivalently, $\mfr{q}^{(n)}:= d_{com,A}^{(n)}(\mfr{q})^{\ast}$ is order-preserving. By definition, if $\mfr{q}=[q_1,\ldots, q_r]$ and $q_i= a_i n + b_i$ for $a_i,b_i \in \Z_{\gest 0}$ with $b_i < n$, then
    \begin{align}\label{eq shrink}
        \mfr{q}^{(n)}= \bigsqcup_{i=1}^r [(a_i+1)^{b_i}, a_i^{n-b_i}].
    \end{align}
    Now we regard a partition $\mfr{q}$ as a non-increasing map $f_{\mfr{q}}: \Z_{>0} \to \Z_{\gest 0}$ which sends $i$ to its $i$-th part of $\mfr{q}$ if $i\lest l(\mfr{q})$ and zero otherwise. Then for each $i,j \in \Z_{>0}$ with $i<j$, we define a collapsing operator $T_{i,j}$ on the set of functions $\{f: \Z_{>0} \to \Z\}$ by
    \[ T_{i,j}(f)(k):= \begin{cases}
        f(k) &\text{ if }k \neq i, j,\\
        f(i)-1 & \text{ if }k=i,\\
        f(j)+1 & \text{ if }k=j.
    \end{cases} \]
    Then $\mfr{q}_1 > \mfr{q}_2$ if and only if there exists a sequence of partitions $\mfr{p}_1,\ldots, \mfr{p}_s$ such that $\mfr{p}_1= \mfr{q}_1$, $\mfr{p}_s= \mfr{q}_2$ and $f_{\mfr{p}_{k+1}}= T_{i_k, j_k}( f_{\mfr{p}_k})$ for some pair $(i_k,j_k)$. Now it is not hard to see from \eqref{eq shrink} that either $f_{\mfr{p}_{k+1}^{(n)}}=f_{\mfr{p}_{k}^{(n)}}$ or there exists a pair $(i_k^{(n)}, j_k^{(n)})$ such that $f_{\mfr{p}_{k+1}^{(n)}}=T_{i_k^{(n)}, j_k^{(n)}}f_{\mfr{p}_{k}^{(n)}}$. This shows that the map $\mfr{q}\mapsto \mfr{q}^{(n)}$ is order preserving and completes the proof of the proposition.
\end{proof}

\subsection{\texorpdfstring{A summary of results for $G$ of classical type}{}}
We give a summary of the main results in this section as follows.

\begin{thm} \label{T:ABCD}
Let $\wt{G}^{(n)}$ be the cover of $G$ of classical type $X \in \set{A, B, C, D}$ given in \S \ref{SS:fix-cov}. Let $d_{BV, G}^{(n)}$be the covering Barbasch--Vogan duality map $d_{BV, G}^{(n)}$ given in Definition \ref{D:dBV}. Then it satisfies the following properties:
\begin{enumerate}
\item[(i)] $d_{BV, G}^{(n)}(\mca{O})= d_{com, X}^{(n)}(\mca{O})$ for every orbit $\mca{O} \in \mca{N}(\wt{G}^\vee)$, where the latter is given in \S \ref{SS:fix-cov};
\item[(ii)] it is order-reversing, i.e., $d_{BV, G}^{(n)}(\mca{O}) \gest d_{BV, G}^{(n)}(\mca{O}')$ whenever $\mca{O} \lest \mca{O}'$;
\item[(iii)] for every covering Levi $\wt{M}^{(n)}$ and orbit $\mca{O} \subset {\rm Lie}(\wt{M}^{(n)})$, one has
\begin{equation}  \label{ind-cla}
d_{BV, G}^{(n)} \circ \Sat_{\wt{M}^\vee}^{\wt{G}^\vee}(\mca{O}) = \Ind_\mbf{M}^\mbf{G}\circ d_{BV, M}^{(n)}(\mca{O}).
\end{equation}
\end{enumerate}
\end{thm}
\begin{proof}
Here, (i) follows from Theorem \ref{T:comparison} and Proposition \ref{P:D-I/II}. In view of (i), the statement (ii) is a consequence of Proposition \ref{P:ord-rev}. Also, for type $A, B, C$ the statement (iii) follows Proposition \ref{P:ind BCD}. For type $D$, (iii) follows from the proof of Proposition \ref{P:D-I/II} and was used to show (i) here.
\end{proof}

\section{\texorpdfstring{Properties of $d_{BV, G}^{(n)}$ for exceptional $G$}{}} \label{S:dBV-exc}
In this section, we assume $G$ is of exceptional type and is almost simple and simply-connected. We consider the cover $\wt{G}^{(n)}$ associated with $Q: Y\to \Z$ satisfying that $Q(\alpha^\vee)=1$ for every short coroot $\alpha^\vee$.

For every orbit $\mca{O} \in \mca{N}(\wt{G}^\vee)$, we get from \eqref{E:hO}
$$\frac{h_\mca{O}^{(n)}}{2} = \sum_{\alpha \in \Delta} \frac{c_\alpha(\mca{O})}{2n_\alpha} \cdot \omega_\alpha.$$
We can thus implement the online algorithm 
    \begin{center}
\textcolor{blue}{http://test.slashblade.top:5000/lie/GKdim}
    \end{center}
given as in \cite{BGWX25} to compute $d_{BV,G}^{(n)}(\mca{O}):={\rm AV}_{\mfr{g}_\BC}(h_{\mca{O}}^{(n)}/2)$. In Appendix \ref{A:dBVexc}, we tabulate the data of $d_{BV, G}^{(n)}$ for such $\wt{G}^{(n)}$. We note that for $\wt{G}^{(n)}$ associated with a more general quadratic form $Q: Y \to \Z$, the map $d_{BV, G}^{(n)}$ is the one obtained from substituting $n_\alpha$ (for any short coroot $\alpha^\vee$) for $n$ in all the lists in Appendix \ref{A:dBVexc}.

\begin{thm} \label{T:EFG}
Let $\wt{G}^{(n)}$ be an $n$-fold cover of $G$ of exceptional type given as above. Then the following hold:
\begin{enumerate}
\item[(i)] the map $d_{BV,G}^{(n)}: \mca{N}(\wt{G}^\vee) \to \mca{N}(\mbf{G})$ is order-reversing;
\item[(ii)] if $\wt{M} \subseteq \wt{G}$ is a covering Levi subgroup, then one has
\begin{equation} \label{ind-exc}
d_{BV, G}^{(n)} \circ \Sat_{\wt{M}^\vee}^{\wt{G}^\vee}(\mca{O}) = \Ind_\mbf{M}^\mbf{G}\circ d_{BV, M}^{(n)}(\mca{O}),
\end{equation}
for  every $\mca{O} \in \mca{N}(\wt{M}^\vee)$. 
\end{enumerate}
\end{thm}
\begin{proof}
Property (i) follows from a direct observation using the data given in Appendix \ref{A:dBVexc}. For (ii), it also follows from the data there and the tables for ${\rm Ind}_{M}^G$ given in \cite{Spa82}*{Page 173--175}.

However, to reduce the work load for checking (ii), we show that it suffices to check the case when $\mca{O} \in \mca{N}(\wt{M}^\vee)$ is a distinguished orbit, by using induction. Indeed, suppose $\mca{O} = {\rm Sat}_{\wt{L}^\vee}^{\wt{M}^\vee}(\mca{O}_1)$ for a Levi $\mbf{L} \subset \mbf{M}$ of $\mbf{M}$ (and thus also of $\mbf{G}$), where $\mca{O}_1 \in \mca{N}(\wt{L}^\vee)$ is a distinguished orbit of $\wt{L}^\vee$.
We see that 
$$
\begin{aligned}
d_{BV, G}^{(n)} \circ \Sat_{\wt{M}^\vee}^{\wt{G}^\vee}(\mca{O}) = & d_{BV, G}^{(n)} \circ  \Sat_{\wt{L}^\vee}^{\wt{G}^\vee}(\mca{O}_1) \\
 = &  \Ind_\mbf{L}^\mbf{G}\circ d_{BV, L}^{(n)}(\mca{O}_1) \\
 = &  \Ind_\mbf{M}^\mbf{G} \circ \Ind_\mbf{L}^\mbf{M} \circ d_{BV, L}^{(n)}(\mca{O}_1) \\
 = & \Ind_\mbf{M}^\mbf{G} \circ d_{BV, M}^{(n)} \circ {\rm Sat}_{\wt{L}^\vee}^{\wt{M}^\vee}(\mca{O}_1) \\
 = & \Ind_\mbf{M}^\mbf{G}\circ d_{BV, M}^{(n)}(\mca{O}).
\end{aligned}
$$
Here, in the second equality above we used the induction hypothesis that \eqref{ind-exc} holds when the orbit is distinguished in the (smaller) Levi $L$. Also, the fourth equality follows from our induction of \eqref{ind-exc} on the semisimple rank of $G$, which is greater than that of $M$.

In view of the above, and the result \eqref{ind-cla} for classical groups (since $G$ might contain a Levi subgroup of classical type), we see that to prove \eqref{ind-exc}, it suffices to consider $G$ of exceptional type, $M \subseteq G$ a proper Levi, and $\mca{O} \in \mca{N}(\wt{M}^\vee)$ a distinguished orbit. For each of such a case, we have verified \eqref{ind-exc} using the lists in Appendix \ref{A:dBVexc} and the tables in \cite{Spa82}*{Page 173-175} regarding induction, as mentioned. Here we give an example to illustrate.

Consider $\wt{G}^{(2)}=\wt{E}_7^{(2)}, M=D_6$ and $\mca{O}=D_6(a_1) =(9,3)$. Then we have from Theorem \ref{T:ABCD} (i) that 
$$d_{BV, D_6}^{(2)}(\mca{O}) = (3^3,1^3) = {\rm Ind}^{D_6}_{A_2 + D_3} 0.$$
This gives 
$$\Ind_\mbf{M}^\mbf{G}\circ d_{BV, M}^{(n)}(\mca{O}) = \Ind_{A_2 + D_3}^{E_7} 0 = D_5(a_1)+ A_1,$$
where the last equality follows from \cite{Spa82}*{Page 174}. This orbit is exactly $d_{BV, E_7}^{(2)}(D_6(a_1))$ by Appendix \ref{A:dBVexc}.
\end{proof}

\begin{remark}
As already used in Proposition \ref{P:D-I/II}, an alternative way of proving \eqref{ind-exc} for $G$ of all types is to use the method of Barbasch--Vogan in the proof of \cite{BV85}*{Proposition A2 (c)}. First, it entails an inclusion $d_{BV, G}^{(n)} \circ \Sat_{\wt{M}^\vee}^{\wt{G}^\vee}(\mca{O}) \subseteq \Ind_\mbf{M}^\mbf{G}\circ d_{BV, M}^{(n)}(\mca{O})$, and then a check of dimension-equality of the two sides gives the equality \eqref{ind-exc}. This method is in essence equivalent to our proof of \eqref{ind-cla} and \eqref{ind-exc}.
\end{remark}

\section{Some evidence for the conjecture}

\subsection{\texorpdfstring{Unramified theta representations of $\wt{G}$}{}} \label{SS:theta}
In this subsection, let $\wt{G}$ be a general Brylinski--Deligne cover of an almost simple $G$. We assume $p\nmid n$, and we use freely the notation from \cites{GLT25, KOW}. In particular, we assume that there is an unramified distinguished genuine character $\chi^\dag \in {\rm Irr}(Z(\wt{T}))$, which gives a distinguished finite-dimensional genuine representation $\pi^\dag \in \Irr(\wt{T})$. For every $\nu \in X\otimes \R$, we have the map $\delta_\nu: T \to \BC^\times$ given by
$$\delta_v(y\otimes a) = \val{a}_F^{\nu(y)}$$
on the generators $y\otimes a \in T$. We get the unramified principal series $I(\pi^\dag, \nu)$ of $\wt{G}$. If $\nu \in X\otimes \R$ is exceptional (i.e., $\nu(\alpha_{Q,n}^\vee)=1$ for all simple root $\alpha$), then $I(\pi^\dag, \nu)$ has an unique irreducible quotient $\Theta(\nu):=\Theta(\pi^\dag, \nu)$ and a unique irreducible subrepresentation ${\rm St}(\nu):={\rm St}(\pi^\dag, \nu)$. One has
$$\AZ(\Theta(\nu)) = \St(\nu).$$
Since $\St(\nu)$ is always generic, and $\mca{O}(\phi_{\Theta(\nu)}) =\set{0}$, we see that the inequality in Conjecture \ref{Conj} (i) is actually an equality for $\pi=\Theta(\nu)$.

On the other hand, the case of $\pi = \St(\nu)$ is treated in the work of Karasiewicz--Okada--Wang \cite{KOW}, as follows.

Let $\wt{G}^{(n)}$ be a Brylinski--Deligne cover of an almost simple $G$. Let $\mca{O}_{\rm reg} \in \mca{N}(\wt{G}^\vee)$ be the regular orbit of $\wt{G}^\vee$. Then we see that 
$$\nu:=\frac{h_{\mca{O}_{\rm reg}}^{(n)}}{2} \in X\otimes \R$$ 
is exactly the exceptional vector in $X\otimes \R$. We have $d_{BV, G}^{(n)}(\mca{O}_{\rm reg}):={\rm AV}_{\mfr{g}_\BC}(\nu)$. Let $W_\nu$ be the integral Weyl-subgroup associated with $\nu$, and let $w_0 \in W_\nu$ be the longest element. Then $w_0(\nu)$ is anti-dominant in the sense of \cite{BGWX25}*{Definition 2.5}. Also the element $w_0$ is associated with the sign (special) representation $\varepsilon_{W_\nu}$ of $W_\nu$. It follows from Joseph's result (see \cite{BGWX25}*{Proposition 2.13}) that $ {\rm AV}_{\mfr{g}_\BC}(\nu) = \mca{O}_{\rm Spr}^{\mfr{g}_\BC}(j_{W_\nu}^W \varepsilon_{W_\nu})$. This shows 
\begin{equation} \label{dBV=j}
d_{BV, G}^{(n)}(\mca{O}_{\rm reg})=\mca{O}_{\rm Spr}^{\mfr{g}_\BC}(j_{W_\nu}^W \varepsilon_{W_\nu}).
\end{equation}
Assume $p$ is big enough. Then it follows from \cite{KOW}*{Theorem 1.1} that the equality
$${\rm WF}^{\rm geo}(\Theta(\nu)) = \{d_{BV, G}^{(n)}(\mca{O}(\phi_{\St(\nu)}))\}$$
holds.
We note that for the more specific $\wt{G}$ given in \S \ref{SS:fix-cov} and \S \ref{S:dBV-exc}, the equality \eqref{dBV=j} can be checked explicitly, since the right hand side of \eqref{dBV=j} is tabulated in \cite{GLT25}.

\begin{remark}
In a recent work \cite{SYZ25} of Shan--Yan--Zhao, the authors defined a cyclotomic level map
$${\rm cl}: \mca{N}(\mbf{G}) \to \Z_{\gest 1}$$
which is an incarnation of the (true) cyclotomic level map ${\rm cl}_W: {\rm Cl}(W) \to \Z$ in view of the Kazhdan--Lusztig map $\mca{N}(\mbf{G}) \to {\rm Cl}(W)$. Here ${\rm Cl}(W)$ means the set of conjugacy classes of $W$. Also, for every $m\in {\rm Im}({\rm cl})$, it is shown that there is a unique maximal orbit
$$\mca{O}(m) \in {\rm cl}^{-1}([1,m]).$$
One can check easily that for such $m$, there is a closely-related $m^*$ such that
$$\mca{O}(m) = d_{BV,G}^{(m^*)}(\mca{O}_{\rm reg}),$$
which is thus equal to ${\rm WF}^{\rm geo}(\Theta_\nu)$.
In fact, if $G$ is simply-laced, simply-connected and almost simple, and if $\wt{G}^{(n)}$ is associated with $Q(\alpha^\vee)=1$ for any root $\alpha$, then one has $m^*=m$. For non-simply-laced types, there are slight deviations. In any case, it will be interesting to explore further connections between the map $d_{BV,G}^{(n)}$ and the topics discussed in \cite{SYZ25}.
\end{remark}

\subsection{\texorpdfstring{Speh representations of $\wt{\GL}_r$}{}} \label{SS:Speh}
In this section, we consider the BD cover $\wt{\GL}_r$ of $\GL_r$ associated with the bilinear from $B_Q$ on $Y$ such that
$$B_{Q}(e_i, e_j) = 
\begin{cases}
2\mbf{c} & \text{ if } i=j;\\
2\mbf{c}+1 & \text{ otherwise},
\end{cases}
$$
where $\mbf{c}\in \Z$ is any integer. Such a cover $\wt{\GL}_r$ is often called Kazhdan--Patterson cover of $\GL_r$.

A classification of $\Irrg(\wt{\GL}_{r}^{(n)})$ is given by Kaplan--Lapid--Zou \cite{KLZ23} in terms of the theory of multi-segments by Zelevinsky, as in the linear case. In particular, the representations $Z(\rho, [a, b])$ or $L(\rho, [a, b])$ are the building blocks. Here, $Z(\rho, [a, b])$ is the covering Speh representation, which we focus in this subsection. To elaborate, we alter notation a bit and consider $\wt{\GL}_{rk}^{(n)}$ instead of $\wt{\GL}_r^{(n)}$. Now, let $\rho \in \Irrg(\wt{\GL}_r^{(n)})_{sc}$ be a supercuspidal representation and $b-a+1=k$. Then $Z(\rho, [a,b])$ is the unique subrepresentation of
\begin{equation} \label{I-rhoab}
\Ind_{\wt{P}}^{\wt{\GL}_{rk}} (\rho\val{\cdot}^a \tilde{\otimes} \rho \val{\cdot}^{a+1} \tilde{\otimes} \cdots \tilde{\otimes} \rho \val{\cdot}^b)_{\tilde{\omega}},
\end{equation}
where $\wt{P}$ has Levi factor equal to the covering of $\GL_r \times \cdots \times \GL_r$, and the inducing representation $(\rho\val{\cdot}^a \tilde{\otimes} \rho \val{\cdot}^{a+1} \tilde{\otimes} \cdots \tilde{\otimes} \rho \val{\cdot}^b)_{\tilde{\omega}}$ represents a certain ``metaplectic tensor product" constructed from $\rho, a, b$ and certain compatible central genuine character $\tilde{\omega}$. For details, see \cite{KLZ23}.  Also, $L(\rho, [a, b])$ is the unique irreducible quotient of the representation in \eqref{I-rhoab}.

The L-parameter of $\rho$ and thus $Z(\rho, [a,b])$ depends sensitively on the Shimura lifting of $\rho$. Recall that the metaplectic correspondence given in \cite{FK86}*{\S 26} restricts to a bijection between the discrete series
$$\MC_{ds}: \Irrg(\wt{\GL}_r^{(n)})_{ds, \tilde{\omega}} \longrightarrow \Irr(\GL_r)_{ds, \omega},$$
where $\omega$ is a central character of $\GL_r$ ``compatible" with $\tilde{\omega}$; see also \cite{Zou23}*{Proposition 3.6} for details. For every $\rho \in \Irrg(\wt{\GL}_r^{(n)})_{sc, \tilde{\omega}}$, one has
\begin{equation} \label{E:MC}
\MC_{ds}(\rho) = L(\rho_0, [c, d]),
\end{equation}
where $\rho_0 \in \Irr(\GL_{r_0})_{sc}$ and $d-c+1=m$ with $r=m r_0$. Moreover, \eqref{E:MC} holds only if $m|n$.

By the classification of $\Irrg(\wt{\GL}_r^{(n)})$ in terms of $\Irrg(\wt{\GL}_r^{(n)})_{ds}$ as given in \cite{KLZ23}, one can extend $\MC_{ds}$ to the full irreducible spectrum and get a map
$$\MC: \Irrg(\wt{\GL}_r^{(n)}) \longrightarrow \Irr(\GL_r).$$
Denote by $\mca{L}$ the local Langlands correspondence for $\GL_r$, which gives the composite
$$\mca{L}_n=\mca{L}\circ \MC: \Irrg(\wt{\GL}_r^{(n)}) \longrightarrow {\rm Par}({\rm WD}_F, \GL_r(\BC)), \quad \pi \mapsto \phi_\pi.$$
Note that for $\wt{G}:=\wt{\GL}_r^{(n)}$, we have from \cite{GG18}*{\S 16.2} that
$$\wt{G}^\vee \simeq \set{(g, \lambda) \in \GL_r(\BC) \times \GL_1(\BC): \ \det(g)=\lambda^{\gcd(2\mbf{c} r + r -1, n)}} \subseteq \GL_r(\BC) \times \GL_1(\BC).$$
There is clearly the map
$$\iota: \wt{G}^\vee \to \GL_r^\vee=\GL_r(\BC)$$
given by $(g, \lambda) \mapsto g$.
Our working hypothesis is the following:

\begin{assumption} \label{as:dual}
For every $\pi \in \Irrg(\wt{\GL}_r^{(n)})$, its parameter $\phi_\pi=\mca{L}_n(\pi): {\rm WD}_F \to \GL_r(\BC)$ factors through $\iota$ and thus is actually $\phi_\pi: {\rm WD}_F \to \wt{G}^\vee$.
\end{assumption}
Note that if $\gcd(2\mbf{c}r + r-1, n)=1$, then $\wt{G}^\vee \simeq \GL_r(\BC)$ and thus the hypothesis \ref{as:dual} is trivially satisfied in this case.

\begin{defn}
Under the working hypothesis \ref{as:dual}, we call 
$$\mca{L}_n: \Irrg(\wt{\GL}_r^{(n)}) \longrightarrow {\rm Par}({\rm WD}_F, \wt{G}^\vee), \quad \pi \mapsto \phi_\pi$$
the local Langlands correspondence for $\Irrg(\wt{\GL}_r^{(n)})$.
\end{defn}

Henceforth, we assume the hypothesis in \ref{as:dual}. Consider $\pi:=Z(\rho, [a, b]) \in \Irrg(\wt{\GL}_{rk}^{(n)}), b-a+1=k$, we have 
$$\AZ(\pi) = L(\rho, [a, b]).$$
Moreover, assume $\MC(\rho) = L(\rho_0, [c, d])$ with $\rho_0 \in \Irr(\GL_{r_0})_{sc}, d-c+1=m$. Then we get that
$$\iota \circ \phi_{\AZ(\pi)} = \phi_{\rho_0} \val{\cdot}^{\frac{mk-1}{2}} \boxtimes S_{mk}: W_F \times \SL_2 \to \GL_{rk},$$
where $S_{mk}$ is the unique irreducible representation of $\SL_2$ of dimension $mk$ and $\phi_{\rho_0}$ is the L-parameter associated with $\rho_0$.
This gives that
\begin{equation}
\mca{O}(\phi_{\AZ(\pi)}) = ((mk)^{r_0}).
\end{equation}

Since $m|n$, we can write $n=m\cdot n_0$. One has
$$k=cn_0 +d \text{ with } 0 \lest d < n_0.$$
An easy computation gives that 
\begin{equation} \label{speh-o}
d_{BV}^{(n)}( \mca{O}(\phi_{\AZ(\pi)}) ) = d_{BV}^{(n)}( (mk)^{r_0} ) = ((nr_0)^c, r_0md).
\end{equation}

\begin{thm} \label{T:Speh}
Assume the working hypothesis in \ref{as:dual}. Also assume $\gcd(p, n)=1$. Consider the genuine Speh representation $\pi:=Z(\rho, [a, b])$ of $\wt{\GL}_{rk}^{(n)}$ with $\MC(\rho) = L(\rho_0, [c, d])$ as above. 
Then 
$${\rm WF}^{\rm geo}(\pi) = \{d_{BV}^{(n)}( \mca{O}(\phi_{\AZ(\pi)}) )\};$$
in particular, Conjecture \ref{Conj} (i) holds regarding the geometric wavefront set of genuine Speh representations.
\end{thm}
\begin{proof}
Set $\mfr{p}:= ((nr_0)^c, rd)$, a partition of $rk$. First, we show that the non-vanishing of the degenerate Whittaker model $\Wh_\psi(\pi; \mfr{p})$ of $\pi$ with respect to $\mfr{p}$. However, by results of Gomez--Gourevich--Sahi \cite{GGS17}, we have a natural injection
\begin{equation} \label{E:surj}
 {\rm Wh}_\psi({\rm Jac}_\mfr{p}(\pi)) \into \Wh_\psi(\pi; \mfr{p}),
\end{equation}
where ${\rm Jac}_\mfr{p}(\pi) \in \Irrg(\wt{M_\mfr{p}}^{(n)})$ is the Jacquet module of $\pi$ with respect to the Levi subgroup $M_\mfr{p}$ associated with the partition $\mfr{p}$, and also ${\rm Wh}_\psi(\sigma)$ denote the $\psi$-Whittaker space of $\sigma \in \Irrg(\wt{M_\mfr{p}})$. Note that $r|(nr_0)$ and also $r|(r_0md)$, and it follows from \cite{KLZ23}*{Proposition 7.2} that 
$${\rm Jac}_\mfr{p}(\pi) \simeq \tilde{\bigotimes}_{i=1}^{c+1} Z(\rho, [a_i, b_i]),$$
i.e., is a metaplectic tensor product of Speh representations $Z(\rho, [a_i, b_i])$ of $\wt{\GL}_{nr_0}^{(n)}$ and $\wt{\GL}_{r_0md}^{(n)}$. Here $b_i - a_i +1 = n_0$ for $1\lest i \lest c$ and $b_{c+1} - a_{c+1} + 1 = d$. Now it follows from \cite{Zou25}*{(10.2)} and the proof in \cite{Zou23}*{Corollary 4.11} that 
$$\dim {\rm Wh}_\psi(Z(\rho, [a_i, b_i]) )=
\begin{cases}
\binom{n/m}{n_0} = 1 & \text{ if } 1\lest i \lest c,\\
\frac{1}{\gcd(n, 2rd \mbf{c} - rd+1)} \binom{n/m}{d} & \text{ if } i =c+1.
\end{cases} 
$$
 This shows the non-vanishing of $\Wh_\psi(\pi; \mfr{p})$. We remark in passing that the above formula for $\dim {\rm Wh}_\psi(Z(\rho, [a_i, b_i]) )$ is already given in \cite{Zou23}*{Corollary 4.11},
 which however is conditional on Conjecture 4.6 there. 
 On the other hand, the Whittaker dimension formula for essentially square-integrable representations needed in the proof of \cite{Zou23}*{Corollary 4.11} is established unconditionally in \cite{Zou25}*{Proposition 10.1, (10.2)}. Hence, the formula we stated above is now unconditional.

Now to show 
$${\rm WF}^{\rm geo}(\pi) = \set{ ((nr_0)^c, rd) },$$
we give two proofs. The first approach is to show the vanishing of $\Wh_\psi(\pi; \mfr{p}')$  with respect to an orbit $\mfr{p}'= (p_1, p_2, \ldots, p_s)$ that is bigger or not compatible with $ ((nr_0)^c, rd)$. Then clearly, one must have $p_1> nr_0$. Moreover, in order to show $\Wh_\psi(\pi; \mfr{p}')=0$ for all such $\mfr{p}'$, it suffices to prove the vanishing of  $\Wh_\psi(\pi; \mfr{p}')$ for $\mfr{p}'=(s, 1^{rk - s})$ with $s > nr_0$, see \cite{GRS03}*{Lemma 2.6}. For such $\mfr{p}$, we can argue parallel as in \cite{Cai19}*{\S 6.3} for theta representations, and apply Zou's result \cites{Zou23, Zou25} at the reduced stage to give the vanishing of  $\Wh_\psi(\pi; \mfr{p}')$.

Alternatively, we can argue as follows. To show ${\rm WF}^{\rm geo}(\pi) = \set{ ((nr_0)^c, rd) }$, it suffices to show that any $\mfr{p}' = (p_1, \ldots, p_s)$ with $p_1 > nr_0$ can not lie in ${\rm WF}^{\rm geo}(\pi)$. Suppose not, and we pick such a $\mfr{p}' \in {\rm WF}^{\rm geo}(\pi)$. It then follows from \cite{GGS21}*{Theorem 1.5} that the surjection in \eqref{E:surj} is in fact an isomorphism
$$\Wh_\psi(\pi; \mfr{p}')  \simeq  {\rm Wh}_\psi({\rm Jac}_{\mfr{p}'}(\pi)),$$
which is nonzero by the assumption that $\mfr{p}' \in {\rm WF}^{\rm geo}(\pi)$. However, it follows from \cite{KLZ23}*{Proposition 7.2} that ${\rm Jac}_{\mfr{p}'}(\pi) = 0$ (if $r\nmid p_i$ for some $i$) or is of the form
 $${\rm Jac}_{\mfr{p}'}(\pi) \simeq Z(\rho, [a_1, b_1]) \tilde{\otimes} \cdots,$$
 where $b_1- a_1 + 1=n_1$ with $p_1 = r n_1 > n r_0$. This shows that
$$n_1 > n_0.$$ 
Using this, it again follows from \cite{Zou25}*{(10.2)} and the proof in \cite{Zou23}*{Corollary 4.11} that $\dim \Wh_\psi(Z(\rho, [a_1, b_1]))=0$. 
Thus, we have ${\rm Wh}_\psi({\rm Jac}_{\mfr{p}'}(\pi))=0$, which is a contradiction. 
\end{proof}

In the next section, we develop a reduction argument for Conjecture \ref{Conj} (i) (see Theorem \ref{T:red}). The above theorem provides the base case for proving Conjecture \ref{Conj} (i) for general genuine representations of Kazhdan-Patterson cover satisfying Working Hypothesis \ref{as:dual}. See Corollary \ref{C:GLlest}.

\section{Reduction of the conjecture for general types}
In this section, we state two natural desiderata for local Langlands correspondence for covering groups. Then, we prove that Conjecture \ref{MC} can be reduced to the case of  discrete representations.

\subsection{Desiderata for local Langlands correspondence}
A local Langlands correspondence for a Brylinski--Deligne cover $\overline{G}$ of $G$ is a map
\begin{align*}
    \mca{L}_{\wt{G}}: \Irrg(\wt{G}) &\longrightarrow \Phi({}^L \wt{G}),\\
    \pi & \mapsto \phi_{\pi}.
\end{align*}
Let ${\rm Lev}(\wt{G})$ be the set of standard Levi subgroups of $\wt{G}$. We shall require that the system of maps $\set{\mca{L}_{\wt{M}}: \ \wt{M} \in {\rm Lev}(\wt{G})}$ are compatible. To explicate this, we first recall the Langlands classification for central covering groups.

\begin{thm}[{\cite{BJ13}*{Theorem 1.1}}]\label{T: LC for rep}
    There is a one-to-one correspondence between irreducible admissible genuine representations of $\wt{G}$ and the collection of standard triples $\{ (\wt{P}, \pi_{\rm temp}, \nu)\}$ via Langlands quotient. Here 
    \begin{itemize}
        \item $\wt{P}=\wt{M} N$ is a standard parabolic subgroup of $\wt{G}$, and
        \item $\pi_{\rm temp}$ is an irreducible tempered representation of $\wt{M}$, and
        \item $\nu$ is an unramified character of $\wt{M}$ (or $M$) that is positive with respect to $P$.
    \end{itemize}
    We shall write $ \pi \leftrightarrow (\wt{P}, \pi_{\rm temp}, \nu)$ if they correspond to each other.
\end{thm}

As mentioned in \S \ref{SS:WFconj}, the dual groups of $\wt{G}$ and its Levi subgroups $\wt{M}$ are reductive groups, and we have natural functorial inclusion $\wt{M}^\vee \into \wt{G}^\vee$. One also has the natural inclusion ${}^L\wt{M} \into {}^L \wt{G}$. 
The $\nu$ as stipulated in Theorem \ref{T: LC for rep} corresponds to an element (by abuse of notation)
$$\nu \in Z({}^L M)_{>0}:=\Hom(Y_M/Y_M^{sc}, \R_{>0}),$$
where $Y_M$ is the character lattice of the Langlands dual group $M^\vee$ and $Y_M^{sc}$ the root lattice of $M^\vee$.
Note that there is a natural map 
$$\iota: Z({}^L M) \longrightarrow Z({}^L \wt{M})$$
where $Z({}^L \wt{M}):=\Hom(Y_{M,Q,n}/Y_{M,Q,n}^{sc}, \BC)$, see \cite{GG18}*{\S 15.2} for the notation and explanation of $\iota$. Here $\iota$ is essentially induced from the inclusion $Y_{M,Q,n} \subseteq Y_M$ (of finite index) and is not injective in general. However, the restriction of $\iota$ to the positive part
$$\iota_+: Z({}^L M)_{>0} \into Z({}^L \wt{M})_{>0},$$
where $Z({}^L \wt{M})_{>0}:=\Hom(Y_{M,Q,n}/Y_{M,Q,n}^{sc}, \R_{>0})$, is a natural isomorphism of groups. In view of this, the Langlands classification of $L$-parameter of $\wt{G}$, which for linear groups is given in \cite{SZ18}, can be applied verbatim to $\wt{G}$ here. This gives the following:

\begin{thm}\label{thm LC for L-par}
      There is a one-to-one correspondence between the set of $L$-parameters $\phi$ of $\wt{G}$ and the set of standard triples $\{ (\wt{P}, \phi_{\rm temp}, \nu)\}$ via $\phi= ({}^L \wt{M} \hookrightarrow {}^L \wt{G}) \circ (\phi_{\rm temp} \otimes \iota_+(\nu))$ (up to $\wt{G}^{\vee}$-conjugacy).      Here 
    \begin{itemize}
        \item $\wt{P}=\wt{M} N$ is a standard parabolic subgroup of $\wt{G}$, and
        \item $\phi_{\rm temp}$ is a tempered $L$-parameter of $\wt{M}$, and
        \item $\nu$ is an unramified linear character of $\wt{M}$ (or $M$) that is positive with respect to $P$, regarded as an unramified central cocharacter of $\wt{M}^{\vee}$ via $\iota_+$.
    \end{itemize}
    We shall write $ \phi \leftrightarrow (\wt{P}, \phi_{\rm temp}, \nu)$ if they correspond to each other.
\end{thm}

We require that $\mca{L}_{\wt{G}}$ and $\mca{L}_{\wt{M}}$ are compatible with respect to the Langlands classification and induction of tempered representations.

\begin{assumption}\label{assu LLC}\ 
   \begin{itemize}
       \item [(1)]Suppose that $\mca{L}_{\wt{G}}(\pi)=\phi$ and    $ \pi \leftrightarrow (\wt{P}_{\pi}, \pi_{\rm temp}, \nu_{\pi})$ and $ \phi \leftrightarrow (\wt{P}_{\phi}, \phi_{\rm temp}, \nu_{\phi})$. Then $\wt{P}_{\pi}=\wt{P}_{\phi}$, $\nu_{\pi}=\nu_{\phi}$, and $\mca{L}_{\wt{M}}(\pi_{\rm temp})= \phi_{\rm temp}$.
       \item[(2)] Let $\sigma$ be a genuine (unitary) tempered representation of $\wt{M} \subset \wt{P}$. Suppose that $\pi$ is an irreducible subquotient of the parabolic induction $\Ind_{\wt{P}}^{\wt{G}} (\sigma)$. Then
       \[ \mca{L}_{\wt{G}}(\pi)= \phi= ({}^L \wt{M} \hookrightarrow {}^L \wt{G}) \circ \mca{L}_{\wt{M}}(\sigma). \]
   \end{itemize} 
\end{assumption}

\subsection{Reduction of the conjecture}
In this subsection, we reduce the inequality in Conjecture \ref{Conj} (i) to discrete representations based on the compatibility of $d_{BV,G}^{(n)}$ with induction and saturation, given as in Theorem \ref{T:ABCD} (iii) and Theorem \ref{T:EFG} (ii). The argument is parallel to that in the linear case, see \cite{HLLS24}.

\begin{defn}
  An irreducible representation $\pi$ of $\wt{G}$ is called anti-tempered (resp. anti-discrete) if its Aubert--Zelevinsky dual ${\rm AZ}(\pi)$ is tempered (resp. discrete).
\end{defn}

\begin{thm} \label{T:red}
Let $\wt{G}$ be a Brylinski--Deligne cover of $G$ given in \S \ref{SS:fix-cov} and \S \ref{S:dBV-exc}. Assume the Working Hypothesis \ref{assu LLC}. Consider the inequality 
\begin{equation} \label{E:ineq2}
{\rm WF}^{\rm geo}(\pi) \lest d_{BV, G}^{(n)}(\OO(\phi_{\AZ(\pi)}))
\end{equation}
arising from \eqref{E:ineq}. The following are equivalent.
\begin{itemize}
    \item [(a)] The inequality \eqref{E:ineq2} holds for all genuine representations of every Levi subgroup of $\wt{G}$.
    \item [(b)] The inequality \eqref{E:ineq2} holds for all genuine anti-tempered representations of every Levi subgroup of $\wt{G}$.
    \item [(c)] The inequality \eqref{E:ineq2} holds for all genuine anti-discrete representations of every Levi subgroup of $\wt{G}$.
\end{itemize}
\end{thm}
\begin{proof}
    It is clear that Part (a) implies Part (b), and Part (b) implies Part (c). We first prove that Part (b) implies Part (a) as well.

    Let $\pi$ be any genuine representation of $\wt{G}$. By Theorem \ref{T: LC for rep}, we may write $\AZ(\pi) \leftrightarrow (\wt{P}, \AZ(\pi)_{\rm temp}, \nu)$. Then, we have (here $\tau \lest \Pi$ means that $\tau$ is a subquotient of $\Pi$)
    \[ \pi = \AZ(\AZ(\pi)) \lest \AZ(\Ind_{\wt{P}}^{\wt{G}} (\AZ(\pi))_{\rm temp} \otimes \nu )= \Ind_{\wt{P}}^{\wt{G}}\ \AZ(\AZ(\pi)_{\rm temp}) \otimes \nu.\]
    Note that $\pi^\sharp:= \AZ(\AZ(\pi)_{\rm temp})$ is anti-tempered. Hence, we have an upper bound of its wavefront set from Part (b). Also, 
    $${\rm WF}^{\rm geo}(\pi^\sharp)={\rm WF}^{\rm geo}(\pi^\sharp\otimes \nu)$$ since the unramified character $\nu$ is trivial in a small neighborhood of identity. Therefore, we have an upper bound for the wavefront set of $\pi$ from the induced orbit
 $$\begin{aligned}
  {\rm WF}^{\rm geo}(\pi)  & \lest {\rm WF}^{\rm geo}(\Ind_{\wt{P}}^{\wt{G}}\ \pi^\sharp \otimes \nu)  \\
   & =\Ind_\mbf{M}^\mbf{G} ({\rm WF}^{\rm geo}(\pi^\sharp) ) \\
   & \lest \Ind_\mbf{M}^\mbf{G} (d_{BV,M}^{(n)} (\OO(\phi_{ \mathsf{AZ}(\pi^\sharp)}))).
  \end{aligned}$$
It then follows from Theorem \ref{T:ABCD} (iii) and Theorem \ref{T:EFG} (ii) that
    \[ {\rm WF}^{\rm geo}(\pi) \lest d_{BV,G}^{(n)} ( \textrm{Sat}_{\wt{M}^\vee}^{\wt{G}^\vee}(\OO(\phi_{\mathsf{AZ}(\pi^\sharp)}) )) \lest d_{BV, G}^{(n)} ( \textrm{Sat}_{\wt{M}^\vee}^{\wt{G}^\vee}(\OO(\phi_{ \AZ(\pi)_{\rm temp}}) )). \]
 On the other hand, by Working Hypothesis \ref{assu LLC} (1), we see that 
    \[ \phi_{\AZ(\pi)}= ({}^L \wt{M} \hookrightarrow {}^L\wt{G}) \circ \phi_{\AZ(\pi)_{\rm temp}}\otimes \iota_+(\nu).\]
In particular, $\OO(\phi_{\AZ(\pi)})= \textrm{Sat}_{\wt{M}^\vee}^{\wt{G}^\vee}(\OO(\phi_{\AZ(\pi)_{\rm temp}}))$. This gives
$${\rm WF}^{\rm geo}(\pi) \lest d_{BV, G}^{(n)}(\OO(\phi_{\AZ(\pi)}))$$
and completes the verification that Part (b) implies Part (a).

The verification that Part (c) implies Part (b) is similar, but using the fact that every tempered representation can be realized as an irreducible subquotient of an induction of a discrete series representation (cf. \cite{Wal03}*{Proposition III.4.1}) and Working Hypothesis \ref{assu LLC} (2) instead, which we omit the details.
\end{proof}

\begin{cor} \label{C:GLlest}
Let $\wt{\GL}_r^{(n)}$ be the Kazhdan--Patterson cover given in \S \ref{SS:Speh}. Assume $p \nmid n$ and Working Hypothesis \ref{as:dual}. Then we have
$${\rm WF}^{\rm geo}(\pi) \lest d_{BV, G}^{(n)}(\OO(\phi_{\AZ(\pi)}))$$
for every $\pi \in \Irrg(\wt{\GL}_r^{(n)})$.
\end{cor}
\begin{proof}
By the classification of $\Irrg(\wt{\GL}_r^{(n)})$ in \cite{KLZ23}, an irreducible genuine representation $\pi$ is anti-discrete series if and only if it is a twist of $Z(\rho, [a, b])$, as given in \S \ref{SS:Speh}. The result then follows from combining Theorem \ref{T:Speh} and Theorem \ref{T:red}.
\end{proof}
\appendix

\section{\texorpdfstring{Data of $d_{BV, G}^{(n)}$ for exceptional $G$}{}} \label{A:dBVexc}
We tabulate the data for $d_{BV, G}^{(n)}$ explicitly for $G=G_2, F_4, E_r, 6 \lest r \lest 8$. For simplicity of notation, we write
$$d_{G}^{(n)}:=d_{BV, G}^{(n)}.$$
Note also $d_{G}^{(n)}(\set{0}) = \mca{O}_{\rm reg}$ for all $G$.

\subsection{\texorpdfstring{Type $G_2$}{}}
$$
d_{G_2}^{(n)}(A_1) = 
\begin{cases}
G_2(a_1) & \text{ if } n=1,\\
G_2 & \text{ if } n \gest 2;
\end{cases} \quad
d_{G_2}^{(n)}(\tilde{A}_1) = 
\begin{cases}
G_2(a_1) &  \text{ if } n=1,3,\\
G_2 &  \text{ if } n\in \set{2}\cup \N_{\gest 4}.
\end{cases}
$$

$$
d_{G_2}^{(n)}(G_2(a_1)) = 
\begin{cases}
G_2(a_1) & \text{ if } n=1,2,3,\\
G_2 & \text{ if } n \gest 4;
\end{cases} \quad
d_{G_2}^{(n)}(G_2) = 
\begin{cases}
0 &  \text{ if } n=1,\\
\tilde{A}_1 & \text{ if } n=2,\\
A_1 & \text{ if } n=3,\\
G_2(a_1) & \text{ if } n=4, 5, 6, 9,\\
G_2 & \text{ if } n\in \set{7, 8} \cup \N_{\gest 10}.
\end{cases}
$$

\subsection{\texorpdfstring{Type $F_4$}{}}
$$
d_{F_4}^{(n)}(A_1) = 
\begin{cases}
F_4(a_1) & \text{ if } n=1,\\
F_4 & \text{ if } n \gest 2;
\end{cases} \quad
d_{F_4}^{(n)}(\tilde{A}_1) = 
\begin{cases}
F_4(a_1) &  \text{ if } n=1,2,\\
F_4 &  \text{ if } n \gest 3.
\end{cases}
$$

$$
d_{F_4}^{(n)}(A_1 + \tilde{A}_1) = 
\begin{cases}
F_4(a_2) & \text{ if } n=1,\\
F_4(a_1) & \text{ if } n=2,\\
F_4 & \text{ if } n \gest 3;
\end{cases} \quad
d_{F_4}^{(n)}(A_2) = 
\begin{cases}
B_3 & \text{ if } n=1,\\
F_4(a_1) &  \text{ if } n=2,\\
F_4 &  \text{ if } n \gest 3.
\end{cases}
$$

$$
d_{F_4}^{(n)}(\tilde{A}_2) = 
\begin{cases}
C_3 & \text{ if } n=1,\\
B_3 & \text{ if } n=2,\\
F_4(a_1) & \text{ if } n=4,\\
F_4 & \text{ if } n \in \set{3} \cup \N_{\gest 5};
\end{cases} \quad
d_{F_4}^{(n)}(A_2 + \tilde{A}_1) = 
\begin{cases}
F_4(a_3) &  \text{ if } n=1,\\
F_4(a_2) & \text{ if } n=2,\\
F_4 &  \text{ if } n \gest 3.
\end{cases}
$$

$$
d_{F_4}^{(n)}(B_2) = 
\begin{cases}
F_4(a_3) & \text{ if } n=1,\\
F_4(a_2) & \text{ if } n=2,\\
F_4(a_1) & \text{ if } n =3, 4,\\
F_4 & \text{ if } n \gest 5;
\end{cases} \quad
d_{F_4}^{(n)}(\tilde{A}_2 + A_1) = 
\begin{cases}
F_4(a_3) &  \text{ if } n=1,\\
B_3 & \text{ if } n=2,\\
F_4(a_1) & \text{ if } n=4,\\
F_4 &  \text{ if } n\in \set{3} \cup \N_{\gest 5}.
\end{cases}
$$

$$
d_{F_4}^{(n)}(C_3(a_1)) = 
\begin{cases}
F_4(a_3) &  \text{ if } n=1,\\
B_3 & \text{ if } n=2,\\
F_4(a_1) & \text{ if } n=3, 4,\\
F_4 &  \text{ if } n \gest 5.
\end{cases} \quad
d_{F_4}^{(n)}(F_4(a_3)) = 
\begin{cases}
F_4(a_3) &  \text{ if } n=1,2,\\
F_4(a_1) & \text{ if } n=3,4,\\
F_4 &  \text{ if } n \gest 5.
\end{cases}
$$

$$
d_{F_4}^{(n)}(B_3) = 
\begin{cases}
A_2 &  \text{ if } n=1,\\
C_3(a_1) & \text{ if } n=2,\\
F_4(a_2) & \text{ if } n=3, 4,\\
F_4(a_1) &  \text{ if } n =5, 6,\\
F_4 & \text{ if } n\gest 7;
\end{cases} \quad
d_{F_4}^{(n)}(C_3) = 
\begin{cases}
\tilde{A}_2 &  \text{ if } n=1,\\
B_2 & \text{ if } n=2,\\
F_4(a_2) &  \text{ if } n=3,\\
B_3 &  \text{ if } n=4,\\
F_4(a_1) & \text{ if } n=5, 6,8,\\
F_4 & \text{ if } n \in \set{7} \cup \N_{\gest 9}.
\end{cases}
$$

$$
d_{F_4}^{(n)}(F_4(a_2)) = 
\begin{cases}
A_1 + \tilde{A}_1 &  \text{ if } n=1,\\
B_2 & \text{ if } n=2,\\
F_4(a_2) & \text{ if } n=3,\\
B_3 & \text{ if } n=4,\\
F_4(a_1) & \text{ if } n=5, 6,8,\\
F_4 & \text{ if } n \in \set{7} \cup \N_{\gest 9};
\end{cases} \quad
d_{F_4}^{(n)}(F_4(a_1)) = 
\begin{cases}
\tilde{A}_1 &  \text{ if } n=1,\\
A_2 & \text{ if } n=2,\\
F_4(a_3) &  \text{ if } n=3,4,\\
F_4(a_2) & \text{ if } n=6,\\
F_4(a_1) & \text{ if } n=5, 7, 8, 10,\\
F_4 & \text{ if } n\in \set{9} \cup \N_{\gest 11}.
\end{cases}
$$

$$
d_{F_4}^{(n)}(F_4) = 
\begin{cases}
0 &  \text{ if } n=1,\\
A_1 & \text{ if } n=2,\\
\tilde{A}_2 + A_1 &  \text{ if } n=3,\\
A_2 + \tilde{A}_1 & \text{ if } n=4,\\
F_4(a_3) & \text{ if } n=5, 6,\\
B_3 & \text{ if } n=8,\\
F_4(a_2) & \text{ if } n=7, 10,\\
F_4(a_1) & \text{ if } n=9,11, 12, 14, 16,\\
F_4 & \text{ if } n\in \set{13, 15} \cup \N_{\gest 17}.
\end{cases}
$$

\subsection{\texorpdfstring{Type $E_6$}{}}

We have
$$
d_{E_6}^{(n)}(A_1) = 
\begin{cases}
E_6(a_1) & \text{ if } n=1,\\
E_6 & \text{ if } n \gest 2;
\end{cases} \quad
d_{E_6}^{(n)}(2A_1) = 
\begin{cases}
D_5 &  \text{ if } n=1,\\
E_6 &  \text{ if } n \gest 2.
\end{cases}
$$

$$
d_{E_6}^{(n)}(3A_1) = 
\begin{cases}
E_6(a_3) &  \text{ if } n=1,\\
E_6 &  \text{ if } n \gest 2;
\end{cases} \quad
d_{E_6}^{(n)}(A_2) = 
\begin{cases}
E_6(a_3) &  \text{ if } n=1,\\
E_6(a_1) &  \text{ if } n=2,\\
E_6 & \text{ if } n\gest 3.
\end{cases}
$$

$$
d_{E_6}^{(n)}(A_2+A_1) = 
\begin{cases}
D_5(a_1) &  \text{ if } n=1,\\
E_6(a_1) &  \text{ if } n= 2,\\
E_6 & \text{ if } n\gest 3.
\end{cases} \quad
d_{E_6}^{(n)}(2A_2) = 
\begin{cases}
D_4 &  \text{ if } n=1,\\
D_5 &  \text{ if } n=2,\\
E_6 & \text{ if } n\gest 3.
\end{cases}
$$

$$
d_{E_6}^{(n)}(A_2+2A_1) = 
\begin{cases}
A_4 + A_1 &  \text{ if } n=1,\\
E_6(a_1) &  \text{ if } n= 2,\\
E_6 & \text{ if } n\gest 3.
\end{cases} \quad
d_{E_6}^{(n)}(A_3) = 
\begin{cases}
A_4 &  \text{ if } n=1,\\
D_5 &  \text{ if } n=2,\\
E_6(a_1) & \text{ if } n = 3,\\
E_6 & \text{ if } n\gest 4.
\end{cases}
$$

$$
d_{E_6}^{(n)}(2A_2+A_1) = 
\begin{cases}
D_4(a_1) &  \text{ if } n=1,\\
D_5 &  \text{ if } n= 2,\\
E_6 & \text{ if } n\gest 3.
\end{cases} \quad
d_{E_6}^{(n)}(A_3 + A_1) = 
\begin{cases}
D_4(a_1) &  \text{ if } n=1,\\
D_5 &  \text{ if } n=2,\\
E_6(a_1) & \text{ if } n= 3,\\
E_6 & \text{ if } n\gest 4.
\end{cases}
$$

$$
d_{E_6}^{(n)}(D_4(a_1)) = 
\begin{cases}
D_4(a_1) &  \text{ if } n=1,\\
E_6(a_3) &  \text{ if } n= 2,\\
E_6(a_1) & \text{ if } n=3,\\
E_6 & \text{ if } n\gest 4.
\end{cases} \quad
d_{E_6}^{(n)}(A_4) = 
\begin{cases}
A_3 &  \text{ if } n=1,\\
D_5(a_1) &  \text{ if } n=2,\\
D_5 & \text{ if } n=3,\\
E_6(a_1) & \text{ if } n=4, \\
E_6 & \text{ if } n\gest 5.
\end{cases}
$$

$$
d_{E_6}^{(n)}(D_4) = 
\begin{cases}
2A_2 &  \text{ if } n=1,\\
A_5 &  \text{ if } n= 2,\\
E_6(a_3) & \text{ if } n=3,\\
E_6(a_1) & \text{ if } n=4, 5\\
E_6 & \text{ if } n\gest 6.
\end{cases} \quad
d_{E_6}^{(n)}(A_4 + A_1) = 
\begin{cases}
A_2 + 2A_1 &  \text{ if } n=1,\\
D_5(a_1) &  \text{ if } n=2,\\
D_5 & \text{ if } n=3,\\
E_6(a_1) & \text{ if } n=4, \\
E_6 & \text{ if } n\gest 5.
\end{cases}
$$

$$
d_{E_6}^{(n)}(A_5) = 
\begin{cases}
A_2 &  \text{ if } n=1,\\
D_4 &  \text{ if } n= 2,\\
E_6(a_3) & \text{ if } n=3,\\
D_5 & \text{ if } n=4,\\
E_6(a_1) & \text{ if } n=5,\\
E_6 & \text{ if } n\gest 6.
\end{cases} \quad
d_{E_6}^{(n)}(D_5(a_1)) = 
\begin{cases}
A_2 + A_1 &  \text{ if } n=1,\\
A_4 + A_1 &  \text{ if } n=2,\\
E_6(a_3) & \text{ if } n=3,\\
E_6(a_1) & \text{ if } n=4, 5, \\
E_6 & \text{ if } n\gest 6.
\end{cases}
$$

$$
d_{E_6}^{(n)}(E_6(a_3)) = 
\begin{cases}
A_2 &  \text{ if } n=1,\\
D_4(a_1) &  \text{ if } n= 2,\\
E_6(a_3) & \text{ if } n=3,\\
D_5 & \text{ if } n= 4,\\
E_6(a_1) & \text{ if } n=5,\\
E_6 & \text{ if } n\gest 6.
\end{cases} \quad
d_{E_6}^{(n)}(D_5) = 
\begin{cases}
2A_1 &  \text{ if } n=1,\\
A_3 + A_1 &  \text{ if } n=2,\\
A_4 + A_1 & \text{ if } n=3,\\
E_6(a_3) & \text{ if } n=4, \\
D_5 & \text{ if } n= 5,\\
E_6(a_1) & \text{ if } n=6,7,\\
E_6 & \text{ if } n\gest 8.
\end{cases}
$$

$$
d_{E_6}^{(n)}(E_6(a_1)) = 
\begin{cases}
A_1 &  \text{ if } n=1,\\
A_2 + 2A_1 &  \text{ if } n= 2,\\
D_4(a_1) & \text{ if } n=3,\\
D_5(a_1) & \text{ if } n= 4,\\
E_6(a_3) & \text{ if } n=5,\\
D_5 & \text{ if } n = 6,\\
E_6(a_1) & \text{ if } n = 7, 8,\\
E_6 & \text{ if } n \gest 9.\\
\end{cases} \quad
d_{E_6}^{(n)}(E_6) = 
\begin{cases}
0 &  \text{ if } n=1,\\
3A_1 &  \text{ if } n=2,\\
2A_2 + A_1 & \text{ if } n=3,\\
D_4(a_1) & \text{ if } n=4, \\
A_4 + A_1 & \text{ if } n= 5,\\
E_6(a_3) & \text{ if } n=6,7,\\
D_5 & \text{ if } n=8,\\
E_6(a_1) & \text{ if } n=9, 10, 11,\\
E_6 & \text{ if } n\gest 12.
\end{cases}
$$

\subsection{\texorpdfstring{Type $E_7$}{}}
For $E_7$, the covering Barbasch--Vogan map is given as follows:
$$
d_{E_7}^{(n)}(A_1) = 
\begin{cases}
E_7(a_1) &  \text{ if } n=1,\\
E_7 &  \text{ if } n\gest 2.
\end{cases} \quad
d_{E_7}^{(n)}(2A_1) = 
\begin{cases}
E_7(a_2) &  \text{ if } n=1,\\
E_7 &  \text{ if } n\gest 2.
\end{cases}
$$
$$
d_{E_7}^{(n)}((3A_1)'') = 
\begin{cases}
E_6 &  \text{ if } n=1,\\
E_7 &  \text{ if } n\gest 2.
\end{cases} \quad
d_{E_7}^{(n)}((3A_1)') = 
\begin{cases}
E_7(a_3) &  \text{ if } n=1,\\
E_7 &  \text{ if } n\gest 2.
\end{cases}
$$
$$
d_{E_7}^{(n)}(A_2) = 
\begin{cases}
E_7(a_3) &  \text{ if } n=1,\\
E_7(a_1) &  \text{ if } n =2,\\
E_7 & \text{ if } n \gest 3.
\end{cases} \quad
d_{E_7}^{(n)}(4A_1) = 
\begin{cases}
E_6(a_1) &  \text{ if } n=1,\\
E_7 &  \text{ if } n\gest 2.
\end{cases}
$$
$$
d_{E_7}^{(n)}(A_2 + A_1) = 
\begin{cases}
E_6(a_1) &  \text{ if } n=1,\\
E_7(a_1) &  \text{ if } n= 2,\\
E_7 & \text{ if } n\gest 3.
\end{cases} \quad
d_{E_7}^{(n)}(A_2 + 2A_1) = 
\begin{cases}
E_7(a_4) &  \text{ if } n=1,\\
E_7(a_1) &  \text{ if } n=2,\\
E_7 & \text{ if } n \gest 3.
\end{cases}
$$
$$
d_{E_7}^{(n)}(A_3) = 
\begin{cases}
D_6(a_1) &  \text{ if } n=1,\\
E_7(a_2) &  \text{ if } n= 2,\\
E_7(a_1) & \text{ if } n= 3,\\
E_7 & \text{ if } n \gest 4.
\end{cases} \quad
d_{E_7}^{(n)}(2A_2) = 
\begin{cases}
D_5 + A_1 &  \text{ if } n=1,\\
E_7(a_2) &  \text{ if } n=2,\\
E_7 & \text{ if } n \gest 3.
\end{cases}
$$
$$
d_{E_7}^{(n)}(A_2 + 3A_1) = 
\begin{cases}
A_6 &  \text{ if } n=1,\\
E_7(a_1) &  \text{ if } n= 2,\\
E_7 & \text{ if } n\gest 3.
\end{cases} \quad
d_{E_7}^{(n)}((A_3 + A_1)'') = 
\begin{cases}
D_5 &  \text{ if } n=1,\\
E_7(a_2) &  \text{ if } n=2,\\
E_7(a_1) &  \text{ if } n=3,\\
E_7 & \text{ if } n \gest 4.
\end{cases}
$$
$$
d_{E_7}^{(n)}(2A_2 + A_1) = 
\begin{cases}
E_7(a_5) &  \text{ if } n=1,\\
E_7(a_2) &  \text{ if } n= 2,\\
E_7 & \text{ if } n\gest 3.
\end{cases} \quad
d_{E_7}^{(n)}((A_3 + A_1)') = 
\begin{cases}
E_7(a_5) &  \text{ if } n=1,\\
E_7(a_2) &  \text{ if } n=2,\\
E_7(a_1) &  \text{ if } n=3,\\
E_7 & \text{ if } n \gest 4.
\end{cases}
$$
$$
d_{E_7}^{(n)}(D_4(a_1)) = 
\begin{cases}
E_7(a_5) &  \text{ if } n=1,\\
E_7(a_3) &  \text{ if } n= 2,\\
E_7(a_1) & \text{ if } n=3,\\
E_7 & \text{ if } n\gest 4.
\end{cases} \quad
d_{E_7}^{(n)}(A_3 + 2A_1) = 
\begin{cases}
E_6(a_3) &  \text{ if } n=1,\\
E_7(a_2) &  \text{ if } n=2,\\
E_7(a_1) &  \text{ if } n=3,\\
E_7 & \text{ if } n \gest 4.
\end{cases}
$$
$$
d_{E_7}^{(n)}(D_4) = 
\begin{cases}
A_5'' &  \text{ if } n=1,\\
D_6 &  \text{ if } n= 2,\\
E_7(a_3) & \text{ if } n=3,\\
E_7(a_1) & \text{ if } n=4,5,\\
E_7 & \text{ if } n \gest 6.
\end{cases} \quad
d_{E_7}^{(n)}(D_4(a_1) + A_1) = 
\begin{cases}
E_6(a_3) &  \text{ if } n=1,\\
E_7(a_3) &  \text{ if } n=2,\\
E_7(a_1) &  \text{ if } n=3,\\
E_7 & \text{ if } n \gest 4.
\end{cases}
$$
$$
d_{E_7}^{(n)}(A_3 + A_2) = 
\begin{cases}
D_5(a_1) + A_1 &  \text{ if } n=1,\\
E_7(a_3) &  \text{ if } n= 2,\\
E_7(a_1) & \text{ if } n=3,\\
E_7 & \text{ if } n\gest 4.
\end{cases} \quad
d_{E_7}^{(n)}(A_4) = 
\begin{cases}
D_5(a_1) &  \text{ if } n=1,\\
E_6(a_1) &  \text{ if } n=2,\\
E_7(a_2) &  \text{ if } n=3,\\
E_7(a_1) & \text{ if } n = 4,\\
E_7 & \text{ if } n \gest 5.
\end{cases}
$$
$$
d_{E_7}^{(n)}(A_3 + A_2 + A_1) = 
\begin{cases}
A_4 + A_2 &  \text{ if } n=1,\\
E_7(a_3) &  \text{ if } n= 2,\\
E_7(a_1) & \text{ if } n=3,\\
E_7 & \text{ if } n\gest 4.
\end{cases} \quad
d_{E_7}^{(n)}(A_5'') = 
\begin{cases}
D_4 &  \text{ if } n=1,\\
D_5 + A_1 &  \text{ if } n=2,\\
E_6 &  \text{ if } n=3,\\
E_7(a_2) & \text{ if } n = 4,\\
E_7(a_1) & \text{ if } n = 5,\\
E_7 & \text{ if } n \gest 6.
\end{cases}
$$
$$
d_{E_7}^{(n)}(D_4 + A_1) = 
\begin{cases}
A_4 &  \text{ if } n=1,\\
D_6 &  \text{ if } n= 2,\\
E_7(a_3) & \text{ if } n=3,\\
E_7(a_1) & \text{ if } n=4, 5,\\
E_7 & \text{ if } n \gest 6.
\end{cases} \quad
d_{E_7}^{(n)}(A_4+ A_1) = 
\begin{cases}
A_4 + A_1 &  \text{ if } n=1,\\
E_6(a_1) &  \text{ if } n=2,\\
E_7(a_2) & \text{ if } n = 3,\\
E_7(a_1) & \text{ if } n = 4,\\
E_7 & \text{ if } n \gest 5.
\end{cases}
$$
$$
d_{E_7}^{(n)}(D_5(a_1)) = 
\begin{cases}
A_4 &  \text{ if } n=1,\\
E_7(a_4) &  \text{ if } n= 2,\\
E_7(a_3) & \text{ if } n=3,\\
E_7(a_1) & \text{ if } n=4, 5,\\
E_7 & \text{ if } n \gest 6.
\end{cases} \quad
d_{E_7}^{(n)}(A_4+ A_2) = 
\begin{cases}
A_3 + A_2+ A_1 &  \text{ if } n=1,\\
E_7(a_4) &  \text{ if } n=2,\\
E_7(a_2) & \text{ if } n = 3,\\
E_7(a_1) & \text{ if } n = 4,\\
E_7 & \text{ if } n \gest 5.
\end{cases}
$$
$$
d_{E_7}^{(n)}(A_5') = 
\begin{cases}
D_4(a_1) + A_1 &  \text{ if } n=1,\\
D_5 + A_1 &  \text{ if } n= 2,\\
E_7(a_3) & \text{ if } n=3,\\
E_7(a_2) & \text{ if } n=4,\\
E_7(a_1) & \text{ if } n=5,\\
E_7 & \text{ if } n \gest 6.
\end{cases} \quad
d_{E_7}^{(n)}(A_5+ A_1) = 
\begin{cases}
D_4(a_1) &  \text{ if } n=1,\\
D_5 + A_1 &  \text{ if } n=2,\\
E_6 & \text{ if } n = 3,\\
E_7(a_2) & \text{ if } n = 4,\\
E_7(a_1) & \text{ if } n = 5,\\
E_7 & \text{ if } n \gest 6.
\end{cases}
$$
$$
d_{E_7}^{(n)}(D_5(a_1) + A_1) = 
\begin{cases}
A_3 + A_2 &  \text{ if } n=1,\\
E_7(a_4) &  \text{ if } n= 2,\\
E_7(a_3) & \text{ if } n=3,\\
E_7(a_1) & \text{ if } n=4,5,\\
E_7 & \text{ if } n \gest 6.
\end{cases} \quad
d_{E_7}^{(n)}(D_6(a_2)) = 
\begin{cases}
D_4(a_1) &  \text{ if } n=1,\\
D_5 + A_1 &  \text{ if } n=2,\\
E_6(a_1) & \text{ if } n = 3,\\
E_7(a_2) & \text{ if } n = 4,\\
E_7(a_1) & \text{ if } n = 5,\\
E_7 & \text{ if } n \gest 6.
\end{cases}
$$
$$
d_{E_7}^{(n)}(E_6(a_3)) = 
\begin{cases}
D_4(a_1) + A_1 &  \text{ if } n=1,\\
E_7(a_5) &  \text{ if } n= 2,\\
E_7(a_3) & \text{ if } n=3,\\
E_7(a_2) & \text{ if } n=4,\\
E_7(a_1) & \text{ if } n=5,\\
E_7 & \text{ if } n \gest 6.
\end{cases} \quad
d_{E_7}^{(n)}(D_5) = 
\begin{cases}
(A_3 + A_1)'' &  \text{ if } n=1,\\
D_6(a_2) &  \text{ if } n=2,\\
E_7(a_4) & \text{ if } n = 3,\\
E_7(a_3) & \text{ if } n = 4,\\
E_7(a_2) & \text{ if } n = 5,\\
E_7(a_1) & \text{ if } n = 6,7,\\
E_7 & \text{ if } n \gest 8.
\end{cases}
$$
$$
d_{E_7}^{(n)}(E_7(a_5)) = 
\begin{cases}
D_4(a_1) &  \text{ if } n=1,\\
E_7(a_5) &  \text{ if } n= 2,\\
E_6(a_1) & \text{ if } n=3,\\
E_7(a_2) & \text{ if } n=4,\\
E_7(a_1) & \text{ if } n=5,\\
E_7 & \text{ if } n \gest 6.
\end{cases} \quad
d_{E_7}^{(n)}(A_6) = 
\begin{cases}
A_2 + 3A_1 &  \text{ if } n=1,\\
D_5(a_1) + A_1 &  \text{ if } n=2,\\
E_7(a_4) & \text{ if } n = 3,\\
E_7(a_3) & \text{ if } n = 4,\\
E_7(a_2) & \text{ if } n = 5,\\
E_7(a_1) & \text{ if } n = 6,\\
E_7 & \text{ if } n \gest 7.
\end{cases}
$$
$$
d_{E_7}^{(n)}(D_5 + A_1) = 
\begin{cases}
2A_2 &  \text{ if } n=1,\\
D_6(a_2) &  \text{ if } n= 2,\\
E_7(a_4) & \text{ if } n=3,\\
E_7(a_3) & \text{ if } n=4,\\
E_7(a_2) & \text{ if } n=5,\\
E_7(a_1) & \text{ if } n = 6, 7,\\
E_7 & \text{ if } n \gest 8.
\end{cases} \quad
d_{E_7}^{(n)}(D_6(a_1)) = 
\begin{cases}
A_3 &  \text{ if } n=1,\\
D_5(a_1) + A_1 &  \text{ if } n=2,\\
E_7(a_4) & \text{ if } n = 3,\\
E_7(a_3) & \text{ if } n = 4,\\
E_7(a_2) & \text{ if } n = 5,\\
E_7(a_1) & \text{ if } n = 6,7,\\
E_7 & \text{ if } n \gest 8.
\end{cases}
$$
$$
d_{E_7}^{(n)}(E_7(a_4)) = 
\begin{cases}
A_2 + 2A_1 &  \text{ if } n=1,\\
D_5(a_1) + A_1 &  \text{ if } n= 2,\\
E_7(a_4) & \text{ if } n=3,\\
E_7(a_3) & \text{ if } n=4,\\
E_7(a_2) & \text{ if } n=5,\\
E_7(a_1) & \text{ if } n = 6, 7,\\
E_7 & \text{ if } n \gest 8.
\end{cases} \quad
d_{E_7}^{(n)}(D_6) = 
\begin{cases}
A_2 &  \text{ if } n=1,\\
D_4 + A_1 &  \text{ if } n=2,\\
E_6(a_3) & \text{ if } n = 3,\\
E_7(a_4) & \text{ if } n = 4,\\
E_6(a_1) & \text{ if } n = 5,\\
E_7(a_2) & \text{ if } n = 6,7,\\
E_7(a_1) & \text{ if } n = 8, 9, \\
E_7 & \text{ if } n \gest 10.
\end{cases}
$$
$$
d_{E_7}^{(n)}(E_6(a_1)) = 
\begin{cases}
A_2 + A_1 &  \text{ if } n=1,\\
A_4 + A_1 &  \text{ if } n= 2,\\
E_7(a_5) &  \text{ if } n= 3,\\
E_6(a_1) & \text{ if } n=4,\\
E_7(a_3) &  \text{ if } n=5,\\
E_7(a_2) & \text{ if } n=6,\\
E_7(a_1) & \text{ if } n=7,8,\\
E_7 & \text{ if } n \gest 9.
\end{cases} \quad
d_{E_7}^{(n)}(E_6) = 
\begin{cases}
(3A_1)'' &  \text{ if } n=1,\\
A_3 + 2A_1 &  \text{ if } n=2,\\
A_5 + A_1 & \text{ if } n = 3,\\
E_7(a_5) & \text{ if } n = 4,\\
E_7(a_4) & \text{ if } n = 5,\\
E_7(a_3) & \text{ if } n = 6,7,\\
E_7(a_2) & \text{ if } n = 8, \\
E_7(a_1) & \text{ if } n = 9,10, 11,\\
E_7 & \text{ if } n \gest 12.
\end{cases}
$$
$$
d_{E_7}^{(n)}(E_7(a_3)) = 
\begin{cases}
A_2 &  \text{ if } n=1,\\
A_3 + A_2 + A_1 &  \text{ if } n= 2,\\
E_6(a_3) & \text{ if } n=3,\\
E_7(a_4) & \text{ if } n=4,\\
E_6(a_1) & \text{ if } n=5,\\
E_7(a_2) & \text{ if } n=6,7,\\
E_7(a_1) & \text{ if } n=8,9,\\
E_7 & \text{ if } n \gest 10.
\end{cases} \quad
d_{E_7}^{(n)}(E_7(a_2)) = 
\begin{cases}
2A_1 &  \text{ if } n=1,\\
A_3 + 2A_1 &  \text{ if } n=2,\\
A_4 + A_2 & \text{ if } n = 3,\\
E_7(a_5) & \text{ if } n = 4,\\
E_7(a_4) & \text{ if } n = 5,\\
E_7(a_3) & \text{ if } n = 6,7,\\
E_7(a_2) & \text{ if } n = 8, \\
E_7(a_1) & \text{ if } n = 9,10, 11,\\
E_7 & \text{ if } n \gest 12.
\end{cases}
$$
$$
d_{E_7}^{(n)}(E_7(a_1)) = 
\begin{cases}
A_1 &  \text{ if } n=1,\\
A_2 + 3A_1 &  \text{ if } n= 2,\\
A_3 + A_2 + A_1 &  \text{ if } n= 3,\\
D_5(a_1) + A_1 & \text{ if } n=4,\\
E_7(a_5) & \text{ if } n=5,\\
E_7(a_4) & \text{ if } n=6,\\
E_6(a_1) & \text{ if } n=7,\\
E_7(a_3) & \text{ if } n=8,\\
E_7(a_2) & \text{ if } n=9, 10,\\
E_7(a_1) & \text{ if } n=11, 12, 13,\\
E_7 & \text{ if } n \gest 14.
\end{cases} \quad
d_{E_7}^{(n)}(E_7) = 
\begin{cases}
0 &  \text{ if } n=1,\\
4A_1 &  \text{ if } n=2,\\
2A_2 + A_1 & \text{ if } n = 3,\\
A_3 +A_2 +A_1 & \text{ if } n = 4,\\
A_4 + A_2 & \text{ if } n = 5,\\
E_7(a_5) & \text{ if } n = 6,\\
A_6 & \text{ if } n = 7, \\
E_7(a_4) & \text{ if } n = 8,\\
E_6(a_1) & \text{ if } n = 9,\\
E_7(a_3) & \text{ if } n =10, 11,\\
E_7(a_2) & \text{ if } n = 12, 13,\\
E_7(a_1) & \text{ if } n \in [14, 17],\\
E_7 & \text{ if } n \gest 18.
\end{cases}
$$
\subsection{\texorpdfstring{Type $E_8$}{}}
The covering Barbasch--Vogan map for $\wt{E}_8^{(n)}$ is given as follows.

$$
d_{E_8}^{(n)}(A_1) = 
\begin{cases}
E_8(a_1) &  \text{ if } n=1,\\
E_8 &  \text{ if } n\gest 2.
\end{cases} \quad
d_{E_8}^{(n)}(2A_1) = 
\begin{cases}
E_8(a_2) &  \text{ if } n=1,\\
E_8 &  \text{ if } n\gest 2.
\end{cases}
$$
$$
d_{E_8}^{(n)}(3A_1) = 
\begin{cases}
E_8(a_3) &  \text{ if } n=1,\\
E_8 &  \text{ if } n\gest 2.
\end{cases} \quad
d_{E_8}^{(n)}(A_2) = 
\begin{cases}
E_8(a_3) &  \text{ if } n=1,\\
E_8(a_1) &  \text{ if } n=2,\\
E_8 &  \text{ if } n\gest 3.
\end{cases}
$$
$$
d_{E_8}^{(n)}(4A_1) = 
\begin{cases}
E_8(a_4) &  \text{ if } n=1,\\
E_8 & \text{ if } n \gest 2.
\end{cases} \quad
d_{E_8}^{(n)}(A_2 + A_1) = 
\begin{cases}
E_8(a_4) &  \text{ if } n=1,\\
E_8(a_1) &  \text{ if } n=2,\\
E_8 &  \text{ if } n\gest 3.
\end{cases}
$$
$$
d_{E_8}^{(n)}(A_2 + 2A_1) = 
\begin{cases}
E_8(b_4) &  \text{ if } n=1,\\
E_8(a_1) &  \text{ if } n= 2,\\
E_8 & \text{ if } n\gest 3.
\end{cases} \quad
d_{E_8}^{(n)}(A_3) = 
\begin{cases}
E_7(a_1) &  \text{ if } n=1,\\
E_8(a_2) &  \text{ if } n=2,\\
E_8(a_1) &  \text{ if } n=3,\\
E_8 & \text{ if } n \gest 4.
\end{cases}
$$
$$
d_{E_8}^{(n)}(A_2+ 3A_1) = 
\begin{cases}
E_8(a_5) &  \text{ if } n=1,\\
E_8(a_1) &  \text{ if } n= 2,\\
E_8 & \text{ if } n \gest 3.
\end{cases} \quad
d_{E_8}^{(n)}(2A_2) = 
\begin{cases}
E_8(a_5) &  \text{ if } n=1,\\
E_8(a_2) &  \text{ if } n=2,\\
E_8 & \text{ if } n \gest 3.
\end{cases}
$$
$$
d_{E_8}^{(n)}(2A_2 + A_1) = 
\begin{cases}
E_8(b_5) &  \text{ if } n=1,\\
E_8(a_2) &  \text{ if } n= 2,\\
E_8 & \text{ if } n\gest 3.
\end{cases} \quad
d_{E_8}^{(n)}(A_3 + A_1) = 
\begin{cases}
E_8(b_5) &  \text{ if } n=1,\\
E_8(a_2) &  \text{ if } n=2,\\
E_8(a_1) &  \text{ if } n=3,\\
E_8 & \text{ if } n \gest 4.
\end{cases}
$$
$$
d_{E_8}^{(n)}(D_4(a_1)) = 
\begin{cases}
E_8(b_5) &  \text{ if } n=1,\\
E_8(a_3) &  \text{ if } n= 2,\\
E_8(a_1) &  \text{ if } n= 3,\\
E_8 & \text{ if } n\gest 4.
\end{cases} \quad
d_{E_8}^{(n)}(D_4) = 
\begin{cases}
E_6 &  \text{ if } n=1,\\
E_7 &  \text{ if } n=2,\\
E_8(a_3) &  \text{ if } n=3,\\
E_8(a_1) &  \text{ if } n=4, 5,\\
E_8 & \text{ if } n \gest 6.
\end{cases}
$$
$$
d_{E_8}^{(n)}(2A_2 + 2A_1) = 
\begin{cases}
E_8(a_6) &  \text{ if } n=1,\\
E_8(a_2) &  \text{ if } n= 2,\\
E_8 & \text{ if } n\gest 3.
\end{cases} \quad
d_{E_8}^{(n)}(A_3 + 2A_1) = 
\begin{cases}
E_8(a_6) &  \text{ if } n=1,\\
E_8(a_2) &  \text{ if } n=2,\\
E_8(a_1) &  \text{ if } n=3,\\
E_8 & \text{ if } n \gest 4.
\end{cases}
$$
$$
d_{E_8}^{(n)}(D_4(a_1) + A_1) = 
\begin{cases}
E_8(a_6) &  \text{ if } n=1,\\
E_8(a_3) &  \text{ if } n= 2,\\
E_8(a_1) & \text{ if } n=3,\\
E_8 & \text{ if } n \gest 4.
\end{cases} \quad
d_{E_8}^{(n)}(A_3 + A_2) = 
\begin{cases}
D_7(a_1) &  \text{ if } n=1,\\
E_8(a_3) &  \text{ if } n=2,\\
E_8(a_1) &  \text{ if } n=3,\\
E_8 & \text{ if } n \gest 4.
\end{cases}
$$
$$
d_{E_8}^{(n)}(A_4) = 
\begin{cases}
E_7(a_3) &  \text{ if } n=1,\\
E_8(a_4) &  \text{ if } n= 2,\\
E_8(a_2) & \text{ if } n=3,\\
E_8(a_1) & \text{ if } n=4,\\
E_8 & \text{ if } n\gest 5.
\end{cases} \quad
d_{E_8}^{(n)}(A_3 + A_2 + A_1) = 
\begin{cases}
E_8(b_6) &  \text{ if } n=1,\\
E_8(a_3) &  \text{ if } n=2,\\
E_8(a_1) &  \text{ if } n=3,\\
E_8 & \text{ if } n \gest 4.
\end{cases}
$$
$$
d_{E_8}^{(n)}(D_4 + A_1) = 
\begin{cases}
E_6(a_1) &  \text{ if } n=1,\\
E_7 &  \text{ if } n= 2,\\
E_8(a_3) & \text{ if } n=3,\\
E_8(a_1) & \text{ if } n=4,5,\\
E_8 & \text{ if } n\gest 6.
\end{cases} \quad
d_{E_8}^{(n)}(D_4(a_1) + A_2) = 
\begin{cases}
E_8(b_6) &  \text{ if } n=1,\\
E_8(a_4) &  \text{ if } n=2,\\
E_8(a_1) &  \text{ if } n=3,\\
E_8 & \text{ if } n \gest 4.
\end{cases}
$$
$$
d_{E_8}^{(n)}(A_4 + A_1) = 
\begin{cases}
E_6(a_1) + A_1 &  \text{ if } n=1,\\
E_8(a_4) &  \text{ if } n= 2,\\
E_8(a_2) & \text{ if } n=3,\\
E_8(a_1) & \text{ if } n=4,\\
E_8 & \text{ if } n \gest 5.
\end{cases} \quad
d_{E_8}^{(n)}(2A_3) = 
\begin{cases}
D_7(a_2) &  \text{ if } n=1,\\
E_8(a_4) &  \text{ if } n=2,\\
E_8(a_2) & \text{ if } n = 3,\\
E_8 & \text{ if } n \gest 4.
\end{cases}
$$
$$
d_{E_8}^{(n)}(D_5(a_1)) = 
\begin{cases}
E_6(a_1) &  \text{ if } n=1,\\
E_8(b_4) &  \text{ if } n= 2,\\
E_8(a_3) & \text{ if } n=3,\\
E_8(a_1) & \text{ if } n=4, 5,\\
E_8 & \text{ if } n \gest 6.
\end{cases} \quad
d_{E_8}^{(n)}(A_4+ 2A_1) = 
\begin{cases}
D_7(a_2) &  \text{ if } n=1,\\
E_8(a_4) &  \text{ if } n=2,\\
E_8(a_2) & \text{ if } n = 3,\\
E_8(a_1) & \text{ if } n = 4,\\
E_8 & \text{ if } n \gest 5.
\end{cases}
$$
$$
d_{E_8}^{(n)}(A_4 + A_2) = 
\begin{cases}
D_5 + A_2 &  \text{ if } n=1,\\
E_8(b_4) &  \text{ if } n= 2,\\
E_8(a_2) & \text{ if } n=3,\\
E_8(a_1) & \text{ if } n=4,\\
E_8 & \text{ if } n \gest 5.
\end{cases} \quad
d_{E_8}^{(n)}(A_5) = 
\begin{cases}
D_6(a_1) &  \text{ if } n=1,\\
E_8(a_5) &  \text{ if } n=2,\\
E_8(a_3) & \text{ if } n = 3,\\
E_8(a_2) & \text{ if } n = 4,\\
E_8(a_1) & \text{ if } n = 5,\\
E_8 & \text{ if } n \gest 6.
\end{cases}
$$
$$
d_{E_8}^{(n)}(D_5(a_1) + A_1) = 
\begin{cases}
E_7(a_4) &  \text{ if } n=1,\\
E_8(b_4) &  \text{ if } n= 2,\\
E_8(a_3) & \text{ if } n=3,\\
E_8(a_1) & \text{ if } n=4,5,\\
E_8 & \text{ if } n \gest 6.
\end{cases} \quad
d_{E_8}^{(n)}(A_4 + A_2 + A_1) = 
\begin{cases}
A_6 + A_1 &  \text{ if } n=1,\\
E_8(b_4) &  \text{ if } n=2,\\
E_8(a_2) & \text{ if } n = 3,\\
E_8(a_1) & \text{ if } n = 4,\\
E_8 & \text{ if } n \gest 5.
\end{cases}
$$
$$
d_{E_8}^{(n)}(D_4 + A_2) = 
\begin{cases}
A_6 &  \text{ if } n=1,\\
E_8(b_4) &  \text{ if } n= 2,\\
E_8(a_3) & \text{ if } n=3,\\
E_8(a_1) & \text{ if } n=4, 5\\
E_8 & \text{ if } n \gest 6.
\end{cases} \quad
d_{E_8}^{(n)}(E_6(a_3)) = 
\begin{cases}
D_6(a_1) &  \text{ if } n=1,\\
E_8(b_5) &  \text{ if } n=2,\\
E_8(a_3) & \text{ if } n = 3,\\
E_8(a_2) & \text{ if } n = 4,\\
E_8(a_1) & \text{ if } n = 5,\\
E_8 & \text{ if } n \gest 6.
\end{cases}
$$
$$
d_{E_8}^{(n)}(D_5) = 
\begin{cases}
D_5 &  \text{ if } n=1,\\
E_7(a_2) &  \text{ if } n= 2,\\
E_8(b_4) & \text{ if } n=3,\\
E_8(a_3) & \text{ if } n=4,\\
E_8(a_2) & \text{ if } n=5,\\
E_8(a_1) & \text{ if } n=6, 7,\\
E_8 & \text{ if } n \gest 8.
\end{cases} \quad
d_{E_8}^{(n)}(A_4 + A_3) = 
\begin{cases}
E_8(a_7) &  \text{ if } n=1,\\
E_8(a_5) &  \text{ if } n=2,\\
E_8(a_3) & \text{ if } n = 3,\\
E_8(a_1) & \text{ if } n = 4,\\
E_8 & \text{ if } n \gest 5.
\end{cases}
$$
$$
d_{E_8}^{(n)}(A_5  + A_1) = 
\begin{cases}
E_8(a_7) &  \text{ if } n=1,\\
E_8(a_5) &  \text{ if } n= 2,\\
E_8(a_3) & \text{ if } n=3,\\
E_8(a_2) & \text{ if } n=4,\\
E_8(a_1) & \text{ if } n=5,\\
E_8 & \text{ if } n \gest 6.
\end{cases} \quad
d_{E_8}^{(n)}(D_5(a_1) + A_2) = 
\begin{cases}
E_8(a_7) &  \text{ if } n=1,\\
E_8(a_5) &  \text{ if } n=2,\\
E_8(a_3) & \text{ if } n = 3,\\
E_8(a_1) & \text{ if } n = 4, 5,\\
E_8 & \text{ if } n \gest 6.
\end{cases}
$$
$$
d_{E_8}^{(n)}(D_6(a_2)) = 
\begin{cases}
E_8(a_7) &  \text{ if } n=1,\\
E_8(a_5) &  \text{ if } n= 2,\\
E_8(a_4) & \text{ if } n=3,\\
E_8(a_2) & \text{ if } n=4,\\
E_8(a_1) & \text{ if } n=5,\\
E_8 & \text{ if } n \gest 6.
\end{cases} \quad
d_{E_8}^{(n)}(E_6(a_3) + A_1) = 
\begin{cases}
E_8(a_7) &  \text{ if } n=1,\\
E_8(b_5) &  \text{ if } n=2,\\
E_8(a_3) & \text{ if } n = 3,\\
E_8(a_2) & \text{ if } n = 4,\\
E_8(a_1) & \text{ if } n = 5,\\
E_8 & \text{ if } n \gest 6.
\end{cases}
$$
$$
d_{E_8}^{(n)}(E_7(a_5)) = 
\begin{cases}
E_8(a_7) &  \text{ if } n=1,\\
E_8(b_5) &  \text{ if } n=2,\\
E_8(a_4) & \text{ if } n=3,\\
E_8(a_2) & \text{ if } n=4,\\
E_8(a_1) & \text{ if } n=5,\\
E_8 & \text{ if } n \gest 6.
\end{cases} \quad
d_{E_8}^{(n)}(D_5 + A_1) = 
\begin{cases}
E_6(a_3) &  \text{ if } n=1,\\
E_7(a_2) &  \text{ if } n=2,\\
E_8(b_4) & \text{ if } n = 3,\\
E_8(a_3) & \text{ if } n = 4,\\
E_8(a_2) & \text{ if } n = 5,\\
E_8(a_1) & \text{ if } n = 6,7,\\
E_8 & \text{ if } n \gest 8.
\end{cases}
$$
$$
d_{E_8}^{(n)}(E_8(a_7)) = 
\begin{cases}
E_8(a_7) &  \text{ if } n=1,\\
E_8(a_6) &  \text{ if } n= 2,\\
E_8(a_4) & \text{ if } n=3,\\
E_8(a_2) & \text{ if } n=4,\\
E_8(a_1) & \text{ if } n=5,\\
E_8 & \text{ if } n \gest 6.
\end{cases} \quad
d_{E_8}^{(n)}(A_6) = 
\begin{cases}
D_4 + A_2 &  \text{ if } n=1,\\
D_7(a_1) &  \text{ if } n=2,\\
E_8(b_4) & \text{ if } n = 3,\\
E_8(a_3) & \text{ if } n = 4,\\
E_8(a_2) & \text{ if } n = 5,\\
E_8(a_1) & \text{ if } n = 6,\\
E_8 & \text{ if } n \gest 7.
\end{cases}
$$
$$
d_{E_8}^{(n)}(D_6(a_1)) = 
\begin{cases}
E_6(a_3) &  \text{ if } n=1,\\
D_7(a_1) &  \text{ if } n= 2,\\
E_8(b_4) &  \text{ if } n= 3,\\
E_8(a_3) & \text{ if } n=4,\\
E_8(a_2) & \text{ if } n=5,\\
E_8(a_1) & \text{ if } n=6, 7,\\
E_8 & \text{ if } n \gest 8.
\end{cases} \quad
d_{E_8}^{(n)}(A_6 + A_1) = 
\begin{cases}
A_4 + A_2 + A_1 &  \text{ if } n=1,\\
D_7(a_1) &  \text{ if } n=2,\\
E_8(b_4) & \text{ if } n = 3,\\
E_8(a_3) & \text{ if } n = 4,\\
E_8(a_2) & \text{ if } n = 5,\\
E_8(a_1) & \text{ if } n = 6,\\
E_8 & \text{ if } n \gest 7.
\end{cases}
$$
$$
d_{E_8}^{(n)}(E_7(a_4)) = 
\begin{cases}
D_5(a_1) + A_1 &  \text{ if } n=1,\\
D_7(a_1) &  \text{ if } n= 2,\\
E_8(b_4) &  \text{ if } n= 3,\\
E_8(a_3) & \text{ if } n=4,\\
E_8(a_2) & \text{ if } n=5,\\
E_8(a_1) & \text{ if } n=6,7,\\
E_8 & \text{ if } n \gest 8.
\end{cases} \quad
d_{E_8}^{(n)}(E_6(a_1)) = 
\begin{cases}
D_5(a_1) &  \text{ if } n=1,\\
E_6(a_1) + A_1 &  \text{ if } n=2,\\
E_8(b_5) & \text{ if } n = 3,\\
E_8(a_4) & \text{ if } n = 4,\\
E_8(a_3) & \text{ if } n = 5,\\
E_8(a_2) & \text{ if } n = 6,\\
E_8(a_1) & \text{ if } n = 7, 8,\\
E_8 & \text{ if } n \gest 9.
\end{cases}
$$
$$
d_{E_8}^{(n)}(D_5 + A_2) = 
\begin{cases}
A_4 + A_2 &  \text{ if } n=1,\\
D_7(a_1) &  \text{ if } n= 2,\\
E_8(b_4) &  \text{ if } n= 3,\\
E_8(a_3) & \text{ if } n=4,\\
E_8(a_2) & \text{ if } n=5,\\
E_8(a_1) & \text{ if } n=6, 7,\\
E_8 & \text{ if } n \gest 8.
\end{cases} \quad
d_{E_8}^{(n)}(D_6) = 
\begin{cases}
A_4 &  \text{ if } n=1,\\
D_6 &  \text{ if } n=2,\\
E_8(a_6) & \text{ if } n = 3,\\
E_8(b_4) & \text{ if } n = 4,\\
E_8(a_4) & \text{ if } n = 5,\\
E_8(a_2) & \text{ if } n = 6, 7,\\
E_8(a_1) & \text{ if } n = 8, 9,\\
E_8 & \text{ if } n \gest 10.
\end{cases}
$$
$$
d_{E_8}^{(n)}(E_6) = 
\begin{cases}
D_4 &  \text{ if } n=1,\\
D_5 + A_1 &  \text{ if } n= 2,\\
E_6 + A_1 &  \text{ if } n= 3,\\
E_8(b_5) & \text{ if } n=4,\\
E_8(b_4) & \text{ if } n=5,\\
E_8(a_3) & \text{ if } n=6,7,\\
E_8(a_2) & \text{ if } n=8,\\
E_8(a_1) & \text{ if } n=9, 10, 11,\\
E_8 & \text{ if } n \gest 12.
\end{cases} \quad
d_{E_8}^{(n)}(D_7(a_2)) = 
\begin{cases}
A_4 + 2A_1 &  \text{ if } n=1,\\
E_8(b_6) &  \text{ if } n=2,\\
E_8(a_5) & \text{ if } n = 3,\\
E_8(a_4) & \text{ if } n = 4,\\
E_8(a_2) & \text{ if } n = 5,\\
E_8(a_1) & \text{ if } n = 6, 7,\\
E_8 & \text{ if } n \gest 8.
\end{cases}
$$
$$
d_{E_8}^{(n)}(A_7) = 
\begin{cases}
D_4(a_1) + A_2 &  \text{ if } n=1,\\
D_7(a_2) &  \text{ if } n= 2,\\
E_8(b_5) &  \text{ if } n= 3,\\
E_8(a_4) & \text{ if } n=4,\\
E_8(a_3) & \text{ if } n=5,\\
E_8(a_2) & \text{ if } n=6,\\
E_8(a_1) & \text{ if } n=7,\\
E_8 & \text{ if } n \gest 8.
\end{cases} \quad
d_{E_8}^{(n)}(E_6(a_1) + A_1) = 
\begin{cases}
A_4 + A_1 &  \text{ if } n=1,\\
E_6(a_1) + A_1 &  \text{ if } n=2,\\
E_8(b_5) & \text{ if } n = 3,\\
E_8(a_4) & \text{ if } n = 4,\\
E_8(a_3) & \text{ if } n = 5,\\
E_8(a_2) & \text{ if } n = 6,\\
E_8(a_1) & \text{ if } n=7,8,\\
E_8 & \text{ if } n \gest 9.
\end{cases}
$$
$$
d_{E_8}^{(n)}(E_7(a_3)) = 
\begin{cases}
A_4 &  \text{ if } n=1,\\
D_5 + A_2 &  \text{ if } n= 2,\\
E_8(a_6) &  \text{ if } n= 3,\\
E_8(b_4) & \text{ if } n=4,\\
E_8(a_4) & \text{ if } n=5,\\
E_8(a_2) & \text{ if } n=6, 7,\\
E_8(a_1) & \text{ if } n=8, 9,\\
E_8 & \text{ if } n \gest 10.
\end{cases} \quad
d_{E_8}^{(n)}(E_8(b_6)) = 
\begin{cases}
D_4(a_1) + A_2 &  \text{ if } n=1,\\
D_7(a_2) &  \text{ if } n=2,\\
E_8(b_5) & \text{ if } n = 3,\\
E_8(a_4) & \text{ if } n = 4,\\
E_8(a_3) & \text{ if } n = 5,\\
E_8(a_2) & \text{ if } n = 6,\\
E_8(a_1) & \text{ if } n=7,8,\\
E_8 & \text{ if } n \gest 9.
\end{cases}
$$
$$
d_{E_8}^{(n)}(D_7(a_1)) = 
\begin{cases}
A_3 + A_2 &  \text{ if } n=1,\\
D_5 + A_2 &  \text{ if } n= 2,\\
E_8(a_6) &  \text{ if } n= 3,\\
E_8(b_4) & \text{ if } n=4,\\
E_8(a_4) & \text{ if } n=5,\\
E_8(a_2) & \text{ if } n=6,7,\\
E_8(a_1) & \text{ if } n=8,9,\\
E_8 & \text{ if } n \gest 10.
\end{cases} \quad
d_{E_8}^{(n)}(E_6 + A_1) = 
\begin{cases}
D_4(a_1) &  \text{ if } n=1,\\
D_5 + A_1 &  \text{ if } n=2,\\
E_6 + A_1 & \text{ if } n = 3,\\
E_8(b_5) & \text{ if } n = 4,\\
E_8(b_4) & \text{ if } n = 5,\\
E_8(a_3) & \text{ if } n = 6, 7,\\
E_8(a_2) & \text{ if } n=8,\\
E_8(a_1) & \text{ if } n=9, 10, 11,\\
E_8 & \text{ if } n \gest 12.
\end{cases}
$$
$$
d_{E_8}^{(n)}(E_7(a_2)) = 
\begin{cases}
D_4(a_1) &  \text{ if } n=1,\\
D_5+A_1 &  \text{ if } n= 2,\\
E_8(b_6) &  \text{ if } n= 3,\\
E_8(b_5) & \text{ if } n=4,\\
E_8(b_4) & \text{ if } n=5,\\
E_8(a_3) & \text{ if } n=6,7,\\
E_8(a_2) & \text{ if } n=8,\\
E_8(a_1) & \text{ if } n=9,10,11,\\
E_8 & \text{ if } n \gest 12.
\end{cases} \quad
d_{E_8}^{(n)}(E_8(a_6)) = 
\begin{cases}
D_4(a_1) + A_1 &  \text{ if } n=1,\\
E_8(a_7) &  \text{ if } n=2,\\
E_8(a_6) & \text{ if } n = 3,\\
E_8(a_5) & \text{ if } n = 4,\\
E_8(a_4) & \text{ if } n = 5,\\
E_8(a_3) & \text{ if } n = 6,\\
E_8(a_2) & \text{ if } n=7,\\
E_8(a_1) & \text{ if } n=8,9,\\
E_8 & \text{ if } n \gest 10.
\end{cases}
$$
$$
d_{E_8}^{(n)}(D_7) = 
\begin{cases}
2A_2 &  \text{ if } n=1,\\
D_6(a_2) &  \text{ if } n= 2,\\
D_7(a_2) &  \text{ if } n= 3,\\
D_7 & \text{ if } n=4,\\
E_8(a_5) & \text{ if } n=5,\\
E_8(a_4) & \text{ if } n=6,\\
E_8(a_3) & \text{ if } n=7,\\
E_8(a_2) & \text{ if } n=8,9,\\
E_8(a_1) & \text{ if } n=10,11,\\
E_8 & \text{ if } n \gest 12.
\end{cases} \quad
d_{E_8}^{(n)}(E_8(b_5)) = 
\begin{cases}
D_4(a_1) &  \text{ if } n=1,\\
E_7(a_5) &  \text{ if } n=2,\\
E_8(b_6) & \text{ if } n = 3,\\
E_8(b_5) & \text{ if } n = 4,\\
E_8(b_4) & \text{ if } n = 5,\\
E_8(a_3) & \text{ if } n = 6, 7,\\
E_8(a_2) & \text{ if } n=8,\\
E_8(a_1) & \text{ if } n=9,10, 11\\
E_8 & \text{ if } n \gest 12.
\end{cases}
$$
$$
d_{E_8}^{(n)}(E_7(a_1)) = 
\begin{cases}
A_3 &  \text{ if } n=1,\\
D_4 + A_2 &  \text{ if } n= 2,\\
D_5 + A_2 &  \text{ if } n= 3,\\
D_7(a_1) & \text{ if } n=4,\\
E_8(b_5) & \text{ if } n=5,\\
E_8(b_4) & \text{ if } n=6,\\
E_8(a_4) & \text{ if } n=7,\\
E_8(a_3) & \text{ if } n=8,\\
E_8(a_2) & \text{ if } n=9,10,\\
E_8(a_1) & \text{ if } n=11,12, 13,\\
E_8 & \text{ if } n \gest 14.
\end{cases} \quad
d_{E_8}^{(n)}(E_8(a_5)) = 
\begin{cases}
2A_2 &  \text{ if } n=1,\\
D_6(a_2) &  \text{ if } n=2,\\
D_7(a_2) & \text{ if } n = 3,\\
E_8(a_6) & \text{ if } n = 4,\\
E_8(a_5) & \text{ if } n = 5,\\
E_8(a_4) & \text{ if } n = 6,\\
E_8(a_3) & \text{ if } n=7,\\
E_8(a_2) & \text{ if } n=8,9,\\
E_8(a_1) & \text{ if } n=10, 11,\\
E_8 & \text{ if } n \gest 12.
\end{cases}
$$
$$
d_{E_8}^{(n)}(E_8(b_4)) = 
\begin{cases}
A_2 + 2A_1 &  \text{ if } n=1,\\
D_4 + A_2 &  \text{ if } n= 2,\\
D_5 + A_2 &  \text{ if } n= 3,\\
D_7(a_1) & \text{ if } n=4,\\
E_8(b_5) & \text{ if } n=5,\\
E_8(b_4) & \text{ if } n=6,\\
E_8(a_4) & \text{ if } n=7,\\
E_8(a_3) & \text{ if } n=8,\\
E_8(a_2) & \text{ if } n=9,10,\\
E_8(a_1) & \text{ if } n=11,12, 13,\\
E_8 & \text{ if } n \gest 14.
\end{cases} \quad
d_{E_8}^{(n)}(E_7) = 
\begin{cases}
A_2 &  \text{ if } n=1,\\
D_4 + A_1 &  \text{ if } n=2,\\
E_6(a_3) + A_1 & \text{ if } n = 3,\\
D_5 + A_2 & \text{ if } n = 4,\\
E_8(b_6) & \text{ if } n = 5,\\
E_8(b_5) & \text{ if } n = 6,\\
E_8(a_5) & \text{ if } n=7,\\
E_8(b_4) & \text{ if } n=8,\\
E_8(a_4) & \text{ if } n=9,\\
E_8(a_3) & \text{ if } n=10,11,\\
E_8(a_2) & \text{ if } n=12,13,\\
E_8(a_1) & \text{ if } n \in [14,17],\\
E_8 & \text{ if } n \gest 18.
\end{cases}
$$
$$
d_{E_8}^{(n)}(E_8(a_4)) = 
\begin{cases}
A_2 + A_1 &  \text{ if } n=1,\\
A_4 + 2A_1 &  \text{ if } n= 2,\\
E_8(a_7) &  \text{ if } n= 3,\\
E_8(b_6) & \text{ if } n=4,\\
E_8(a_6) & \text{ if } n=5,\\
E_8(a_5) & \text{ if } n=6,\\
E_8(a_4) & \text{ if } n=7,8,\\
E_8(a_3) & \text{ if } n=9,\\
E_8(a_2) & \text{ if } n=10, 11,\\
E_8(a_1) & \text{ if } n \in [12, 14],\\
E_8 & \text{ if } n \gest 15.
\end{cases} \quad
d_{E_8}^{(n)}(E_8(a_3)) = 
\begin{cases}
A_2 &  \text{ if } n=1,\\
A_3 + A_2 + A_1 &  \text{ if } n=2,\\
E_6(a_3) + A_1 & \text{ if } n = 3,\\
D_5 + A_2 & \text{ if } n = 4,\\
E_8(b_6) & \text{ if } n = 5,\\
E_8(b_5) & \text{ if } n = 6,\\
E_8(a_5) & \text{ if } n=7,\\
E_8(b_4) & \text{ if } n=8,\\
E_8(a_4) & \text{ if } n=9,\\
E_8(a_3) & \text{ if } n=10,11,\\
E_8(a_2) & \text{ if } n=12,13,\\
E_8(a_1) & \text{ if } n \in [14,17],\\
E_8 & \text{ if } n \gest 18.
\end{cases}
$$
$$
d_{E_8}^{(n)}(E_8(a_2)) = 
\begin{cases}
2A_1 &  \text{ if } n=1,\\
A_3 + 2A_1 &  \text{ if } n= 2,\\
A_4 + A_2 + A_1 &  \text{ if } n= 3,\\
E_8(a_7) & \text{ if } n=4,\\
D_7(a_2) & \text{ if } n=5,\\
D_7(a_1) & \text{ if } n=6,\\
E_8(a_6) & \text{ if } n=7,\\
E_8(a_5) & \text{ if } n=8,\\
E_8(b_4) & \text{ if } n=9,\\
E_8(a_4) & \text{ if } n=10, 11,\\
E_8(a_3) & \text{ if } n=12,\\
E_8(a_2) & \text{ if } n=13,14,15,\\
E_8(a_1) & \text{ if } n\in [16,19],\\
E_8 & \text{ if } n \gest 20.
\end{cases} \quad
d_{E_8}^{(n)}(E_8(a_1)) = 
\begin{cases}
A_1 &  \text{ if } n=1,\\
A_2 + 3A_1 &  \text{ if } n=2,\\
D_4(a_1) + A_2 & \text{ if } n = 3,\\
D_5(a_1) + A_2 & \text{ if } n = 4,\\
E_8(a_7) & \text{ if } n = 5,\\
D_7(a_2) & \text{ if } n = 6,\\
E_8(b_6) & \text{ if } n=7,\\
E_8(a_6) & \text{ if } n=8,\\
E_8(b_5) & \text{ if } n=9,\\
E_8(a_5) & \text{ if } n=10,\\
E_8(b_4) & \text{ if } n=11,\\
E_8(a_4) & \text{ if } n=12,13,\\
E_8(a_3) & \text{ if } n=14,15,\\
E_8(a_2) & \text{ if } n\in [16,18],\\
E_8(a_1) & \text{ if } n \in [19,23],\\
E_8 & \text{ if } n \gest 24.
\end{cases}
$$
$$
d_{E_8}^{(n)}(E_8) = 
\begin{cases}
0 &  \text{ if } n=1,\\
4A_1 &  \text{ if } n=2,\\
2A_2 + 2A_1 & \text{ if } n = 3,\\
2A_3 & \text{ if } n = 4,\\
A_4 + A_3 & \text{ if } n = 5,\\
E_8(a_7) & \text{ if } n = 6,\\
A_6 + A_1 & \text{ if } n=7,\\
A_7 & \text{ if } n=8,\\
E_8(b_6) & \text{ if } n=9,\\
E_8(a_6) & \text{ if } n=10, 11,\\
E_8(a_5) & \text{ if } n=12, 13,\\
E_8(b_4) & \text{ if } n=14,\\
E_8(a_4) & \text{ if } n \in [15,17],\\
E_8(a_3) & \text{ if } n=18, 19,\\
E_8(a_2) & \text{ if } n \in [20,23],\\
E_8(a_1) & \text{ if } n \in [24,29],\\
E_8 & \text{ if } n \gest 30.
\end{cases}
$$


\end{document}